\documentclass[12pt]{article}

\usepackage{times}

\usepackage{enumitem}
\usepackage[utf8]{inputenc}
\usepackage{commath}
\usepackage{amsthm, amscd}
\usepackage{amsmath}
\usepackage{amssymb}
\usepackage{cite}
\usepackage{setspace}

\newcommand*\tasklabelformat[1]{#1)}

\usepackage{tasks}[newest]
\settasks{
  label = \alph* ,
  label-format = \tasklabelformat ,
  label-width  = 12pt
}
\usepackage{bigints}
\usepackage{comment}
\usepackage{mathtools}
\usepackage{mathrsfs}
\usepackage{fancyhdr}

\allowdisplaybreaks[4]

\numberwithin{equation}{section}
\usepackage{etoolbox}
\patchcmd{\thebibliography}
  {\settowidth}
  {\setlength{\itemsep}{0pt plus -10pt}\settowidth}
  {}{}
\apptocmd{\thebibliography}
  {
  }
  {}{}
\makeatletter
\newtheorem*{rep@theorem}{\rep@title}
\newcommand{\newreptheorem}[2]{%
\newenvironment{rep#1}[1]{%
 \def\rep@title{#2 \ref{##1}}%
 \begin{rep@theorem}}%
 {\end{rep@theorem}}}
\makeatother

\theoremstyle{theorem}

\newreptheorem{theorem}{Theorem}
\newtheorem{thm}{Theorem}[section]
\newtheorem*{thm*}{Theorem}
\theoremstyle{definition}
\newtheorem{prop}[thm]{Proposition}
\newtheorem*{prop*}{Proposition}
\newtheorem{defn}[thm]{Definition}
\newtheorem{lem}[thm]{Lemma}

\newtheorem*{cor*}{Corollary}
\theoremstyle{remark}

\newtheorem{rem}[thm]{Remark}

\title{\vspace*{-0.5cm} The asymptotics of the optimal holomorphic \\
extensions of holomorphic jets along submanifolds}

\author
{Siarhei Finski
}

\date{}

\usepackage[%
    left=1in,%
    right=1in,%
    top=1.1in,%
    bottom=0.8in,%
    paperheight=11in,%
    paperwidth=8.5in%
]{geometry}

\newcommand{\imun} {\sqrt{-1}}

\newcommand{\res}{{\rm{Res}}}
\newcommand{\ext}{{\rm{E}}}

\newcommand{\comp}{\mathbb{C}}
\newcommand{\real}{\mathbb{R}}

\newcommand{\nat}{\mathbb{N}}

\newcommand{\dist}{{\rm{dist}}}

\newcommand{\enmr}[1]{\text{End}{(#1)}}

\newcommand{\ccal}{\mathscr{C}}

\newcommand{\dbar}{ \overline{\partial} }

\renewcommand{\Im}{\operatorname{Im}}
\newcommand{\scal}[2]{\langle #1, #2 \rangle}

\makeatletter
\DeclareFontFamily{OMX}{MnSymbolE}{}
\DeclareSymbolFont{MnLargeSymbols}{OMX}{MnSymbolE}{m}{n}
\SetSymbolFont{MnLargeSymbols}{bold}{OMX}{MnSymbolE}{b}{n}
\DeclareFontShape{OMX}{MnSymbolE}{m}{n}{
    <-6>  MnSymbolE5
   <6-7>  MnSymbolE6
   <7-8>  MnSymbolE7
   <8-9>  MnSymbolE8
   <9-10> MnSymbolE9
  <10-12> MnSymbolE10
  <12->   MnSymbolE12
}{}
\DeclareFontShape{OMX}{MnSymbolE}{b}{n}{
    <-6>  MnSymbolE-Bold5
   <6-7>  MnSymbolE-Bold6
   <7-8>  MnSymbolE-Bold7
   <8-9>  MnSymbolE-Bold8
   <9-10> MnSymbolE-Bold9
  <10-12> MnSymbolE-Bold10
  <12->   MnSymbolE-Bold12
}{}

\let\llangle\@undefined
\let\rrangle\@undefined
\DeclareMathDelimiter{\llangle}{\mathopen}%
                     {MnLargeSymbols}{'164}{MnLargeSymbols}{'164}
\DeclareMathDelimiter{\rrangle}{\mathclose}%
                     {MnLargeSymbols}{'171}{MnLargeSymbols}{'171}
\makeatother

\newenvironment{sciabstract}{}

\begin{document} 
\maketitle 

\vspace*{-0.5cm}

\vspace*{0.5cm}

\begin{sciabstract}
  \textbf{Abstract.} We study the asymptotics of the $L^2$-optimal holomorphic extensions of holomorphic jets associated with high tensor powers of a positive line bundle along submanifolds.
  \par 
  More precisely, for a fixed complex submanifold in a complex manifold, we consider the operator which for a given holomorphic jet along the submanifold of a positive line bundle associates the $L^2$-optimal holomorphic extension of it to the ambient manifold.
  When the tensor power of the line bundle tends to infinity, we give an explicit asymptotic formula for this extension operator.
  This is done by a careful study of the Schwartz kernels of the extension operator and related Bergman projectors. It extends our previous results, done for holomorphic sections instead of jets.
\par 
As an application, we prove the asymptotic isometry between two natural norms on the space of holomorphic jets: one induced from the ambient manifold and another from the submanifold.

\end{sciabstract}

\pagestyle{fancy}
\lhead{}
\chead{Asymptotics of extensions of holomorphic jets}
\rhead{\thepage}
\cfoot{}

%\fancypagestyle{mypagestyle}{%
%  \fancyhf{}% Clear header/footer
%  \fancyhead[OC]{An Author}% Author on Odd page, Centred
%  \fancyhead[EC]{A titlesdfdsfdsfds}% Title on Even page, Centred
%  \fancyfoot[C]{\thepage}%
%  \renewcommand{\headrulewidth}{.4pt}% Header rule of .4pt
%}
%\pagestyle{mypagestyle}

\newcommand{\Addresses}{{% additional braces for segregating \footnotesize
  \bigskip
  \footnotesize
  \noindent \textsc{Siarhei Finski, CNRS-CMLS, École Polytechnique F-91128 Palaiseau Cedex, France.}\par\nopagebreak
  \noindent  \textit{E-mail }: \texttt{finski.siarhei@gmail.com}.
}} 

\vspace*{-0.4cm}

\tableofcontents

\section{Introduction}\label{sect_intro}
	The main goal of this paper is to study the asymptotics of the $L^2$-optimal holomorphic extensions of holomorphic jets associated to high tensor powers of a positive line bundle along a submanifold.
	\par 
	More precisely, we fix two (not necessarily compact) complex manifolds $X, Y$, of dimensions $n$ and $m$ respectively.
	We fix also a complex embedding $\iota : Y \to X$, a positive line bundle $(L, h^L)$ over $X$ and an arbitrary Hermitian vector bundle $(F, h^F)$ over $X$.
	In particular, we assume that for the curvature $R^L$ of the Chern connection on $(L, h^L)$, the closed real $(1, 1)$-differential form
	\begin{equation}\label{eq_omega}
		\omega := \frac{\imun}{2 \pi} R^L
	\end{equation}
	is positive.
	We denote by $g^{TX}$ the Riemannian metric on $X$ induced by $\omega$ as follows
	\begin{equation}\label{eq_gtx_def}
		g^{TX}(\cdot, \cdot) := \omega(\cdot, J \cdot),
	\end{equation}
	where $J : TX \to TX$ is the complex structure on $X$.
	We denote by $g^{TY}$ the induced metric on $Y$. 	
	\par 
	\textit{We assume throughout the whole article that the triple $(X, Y, g^{TX})$, and the Hermitian vector bundles $(L, h^L)$, $(F, h^F)$, are of bounded geometry in the sense of Definitions \ref{defn_bnd_subm}, \ref{defn_vb_bg}.}
	\par 
	This means that we assume uniform lower bounds $r_X, r_Y > 0$ on the injectivity radii of $X$, $Y$, the existence of the geodesic tubular neighborhood of $Y$ of uniform size $r_{\perp} > 0$ in $X$, and some uniform bounds on related curvatures and the second fundamental form of the embedding.
	\par 
	Now, we fix some positive (with respect to the orientation given by the complex structure) volume forms  $dv_X$, $dv_Y$ on $X$ and $Y$.
	For smooth sections $f, f'$ of $L^p \otimes F$, $p \in \nat$, over $X$, we define the $L^2$-scalar product using the pointwise scalar product $\langle \cdot, \cdot \rangle_h$, induced by $h^L$ and $h^F$, by
	\begin{equation}\label{eq_l2_prod}
		\scal{f}{f'}_{L^2(X, L^p \otimes F)} := \int_X \scal{f(x)}{f'(x)}_h dv_X(x).
	\end{equation}
	We denote by $H^0_{(2)}(X, L^p \otimes F)$ the vector space of $L^2$-holomorphic sections of $L^p \otimes F$ over $X$.
	\par 
	Given a continuous smoothing linear operator $K : L^2(X, L^p \otimes F) \to L^2(X, L^p \otimes F)$, the Schwartz kernel theorem  guarantees the existence of the Schwartz kernel, $K(x_1, x_2) \in (L^p \otimes F)_{x_1} \otimes (L^p \otimes F)_{x_2}^*$; $x_1, x_2 \in X$, evaluated with respect to $dv_X$, i.e. 
	\begin{equation}
		(Ks) (x_1) = \int_X K(x_1, x_2) \cdot s(x_2) dv_X(x_2), \qquad s \in L^2(X, L^p \otimes F).
	\end{equation}
	Similarly, we define the Schwartz kernels $K_1(y, x)$, $K_2(x, y)$, $x \in X$, $y \in Y$, for smoothing operators $K_1 : L^2(X, L^p \otimes F) \to L^2(Y,\iota^*( L^p \otimes F))$,  $K_2 : L^2(Y,\iota^*( L^p \otimes F)) \to L^2(X, L^p \otimes F)$ with respect to the volume forms $dv_X$ and $dv_Y$ respectively.
	\par
	We denote below by $TX$ and $TY$ the (real) tangent bundles of $X$, $Y$. 
	We identify the holomorphic parts $T^{1, 0}X$, $T^{1, 0}Y$ of the corresponding complexifications with the associated holomorphic tangent bundles.
	We denote by $g^{N}$ the metric on the normal bundle $N$ of $Y$ in $X$ induced by $g^{TX}$.
	We introduce similarly the complex vector bundle $N^{1, 0}$ and endow it with the holomorphic structure given by $T^{1, 0} X / T^{1, 0} Y$.
	From $g^{TX}$, we induce the Hermitian structure on $N^{1, 0}$ and denote it by an abuse of notation by $g^N$.
	Let $P^N : TX|_Y \to N$,  $P^Y : TX|_Y \to TY$, be the orthogonal projections induced by $g^{TX}$.
	Clearly, $\nabla^{N} := P^N \nabla^{TX}|_{Y}$ defines a connection on $N$.
	\par 
	We induce the Hermitian product on ${\rm{Sym}}^k (N^{1, 0})^*$ as explained in (\ref{eq_sym_emb_tens}).
	Then, similarly to (\ref{eq_l2_prod}), using $dv_Y$, we define the $L^2$-product $\scal{\cdot}{\cdot}_{k, L^2(Y, L^p \otimes F)}$ on $H^0_{(2)}(Y, {\rm{Sym}}^k (N^{1, 0})^* \otimes \iota^*( L^p \otimes F))$.
	\par 
	We denote by $\mathcal{J}_Y$ the ideal sheaf of holomorphic germs on $X$, which vanish on $Y$.
	For $k \in \nat$, we endow the space $H^0_{(2)}(X, L^p \otimes F \otimes \mathcal{J}_Y^k)$ with the $L^2$-metric induced by the natural inclusion $
		H^0_{(2)}(X, L^p \otimes F \otimes \mathcal{J}_Y^k) 
		\hookrightarrow
		H^0_{(2)}(X, L^p \otimes F)$.
	\par 
	We assume that for the Riemannian volume forms $dv_{g^{TX}}$, $dv_{g^{TY}}$ of $(X, g^{TX})$, $(Y, g^{TY})$, for any $k \in \nat$, there is $C_k > 0$, such that over $X$ and $Y$, the following bounds hold
	\begin{equation}\label{eq_vol_comp_unif}
		\Big\| \frac{dv_{g^{TX}}}{dv_X} \Big\|_{\ccal^k(X)}, 
		\Big\| \frac{dv_X}{dv_{g^{TX}}} \Big\|_{\ccal^k(X)}, 
		\Big\| \frac{dv_{g^{TY}}}{dv_Y} \Big\|_{\ccal^k(Y)}, 
		\Big\| \frac{dv_Y}{dv_{g^{TY}}} \Big\|_{\ccal^k(Y)} \leq C_k.
	\end{equation}
	\par 
	From Weierstrass division theorem, cf. (\ref{eq_weier_div}), we see that for any $k \in \nat$, a $k$-jet associated to a section from $H^0_{(2)}(X, L^p \otimes F \otimes \mathcal{J}_Y^k)$ is holomorphic. Moreover, in Theorem \ref{thm_res_is_l2}, we prove that for any $k \in \nat$, there is $p_0 \in \nat$, depending only on $k$ and the triple $(X, Y, g^{TX})$, such that for any $p \geq p_0$, the $k$-jets have bounded $L^2$-norm.
	In other words, we prove that the operator
	\begin{equation}\label{eq_deriv_f_along_y}
		\res_{k, p} :  H^0_{(2)}(X, L^p \otimes F \otimes \mathcal{J}_Y^k) \to H^0_{(2)}(Y, {\rm{Sym}}^k (N^{1, 0})^* \otimes \iota^*(L^p \otimes F)),
	\end{equation}
	given by the following identity
	\begin{equation}
		\res_{k, p} (f) = (\nabla^k f)|_Y,
	\end{equation}
	is well-defined, where $\nabla$ is some connection on $L^p \otimes F$ (due to the vanishing condition of $f$ on $Y$, the definition doesn't depend on the choice of the connection).
	\par By extending Ohsawa-Takegoshi theorem for holomorphic jets, in Theorem \ref{thm_ot_weak}, we prove that for any $k \in \nat$, there is $p_0 \in \nat$, such that for any $p \geq p_0$, the operator (\ref{eq_deriv_f_along_y}) is surjective.
	Then, for the \textit{Bergman projector} $B_{k, p}^{Y}$, given by the orthogonal projection from the space of $L^2$-sections $L^2(Y, {\rm{Sym}}^k (N^{1, 0})^* \otimes  \iota^*(L^p \otimes F))$ to $H^{0}_{(2)}(Y, {\rm{Sym}}^k (N^{1, 0})^* \otimes \iota^*(L^p \otimes F))$, we define the \textit{extension} operator
	\begin{equation}\label{eq_ext_op}
		\ext_{k, p} :  L^2(Y, {\rm{Sym}}^k (N^{1, 0})^* \otimes \iota^*(L^p \otimes F)) \to H^{0}_{(2)}(X, L^p \otimes F \otimes \mathcal{J}_Y^k),
	\end{equation}
	by putting $\ext_{k, p} (g) = f$, where $f$ satisfies $\res_{k, p} (f) = B_{k, p}^{Y} g$ and has the minimal $L^2$-norm among the sections from $H^{0}_{(2)}(X, L^p \otimes F \otimes \mathcal{J}_Y^k)$ with this property (since $H^{0}_{(2)}(X, L^p \otimes F \otimes \mathcal{J}_Y^{k + 1})$ is closed inside of $H^{0}_{(2)}(X, L^p \otimes F \otimes \mathcal{J}_Y^k)$, such $f$ is unique).
	Ohsawa-Takegoshi extension theorem for holomorphic jets in our setting means precisely that the operator $\ext_{k, p}$ is well-defined and bounded as long as $p$ is big enough.
	Study of such operator is motivated by the previous works of Popovici \cite{PopovNonredExt}, Demailly \cite{DemExtRed} and Cao-Demailly-Matsumura \cite{CaoDemMatsumura} on the extension of holomorphic jets in the classical setting (i.e. $p = 0$), and the works of Ohsawa-Takegoshi \cite{OhsTak1}, \cite{Ohsawa}, Manivel \cite{ManExt} and Demailly \cite[\S 13]{DemBookAnMet}, on the extensions of holomorphic sections (again for $p = 0$).
	\par 
	One of the main goals of this article is to find an explicit asymptotic expansion of $\ext_{k, p}$, as $p \to \infty$.
	To describe our results precisely, we need to fix some further notation.
	\par 	
	For $y \in Y$, $Z_N \in N_y$, let $\real \ni t \mapsto \exp_y^{X}(tZ_N) \in X$ be the geodesic in $X$ in direction $Z_N$, where we identified $N_y$ as an orthogonal complement of $T_yY$ in $T_yX$.
	Bounded geometry condition means, in particular, that this map induces a diffeomorphism of $r_{\perp}$-neighborhood of the zero section in $N$ with a tubular neighborhood $U$ of $Y$ in $X$.
	From now on, we use this identification implicitly.
	Of course, $(y, 0)$, $y \in Y$, then corresponds to $Y$.
	\par 
	We denote by $\pi : U \to Y$ the natural projection $(y, Z_N) \mapsto y$. 
	Over $U$, we identify $L, F$ to $\pi^* (L|_Y), \pi^* (F|_Y)$ by the parallel transport with respect to the respective Chern connections along the geodesic $[0, 1] \ni t \mapsto (y,t  Z_N) \in X$, $|Z_N| < r_{\perp}$. We also define a function $\kappa_N$ as follows
	\begin{equation}\label{eq_kappan}
		dv_X = \kappa_N dv_Y \wedge dv_N,
	\end{equation}
	where $dv_N$ is the relative Riemannian volume form on $(N, g^N)$.
	Of course, we have $\kappa_N|_Y = 1$ if
	\begin{equation}\label{eq_comp_vol_omeg}
		dv_X = dv_{g^{TX}}, \qquad dv_Y = dv_{g^{TY}}.
	\end{equation}
	\par 
	We fix a smooth function $\rho : \real_{+} \to [0, 1]$, satisfying
	\begin{equation}\label{defn_rho_fun}
		\rho(x) =
		\begin{cases}
			1, \quad \text{for $x < \frac{1}{4}$},\\
			0, \quad \text{for $x > \frac{1}{2}$}.
		\end{cases}
	\end{equation}
	We fix $g \in \ccal^{\infty}(Y, {\rm{Sym}}^k N^* \otimes \iota^* (L^p \otimes F))$, $k \in \nat$, and define using the above isomorphisms over $U$ the section $\{ g \} \in \ccal^{\infty}(X, L^p \otimes F)$ as follows
	\begin{equation}\label{eq_brack_defn}
		\{ g \} (y, Z_N)
		:=
		\rho \Big(\frac{|Z_N|}{r_{\perp}} \Big) \cdot g(y) \cdot Z_N^{\otimes k},
	\end{equation}
	where the norm $|Z_N|$, is taken with respect to $g^{N}$.
	Away from $U$, we extend $\{ g \}$ by zero.
	\par 
	\begin{sloppypar}
	We define the operator $\ext_{k, p}^0 : L^2(Y, {\rm{Sym}}^k (N^{1, 0})^* \otimes \iota^*(L^p \otimes F)) \to L^2(X, L^p \otimes F)$ as follows.
	For $g \in L^2(Y, {\rm{Sym}}^k (N^{1, 0})^* \otimes \iota^*(L^p \otimes F))$, we let $(\ext_{k, p}^0 g)(x) = 0$, $x \notin U$, and in $U$, we put
	\begin{equation}\label{eq_ext0_op}
		(\ext_{k, p}^{0} g)(y, Z_N) = \{ B_{k, p}^Y g \} (y, Z_N) \exp \Big(- p \frac{\pi}{2} |Z_N|^2 \Big).
	\end{equation}
	where the norm $|Z_N|$, $Z_N \in N$, is taken with respect to $g^{N}$.
	Now, for $g \in H^{0}_{(2)}(Y, {\rm{Sym}}^k (N^{1, 0})^* \otimes \iota^*(L^p \otimes F))$, the section $\ext_{k, p}^{0} g$ satisfies $(\nabla^k \ext_{k, p}^{0} g)|_Y = g$, but $\ext_{k, p}^{0} g$ is not holomorphic over $X$, unless $g = 0$.
	Nevertheless, as we shall see, $\ext_{k, p}^{0} g$ can be used to approximate very well the holomorphic section $\ext_{k, p} g$.
	More precisely, we have the following result.
	\end{sloppypar}
	\begin{thm}\label{thm_high_term_ext}
		For any $k \in \nat$, there are $C > 0$, $p_1 \in \nat^*$, such that for any $p \geq p_1$, we have 
		\begin{equation}\label{eq_ext_as}
			\big\| \ext_{k, p} - \ext_{k, p}^{0} \big\| \leq \frac{C}{p^{\frac{n - m + k + 1}{2}}}.
		\end{equation}
		where $\| \cdot \|$ is the operator norm.
		Also, as $p \to \infty$, we have
		\begin{equation}\label{eq_norm_asymp}
			\big\|  \ext_{k, p}^{0} \big\| \sim \frac{1 }{p^{\frac{n - m + k}{2}}} \cdot  \sup_{y \in Y} \kappa_N^{\frac{1}{2}}(y) \cdot \frac{1}{\sqrt{k! \cdot (2 \pi)^k}}.
		\end{equation}
		Moreover, under assumption (\ref{eq_comp_vol_omeg}), in (\ref{eq_ext_as}), one can replace $p^{-\frac{n - m + k + 1}{2}}$ by an asymptotically better estimate if and only if $Y$ is a totally geodesic submanifold of $(X, g^{TX})$, i.e. the second fundamental form, see (\ref{eq_sec_fund_f}), vanishes.
	\end{thm}
	\begin{rem}
		a) For $k = 0$, this result was proved in \cite[Theorem 1.1]{FinOTAs}. The proof we present here is different even for $k = 0$. It is still based a lot on the ideas and the techniques from \cite{FinOTAs} and \cite{FinToeplImm}.
		\par
		b) The boundness of $\kappa_N$ follows from the bounded geometry condition, see \cite[Corollary 2.10]{FinOTAs}. 
		In particular, the right-hand side of (\ref{eq_norm_asymp}) is finite.
		\par 
		c)
		Our result refines a theorem of Randriambololona \cite[Théorème 3.1.10]{RandriamTh}, stating in the compact case that for any $\epsilon > 0$, $k \in \nat$, there is $p_1 \in \nat^*$, such that $\big\|  \ext_{k, p} \big\| \leq \exp(\epsilon p)$ for $p \geq p_1$.
	\end{rem}
	\par 
	Let us now describe our second result, establishing a relation between two natural metrics on the space of holomorphic jets.
	For this, define the map 
	\begin{multline}\label{eq_defn_jetmap}
		{\rm{Jet}}_{k, p}: 
		H^0_{(2)}(X, L^p \otimes F) / H^0_{(2)}(X, L^p \otimes F \otimes \mathcal{J}_Y^{k + 1})
		\\
		\to
		\oplus_{l = 0}^{k} H^0_{(2)}(Y, {\rm{Sym}}^l (N^{1, 0})^* \otimes \iota^*(L^p \otimes F)),
	\end{multline}
	as follows.
	We let ${\rm{Jet}}_{k, p} (f) = (g_0, \ldots, g_k)$ if and only if for any $r \in \nat$, $r \leq k$, the following holds
	\begin{equation}\label{eq_jet_char}
		f - \sum_{l = 0}^{r} \ext_{l, p} g_l \in H^0_{(2)}(X, L^p \otimes F \otimes \mathcal{J}_Y^{r + 1}).
	\end{equation}
	Alternatively, we let $g_0 = f|_Y$, and inductively define 
	\begin{equation}
		g_{r} 
		:=
		\res_{r, p} \Big( f - \sum_{l = 0}^{r - 1} \ext_{l, p} g_l \Big).
	\end{equation}
	\par 
	As a direct consequence of the Ohsawa-Takegoshi extension theorem for holomorphic jets, cf. Theorem \ref{thm_ot_weak}, we obtain that for any $k \in \nat$, there is $p_1 \in \nat$, such that the map ${\rm{Jet}}_{k, p}$ is an isomorphism for $p \geq p_1$.
	As we explain next, this statement can be refined in the metric setting.
	\par 
	More precisely, consider the scalar product $\scal{\cdot}{\cdot}_{{\rm{Jet}}_{k, p}}$ on the space $\oplus_{l = 0}^{k} H^0_{(2)}(Y, {\rm{Sym}}^l (N^{1, 0})^* \otimes \iota^*(L^p \otimes F))$, defined as follows
	\begin{equation}\label{eq_scal_prod_jet}
		\scal{\cdot}{\cdot}_{{\rm{Jet}}_{k, p}}
		:=
		\sum_{r = 0}^{k}
		\frac{1}{r! \cdot (2 \pi)^r} \frac{1}{p^{n - m + r}} \cdot \scal{\cdot}{\cdot}_{r, L^2(Y, L^p \otimes F)}
	\end{equation}
	Let us now consider the scalar product on $H^0_{(2)}(X, L^p \otimes F) / H^0_{(2)}(X, L^p \otimes F \otimes \mathcal{J}_Y^{k + 1})$ induced by the $L^2$-scalar product on $H^0_{(2)}(X, L^p \otimes F)$, which we denote by an abuse of notation by $\scal{\cdot}{\cdot}_{L^2(X, L^p \otimes F)}$.
	\begin{thm}\label{thm_isom}
		Under the assumption (\ref{eq_comp_vol_omeg}), the map ${\rm{Jet}}_{k, p}$ is an asymptotic isometry with respect to the above scalar products.
		More precisely, for the norms $\| \cdot \|_{L^2(X, L^p \otimes F)}$, $\| \cdot \|_{{\rm{Jet}}_{k, p}}$ associated to the scalar products defined above, there are $C > 0$, $p_1 \in \nat$, such that for any $p \geq p_1$, we have
		\begin{equation}
			1 - \frac{C}{p} \leq  \frac{\| \cdot \|_{{\rm{Jet}}_{k, p}}}{\| \cdot \|_{L^2(X, L^p \otimes F)}} \leq 1 + \frac{C}{p}.
		\end{equation}
	\end{thm}
	\begin{rem}
		a) An analogue of Theorem \ref{thm_isom} in the setting of geometric quantization was established by Ma-Zhang \cite[Theorem 0.10]{MaZhBKSR}.
		Authors proved that there is a similar asymptotic relation between the $L^2$-metric on the space of holomorphic sections invariant under the Hamiltonian action of a compact connected Lie group and the $L^2$-metric on the corresponding symplectic reduction.
		\par 
		b) For $k = 0$, a refinement of Theorem \ref{thm_isom} was established in \cite[Theorem 1.8]{FinToeplImm}.
	\end{rem}
	\par
	The proofs of Theorems \ref{thm_high_term_ext}, \ref{thm_isom} are done in Section \ref{sect_asy_main_op}. 
	The core of those results lies in several technical statements.
	Most relevant are Theorem \ref{thm_ext_exp_dc}.a), where we establish the exponential bound for the Schwartz kernel of the extension operator and Theorems \ref{thm_ext_as_expk}, where we establish the asymptotic expansion of this Schwartz kernel.
	In Theorems \ref{thm_ext_exp_dc}.b) and \ref{thm_berg_perp_off_diagk}, we also prove the corresponding statements for the related Bergman projector.
	\par 
	Those more precise results seem to be new even when $Y = \{ x \}$, where $x \in X$ is a fixed point.
	In this case, they give some results on the asymptotic expansion of \textit{higher order peak sections}, previously studied by Tian \cite{TianBerg}. See Section \ref{sect_peak_ho} for more details.
	\par 
	As the last application of our techniques, we prove the asymptotic expansion of the Schwartz kernel of the logarithmic Bergman kernel of order $k \in \nat$.
	More precisely, let $B_p^{X, k  Y}$ be the orthogonal projection from $L^2(X, L^p \otimes F)$ to $H^0_{(2)}(X, L^p \otimes F \otimes \mathcal{J}_Y^k)$.
	We call it the \textit{logarithmic Bergman kernel of order} $k$.
	The following theorem shows that away from $Y$, $B_p^{X, k  Y}$ is asymptotically close to $B_p^X$, and in a tubular neighborhood around $Y$, it can be expressed as a product of the Bergman kernel $B_p^Y$ and the logarithmic Bergman kernel of the model space, as described in Section \ref{sect_model_calc}.
	\begin{thm}\label{thm_log_bk}
		There are $c > 0$, $p_1 \in \nat^*$, such that for any $k, r, l \in \nat$, there is $C > 0$, such that for any $p \geq p_1$, $x_1, x_2 \in X$, the following estimate holds
		\begin{multline}\label{eq_ext_log_bk}
			\Big|  \big( B_p^{X, k Y} - B_p^{X} \big)(x_1, x_2) \Big|_{\ccal^r} 
			\\
			\leq C p^{n + \frac{r}{2}} \exp \Big(- c \sqrt{p} \cdot \big( \dist(x_1, x_2) + \dist(x_2, Y) + \dist(x_1, Y) \big) \Big),
		\end{multline}
		where the pointwise $\ccal^r$-norm at a point $(x_1, x_2) \in X \times X$  is the sum of the norms induced by $h^L, h^F$ and $g^{TX}$, evaluated at $(x_1, x_2)$, of the derivatives up to order $r$ with respect to the connection induced by the Chern connections on $L, F$ and the Levi-Civita connection on $TX$. 
		Moreover, in the tubular neighborhood around $Y$, under the identification as described before (\ref{eq_kappan}), the following holds.
		There are $c, \epsilon, C, Q > 0$, $p_1 \in \nat^*$, such that for any $y_1, y_2 \in Y$, $Z_1, Z_2 \in \real^{2(n - m)}$, $|Z_1|, |Z_2| < \epsilon$, $p \geq p_1$, we have
		\begin{multline}\label{eq_ext_log_bk2}
			\bigg|  B_p^{X, k Y} \big( (y_1, Z_1), (y_2, Z_2) \big) 
			- 
			p^{n - m}
			\cdot
			B_p^{Y}(y_1, y_2) 
			\cdot
			\kappa_N^{-\frac{1}{2}}(y_1)
			\cdot
			\kappa_N^{-\frac{1}{2}}(y_2)
			\cdot
			\\
			\cdot 
			\exp \Big(
			 	-\frac{\pi}{2} \sqrt{p} \big( |Z_1|^2 + |Z_2|^2 \big)
			\Big)
			\cdot
			\sum_{r = k + 1}^{\infty}
			\pi^{r}
			\cdot
			p^r
			\cdot
			\sum_{\substack{\beta \in \nat^{n - m} \\ |\beta| = r}}
			\frac{ (z_1 \overline{z}_2)^{\beta} }{\beta!}
			\bigg|_{\ccal^r} 
			\leq C 
			p^{n + k + \frac{r - 1}{2}}			
			\cdot
			|Z_1 Z_2|^k
			\cdot
			\\
			\cdot
			\big( 1 + \sqrt{p}|Z_1| +  \sqrt{p}|Z_2| \big)^Q 
			\cdot
			\exp \Big(- c \sqrt{p} \cdot \big( \dist_X(y_1, y_2) + |Z_1 - Z_2| \big) \Big),
		\end{multline}
		where the $\ccal^r$-norm is taken with respect to $y_1, y_2, Z_1, Z_2$, and we used the multiindex notation as we explain in (\ref{mult_not}).
	\end{thm}
	\begin{rem}
		The study of the asymptotic expansion of $B_p^{X, k Y}$ has recently received considerable attention for $k = [\epsilon p]$, where $\epsilon > 0$ is some small constant. See for example Ross-Thomas \cite{RossThomasObstr}, Pokorny-Singer \cite{PokSing}, Ross-Singer \cite{RossSinger}, Zelditch-Zhou \cite{ZeldZhouInter}, Coman-Marinescu \cite{ComMarPartBerg}.
		We believe that our methods can be applied in this asymptotic regime as well. 
	\end{rem}
	\par 
	Let us finally say few words about the tools we use in this article.
	The proofs of Theorems \ref{thm_high_term_ext}, \ref{thm_isom} rely on the exponential estimate for the Bergman kernel, cf. Ma-Marinescu \cite{MaMarOffDiag}, on the asymptotic expansion of the Bergman kernel due to Dai-Liu-Ma \cite{DaiLiuMa},  and on some technical results about the algebras of operators with Taylor-type expansion of the Schwartz kernel, which are inspired by the work of Ma-Marinescu \cite{MaMarToepl}, cf. \cite[\S 7]{MaHol}.
	\par 
	As an important intermediate result, we establish the asymptotic version of Ohsawa-Takegoshi extension theorem for holomorphic jets.
	For this, we follow our strategy from \cite{FinOTAs}, which treats holomorphic sections instead of jets, and which was itself inspired by Bismut-Lebeau \cite{BisLeb91} and Demailly \cite{Dem82}.
	Using this result, we prove the existence of a sequence of operators, which we call multiplicative defect, relating the extension operator and the adjoint of the restriction operator.
	We establish then that this sequence of operators forms a Toeplitz operator with weak exponential decay, which is a notion introduced in \cite{FinToeplImm} (as a refinement of the notion of Toeplitz operators in the sense of \cite[\S 7]{MaHol}). 
	The asymptotic criteria for those operators, established in \cite{FinToeplImm}, relying on the analogous result of Ma-Marinescu \cite[Theorem 4.9]{MaMar08a} for compact manifolds and Toeplitz operators in the sense of \cite[\S 7]{MaHol}, plays a foundational role in our approach.
	\par
	Differently from the approach from \cite{FinOTAs}, the main results of this article are obtained through the use of the multiplicative defect. We also use the induction on the order of jets, $k \in \nat$.
	\par 
	The general strategy for dealing with semi-classical limits here is inspired by Bismut \cite{BisDem} and Bismut-Vasserot \cite{BVas}.
	\par 
	To conclude, we mention that recently there has been a surge of interest in Ohsawa-Takegoshi extension theorem for holomorphic jets, see for example Hosono \cite{HosonoJet}, McNeal-Varolin \cite{McNealVarolin}, Rao-Zhang \cite{RaoZhang}. See also Cao-P{\u a}un \cite{CaoPaunOT} for an application of a version of extension theorem for jets to a conjecture of deformational invariance of plurigenera for Kähler families.
	\par 
	This article is organized as follows.
	In Section 2, we recall the bounded geometry assumptions, and the results about the convergence of the exponential integrals on such manifolds.
	In Section 3, we recall the kernel calculus which studies the composition rules of basic operators on the model vector space.
	We recall a notion of Toeplitz operators with weak exponential decay and an asymptotic criteria for those operators.
	In Section 4, we prove the existence of a sequence of operators, which we call multiplicative defect.
	Finally, in Section 5, by the use of the above results, we prove the exponential bounds for the extension operator, orthogonal Bergman kernel of order $k$, and study the asymptotic expansion of their Schwartz kernels.
	As a consequence of those studies, we establish the main results of this article.
	\par {\bf{Notations.}}
	For $\alpha = (\alpha_1, \ldots, \alpha_k) \in \nat^k$, $B = (B_1, \ldots, B_k) \in \comp^k$, we write by 
	\begin{equation}\label{mult_not}
		|\alpha| = \sum_{i = 1}^{k} \alpha_i, \quad \alpha! = \prod_{i = 1}^{k} \alpha_i!, \quad B^{\alpha} = \prod_{i = 1}^{k} B_i^{\alpha_i}.
	\end{equation}
	\par
	Let $(V, h^V)$ be a Hermitian (or Euclidean) vector space.
	We endow ${\rm{Sym}}^k V$ with a Hermitian metric induced by the induced metric on $V^{\otimes k}$ and the inclusion ${\rm{Sym}}^k V \to V^{\otimes k}$, defined as
	\begin{equation}\label{eq_sym_emb_tens}
		v_1 \odot \ldots \odot v_k \mapsto \frac{1}{k!} \sum v_{\sigma(1)} \otimes \ldots \otimes v_{\sigma(k)},
	\end{equation}
	where the sum runs over all permutations $\sigma$ on $k$ indices.
	Clearly, if $v_1, \cdots, v_l$ form an orthonormal basis of $V$, then $\sqrt{\frac{k!}{\alpha!}} \cdot v^{\odot \alpha}$, $\alpha \in \nat^k$, $|\alpha| = k$, forms an orthonormal basis of ${\rm{Sym}}^k V$.
	The inclusion (\ref{eq_sym_emb_tens}) gives also a natural isomorphism ${\rm{Sym}}^k V^* \to ({\rm{Sym}}^k V)^*$.
	For a fixed basis $v_1, \ldots, v_l$ of $V$ and the dual basis $u_1, \ldots, u_l$ of $V^*$, we have the following relation
	\begin{equation}\label{eq_pairing_sym}
		u^{\odot \beta} ( v^{\odot \alpha} ) = 
		\begin{cases}
			0, \text{ if } \alpha \neq \beta, \\
			\frac{\alpha!}{k!}, \text{ otherwise.}
		\end{cases}
	\end{equation}
	\par 
	For a $(V, h^V)$ be a Hermitian vector space with a complex structure $J$, we denote by $V^{1, 0}$, $V^{0, 1}$ the holomorphic and antiholomorphic components of $V \otimes \comp$.
	We endow $V^{1, 0}$ and $V^{0, 1}$ with the natural metrics, verifying $\|v \pm \imun J v \| = \sqrt{2} \| v \|$.
	In this way, when $V = \comp^l$ is endowed with the standard Hermitian product and with the usual linear (complex) coordinates $z_i$, we have $\| dz_i \| = \sqrt{2}$ and $\| \frac{\partial}{\partial z_i} \| = \frac{1}{\sqrt{2}}$.
	\par 
	We use notations $X, Y$ for complex manifolds and $M, H$ for real manifolds.
	The complex (resp. real) dimensions of $X, Y$ (resp. $M, H$) are denoted here by $n, m$.
	An operator $\iota$ always means an embedding $\iota : Y \to X$ (resp. $\iota : H \to M$).
	We denote by $\res_Y$ (resp. $\res_H$) the restriction operator from $X$ to $Y$ (resp. $M$ to $H$). 
	\par 
	For a Riemannian manifold $(M, g^{TM})$, we denote the Levi-Civita connection by $\nabla^{TM}$, by $R^{TM}$ the curvature of it, and by $dv_{g^{TM}}$ the Riemannian volume form.
	For a closed subset $W \subset M$, $r \geq 0$, let $B_{W}^{M}(r)$ be the ball of radius $r$ around $W$. 
	For a Hermitian vector bundle $(E, h^E)$, note
	$
		B_{r}(E) := \{ Z \in E : |Z|_{h^E} < r \}
	$.
	\par 
	For a fixed volume form $dv_M$ on $M$, we denote by $L^2(dv_M, h^E)$ the space of $L^2$-sections of $E$ with respect to $dv_M$ and $h^E$. When $dv_M = dv_{g^{TM}}$, we also use the notation $L^2(g^{TM}, h^E)$. When there is no confusion about the data, we also use the simplified notation $L^2(M, E)$ or $L^2(M)$.
	\par 
	For $n \in \nat^*$, we denote by $dv_{\comp^n}$ the standard volume form on $\comp^n$.
	We view $\comp^m$ (resp. $\real^m$) embedded in $\comp^n$ (resp. $\real^n$) by the first $m$ coordinates.
	For $Z \in \real^k$, we denote by $Z_l$, $l = 1, \ldots, k$, the coordinates of $Z$.
	If $Z \in \real^{2n}$, we denote by $z_i$, $i = 1, \ldots, n$, the induced complex coordinates $z_i = Z_{2i - 1} + \imun Z_{2i}$.
	We frequently use the decomposition $Z = (Z_Y, Z_N)$, where $Z_Y = (Z_1, \ldots, Z_{2m})$ and $Z_N = (Z_{2m + 1}, \ldots, Z_{2n})$. 
	For a fixed frame $(e_1, \ldots, e_{2n})$ in $T_{x}X$, $x \in X$, (resp. $y \in Y$) we implicitly identify $Z$ (resp. $Z_Y$, $Z_N$) to an element in $T_xX$ (resp. $T_yY$, $N_y$) by
	\begin{equation}\label{eq_Z_ident}
		Z = \sum_{i = 1}^{2n} Z_i e_i, \quad Z_Y = \sum_{i = 1}^{2m} Z_i e_i, \quad Z_N = \sum_{i = 2m + 1}^{2n} Z_i e_i.
	\end{equation}
	If the frame $e_i$ satisfies the condition 		
	\begin{equation}\label{eq_cond_jinv}
		J e_{2i - 1} = e_{2i},
	\end{equation}
	we denote $\frac{\partial}{\partial z_i} := \frac{1}{2} (e_{2i-1} - \imun e_{2i})$, $\frac{\partial}{\partial \overline{z}_i} := \frac{1}{2} (e_{2i-1} + \imun e_{2i})$, and identify $z, \overline{z}$ to vectors in $T_xX \otimes_{\real} \comp$ as follows
	\begin{equation}\label{eq_z_ovz_id}
		z = \sum_{i = 1}^{n} z_i \cdot \frac{\partial}{\partial z_i}, \qquad
		\qquad
		\overline{z} = \sum_{i = 1}^{n} \overline{z}_i \cdot \frac{\partial}{\partial \overline{z}_i}.
	\end{equation}
	Clearly, in this identification, $Z = z + \overline{z}$.
	We define $z_Y, \overline{z}_Y \in T_yY \otimes_{\real} \comp$, $z_N, \overline{z}_N \in N_y \otimes_{\real} \comp$ in a similar way.

\section{Second fundamental form and bounded geometry}\label{sect_bound_geom}
	The main goal of this section is to recall the basic facts about the geometry of manifolds of bounded geometry and the second fundamental form.
	More precisely, in Section \ref{sect_sff}, we recall the definition and various properties satisfied by the second fundamental form.
	We also recall the definition of manifolds (resp. pairs of manifolds, vector bundles) of bounded geometry.
	In Section \ref{sect_bnd_geom_cf}, we recall the results about the convergence of exponential integrals on manifolds of bounded geometry.
	
\subsection{The second fundamental form and bounded geometry assumption}\label{sect_sff}
	Here we recall some basic facts about the second fundamental form and its relation with bounded geometry assumptions.	
	Let $H$ be an embedded submanifold of a Riemannian manifold $(M, g^{TM})$, $g^{TH} := g^{TM}|_H$.
	We identify the normal bundle $N^{M|H}$ of $H$ in $M$ to an orthogonal complement of $TH$ in $TM$ as 
	\begin{equation}\label{eq_tx_rest}
		TM|_H \to TH \oplus N^{M|H}.
	\end{equation}
	We denote by $g^{N^{M|H}}$ the metric on $N^{M|H}$ induced by $g^{TM}$. 
	We denote by $P_N^{M|H} : TM|_H \to N^{M|H}$,  $P_H^{M|H} : TM|_H \to TH$, the projections induced by (\ref{eq_tx_rest}).
	Clearly, $\nabla^{N^{M|H}} := P_N^{M|H} \nabla^{TM}|_H$ defines a connection on $N^{M|H}$.
	Recall that the definition of the \textit{second fundamental form} $A \in \ccal^{\infty}(H, T^*H \otimes \enmr{TM|_H})$ is given by
	\begin{equation}\label{eq_sec_fund_f}
		A := \nabla^{TM}|_{H} - \nabla^{TH} \oplus \nabla^{N^{M|H}}.
	\end{equation}
	Recall that the \textit{mean curvature} $\nu^{M|H} \in \ccal^{\infty}(H, N^{M|H})$ of $\iota$ is defined as follows
	\begin{equation}\label{eq_mn_curv_d}
		\nu^{M|H} := \frac{1}{m} \sum_{i=1}^{m} A(e_i)e_i,
	\end{equation}
	where the sum runs over an orthonormal basis of $(TH, g^{TH})$.
	\begin{prop}[{cf. \cite[Proposition 2.3]{FinToeplImm}}]\label{prop_prop_sfndform}
		The second fundamental form satisfies the following properties.
		\begin{enumerate}
			\item It takes values in skew-symmetric endomorphisms of $TM|_H$, interchanging $TH$ and $N^{M|H}$.
			\item For any $U, V \in TH$, we have $A(U) V = A(V) U$.
		\end{enumerate}
		Assume, moreover, that $(M, g^{TM})$ is Kähler. Then the following holds.
		\begin{enumerate}[resume]
			\item $A$ commutes with the action of the complex structure.
			\item For any $U \in TH$, $V \in TM$, $U = u + \overline{u}$, $V = v + \overline{v}$, $u, v \in T^{1, 0}M$, we have
			\begin{equation}
			\begin{aligned}
				&
				A(U)v = A(\overline{u}) v, 
				&&
				A(U)\overline{v} = A(u) \overline{v}, 
				&&&
				\text{if } V \in N^{M|H},
				\\
				&
				A(U)v = A(u) v, 
				&&
				A(U)\overline{v} = A(\overline{u}) \overline{v},
				&&&
				\text{if } V \in TH.
			\end{aligned}
			\end{equation}
			\item We have $\nu^{M|H} = 0$.
		\end{enumerate}
	\end{prop}
	\par 
	Let us now recall the definitions of manifolds (resp. pairs of manifolds, vector bundles) of bounded geometry. 
	For more detailed overview of this part, refer to \cite{EichBoundG}, \cite{SchBound}, \cite{GrosSchnBound}, cf. also \cite{FinOTAs}.
	Now, for a Hermitian vector bundle $(E, h^E)$ with a fixed connection $\nabla^E$ over $M$, we denote 
	\begin{equation}
		\ccal^{\infty}_{b}(M, E) 
		:= 
		\Big\{
			f \in \ccal^{\infty}(M, E) : \text{ for any } k \in \nat, \text{ there is } C > 0, \text{ so that } |\nabla^k f| \leq C
		\Big\},
	\end{equation}
	where $\nabla$ is the connection induced by $\nabla^E$ and the Levi-Civita connection on $TM$, and $| \cdot |$ is the norm induced by the metrics $g^{TM}$, $h^E$.
	When $M$ is complex and $E$ is holomorphic, we implicitly take $\nabla^E$ to be the associated Chern connection.
	\begin{defn}\label{defn_bnd_g_man}
		We say that a Riemannian manifold $(M, g^{TM})$ is of bounded geometry if the following two conditions are satisfied.
		\\ \hspace*{0.3cm} \textit{(i)} The injectivity radius of $(M, g^{TM})$ is bounded below by a positive constant $r_M$. 
		\\ \hspace*{0.3cm} \textit{(ii)} For the Riemann curvature tensor $R^{TM}$ of $M$, we have $R^{TM} \in \ccal^{\infty}_b(M, \Lambda^2 T^*M \otimes \enmr{TM})$.
	\end{defn}
	\begin{rem}\label{rem_hopf_rin}
		By Hopf-Rinow theorem, the condition \textit{(i)} implies that $(M, g^{TM})$ is complete.	
	\end{rem}
	\par 
	\begin{defn}\label{defn_bnd_subm}
		We say that the triple $(M, H, g^{TM})$ of a Riemannian manifold $(M, g^{TM})$ and a submanifold $H$ is of bounded geometry if the following conditions are fulfilled.
		\\ \hspace*{0.3cm} \textit{(i)} The manifold $(M, g^{TM})$ is of bounded geometry.
		\\ \hspace*{0.3cm} \textit{(ii)} The injectivity radius of $(H, g^{TH})$ is bounded below by a positive constant $r_H$. 
		\\ \hspace*{0.3cm} \textit{(iii)} There is a collar around $H$ (a tubular neighborhood of fixed radius), i.e. there is $r_{\perp} > 0$ such that for any $x, y \in H$, the normal geodesic balls $B^{\perp}_{r_{\perp}}(x), B^{\perp}_{r_{\perp}}(y)$, obtained by the application of the exponential mapping to vectors, orthogonal to $H$, of norm bounded by $r_{\perp}$, are disjoint.
		\\ \hspace*{0.3cm} \textit{(iv)} 
		The second fundamental form, $A$, satisfies $A \in \ccal^{\infty}_b(H, T^*M|_H \otimes \enmr{TM|_H})$.
	\end{defn}	
	We will now introduce a local coordinate system in $M$, which is particularly well-adapted to the study of triples of bounded geometry.
	We fix a point $y_0 \in H$ and an orthonormal frame $(e_1, \ldots, e_m)$ (resp. $(e_{m+1}, \ldots, e_n)$) in $(T_{y_0}H, g_{y_0}^{TH})$ (resp. in $(N_{y_0}^{M|H}, g_{y_0}^{N^{M|H}})$). 
	For $Z = (Z_H, Z_N)$, $Z_H \in \real^{m}$, $Z_N \in \real^{n - m}$, $Z_H = (Z_1, \ldots, Z_m)$, $Z_N = (Z_{m + 1}, \ldots, Z_{n})$, $|Z_H| \leq r_H$, $|Z_N| \leq r_{\perp}$, we define a coordinate system $\psi_{y_0}^{M|H} : B_0^{\real^{m}}(r_H) \times B_0^{\real^{n - m}}(r_{\perp}) \to M$ by 
	\begin{equation}\label{eq_defn_fermi}
		\psi_{y_0}^{M|H}(Z_H, Z_N) := \exp_{\exp_{y_0}^{H}(Z_H)}^{M}(Z_N(Z_H)),
	\end{equation}
	where $Z_N(Z_H)$ is the parallel transport of $Z_N \in N_{y_0}^{M|H}$ along $\exp_{y_0}^{H}(t Z_H)$, $t = [0, 1]$, with respect to the connection $\nabla^{N^{M|H}}$ on $N^{M|H}$.
	The coordinates $\psi_{y_0}^{M|H}$ are called the \textit{Fermi coordinates} at $y_0$.
	\par 
	In the special case when $Y = X$ and $x_0 := y_0$, those coordinates correspond to the exponential coordinates $\phi_{x_0}^X : \real^{2n} \to X$, $x_0 \in X$, defined as follows
	\begin{equation}\label{eq_phi_defn}
		\phi_{x_0}^{X}(Z) := \exp^{X}_{x_0}(Z).
	\end{equation}	
	The importance of Fermi coordinates from the following proposition.
	\begin{prop}[{\cite[Lemma 3.9]{SchBound}, \cite[Theorem 4.9]{GrosSchnBound}}]\label{prop_bndg_tripl}
		For any triple $(M, H, g^{TM})$ of bounded geometry, for any $r_0 > 0$, there is $D_k > 0$, such that for any $y_0 \in H$, $l = 0, \ldots, k$, we have
		\begin{equation}\label{eq_bndtr_metr_tens}
			\| g_{ij} \|_{\ccal^{l}(B_0^{\real^n}(r_0))} \leq D_k, \quad \| g^{ij} \|_{\ccal^{l}(B_0^{\real^n}(r_0))} \leq D_k.
		\end{equation}
		where $g_{ij}$, $i, j = 1, \ldots, n$, are the coefficients of the metric tensor $\psi_{y_0}^* g^{TM}$, and $g^{ij}$ are the coefficients of the inverse matrix.
	\end{prop}
	\begin{rem}\label{rem_bndg_tripl}
		a) In particular, for a triple of bounded geometry $(M, H, g^{TM})$, the Riemannian manifold $(H, g^{TH})$ has bounded geometry.
		\par 
		b) Clearly, this result along with the assumption (\ref{eq_vol_comp_unif}) imply that in the notations of (\ref{eq_kappan}), we have $\kappa_N \in \ccal^{\infty}_{b}(U)$.
		\par 
		c)
		By taking $Y = \{p\}$, $p \in M$, we see that the analogue of Proposition \ref{prop_bndg_tripl} holds for manifolds of bounded geometry and exponential coordinates $\phi_{x_0}^{M}$, introduced in (\ref{eq_phi_defn}), accordingly to \cite[Theorem A]{EichBoundG}.
	\end{rem}
	\begin{defn}\label{defn_vb_bg}
		Let $(E, \nabla^E, h^{E})$ be a Hermitian vector bundle with a fixed Hermitian connection over a manifold $(M, g^{TM})$ of bounded geometry.
		We say that $(E, \nabla^E, h^{E})$ is of bounded geometry if $R^E \in \ccal^{\infty}_b(M, \Lambda^2 T^*M \otimes \enmr{E})$.
		\par 
		If $(E, h^{E})$ is a Hermitian vector bundle over a \textit{complex manifold}, we say that it is of bounded geometry if $(E, \nabla^E, h^{E})$ is of bounded geometry for the Chern connection $\nabla^E$ on $(E, h^{E})$.
	\end{defn}
	\par 
	Finally, for further use, let us introduce the notations for the coordinates of local sections of $(E, \nabla^E)$ in special frames.
	Typically, later on $E$ would come from a Hermitian vector bundle $(E, h^E)$ and $\nabla^E$ will be the Chern connection.
	\par 
	Let us fix $y_0 \in H$ and an orthonormal frame $f_1, \ldots, f_r \in E_{x_0}$.
	Define the local orthonormal frame $\tilde{f}_1, \ldots, \tilde{f}_r$ of $E$ around $y_0$ by the parallel transport of $f_1, \ldots, f_r$ with respect to $\nabla^E$, done first along the path $\psi(t Z_H, 0)$, $t \in [0, 1]$, and then along the path $\psi(Z_H, tZ_N)$, $t \in [0, 1]$, $Z_H \in \real^{m}$, $Z_N \in \real^{n-m}$, $|Z_H| < r_H$, $|Z_N| < r_{\perp}$.
	When $Y = X$, we denote this frame by $\tilde{f}'_1, \ldots, \tilde{f}'_r$.
	\par \textbf{Notation.} For $g \in \ccal^{\infty}(M, F)$, by an abuse of notation, we write $g(\phi_{y_0}^M(Z)) \in \real^r$, $Z \in \real^{n}$, $|Z| \leq R$, for coordinates of $g$ in the frame $(\tilde{f}'_1{}^{M}, \ldots, \tilde{f}'_r{}^{M})$. 
	We identify $g(\phi_{y_0}^M(Z))$ with an element in $F_{y_0}$ using the frame $(f_1, \ldots, f_r)$.
	Similarly, we denote by $g(\psi_{y_0}(Z)) \in \real^r$ the coordinates in the frame $(\tilde{f}_1, \ldots, \tilde{f}_r)$ and identify them with an element from $F_{y_0}$.
	Similar notations are used for sections of $F^*$, $F \otimes L^p$, $(F \otimes L^p)^*$, $F \boxtimes F^*$, etc.
	\par 
	As it was explained in \cite{EichBoundG}, \cite{GrosSchnBound}, cf. \cite[Propositions 2.14, 2.15]{FinOTAs}, if $(E,  \nabla^E, h^E)$ has bounded geometry over $(M, H, g^{TM})$ and $f \in \ccal^{\infty}_b(M, E)$, then for any $y_0 \in H$ (resp. $x_0 \in M$), the function $Z \mapsto g(\psi_{y_0}^{M|H}(Z))$ (resp. $Z \mapsto g(\phi_{x_0}^M(Z))$) is in $\ccal^{\infty}_b(B_0^{\real^n}(r), F_{y_0})$ (resp. $\ccal^{\infty}_b(B_0^{\real^n}(r), F_{x_0})$) for some constant $r > 0$, independent of $x_0$ and $y_0$.
	
\subsection{Exponential bounds over manifolds of bounded geometry}\label{sect_bnd_geom_cf}
	The main goal of this section is to recall the results about the convergence of exponential integrals for triples of bounded geometry from \cite{FinOTAs}, \cite{FinToeplImm}.
	More precisely, fix a triple $(M, H, g^{TM})$ of bounded geometry.
	Let us recall the exponential bound from \cite[Corollary 3.3]{FinOTAs}, established using Bishop-Gromov inequality.
	\begin{prop}\label{prop_exp_bound_int}
		There are $c, C' > 0$, which depend only on $n$, $m$, $r_M$, $r_N$, $r_{\perp}$ and $\sup$-norm on $R^{TM}$, $R^{TH}$, $A$, such that for any $y_0 \in H$, $l \geq c$, we have
		\begin{equation}
			\int_{H} \exp \big(-l \dist_M(y_0, y) \big) dv_{g^{TH}}(y) < \frac{C'}{l^{m}}.
		\end{equation}
	\end{prop}
	
	Now, in addition to the triple $(M, H, g^{TM})$ of bounded geometry, we consider a Riemannian manifold $(K, g^{TK})$ with an embedding $\iota_1 : M \to K$, such that $\iota_1^* g^{TK} = g^{TM}$.
	We assume that the triple $(K, M, g^{TK})$ is of bounded geometry.
	Let $(E, h^E)$ be a Hermitian vector bundle over $M$ and $D: L^2(H, \iota^* E) \to L^2(M, E)$ be a fixed linear operator.
	Assume that there is $c > 0$ as in Proposition \ref{prop_exp_bound_int} and $C > 0$, such that for some $l \geq c$ and any $y \in H$, $x \in M$, the Schwartz kernel of $D$, evaluated with respect to $dv_{g^{TH}}$, satisfies the bound
	\begin{equation}\label{eq_d_bnd2}
		\big| D(x, y) \big|
		\leq 
		C
		l^m
		\exp \big(-l \dist_K(x, y) \big).
	\end{equation}
	By essentially relying on Heintze-Karcher estimate \cite[Corollary 3.3.1]{HeintzKarch} and Young's inequality for integral operators, we obtained in \cite[Proposition 2.12]{FinToeplImm} the following bound.
	\begin{prop}\label{prop_norm_bnd_distk_expbnd}
		There is $C' > 0$, which depends on the same data as $C'$ from Proposition \ref{prop_exp_bound_int} and the analogous data on $(K, M, g^{TK})$, such that 
		\begin{equation}
			\| D \| \leq \frac{C' C}{l^{\frac{n - m}{2}}}.
		\end{equation}
	\end{prop}
	\par 
	We will now recall a related result about the bound on the composition of operators with Schwartz kernels having exponential bounds.
	More precisely, we fix $q \in \nat$, $q \geq 2$, and operators $\mathcal{G}_t$, $\mathcal{A}_t^{1}, \ldots, \mathcal{A}_t^{q}$, $t \in [0, 1]$, acting on the sections of the trivial vector bundle $\comp^{r_0} \times \comp^n$ over $\comp^n$ by the convolutions with smooth kernels $\mathcal{G}_t(Z, Z')$, $\mathcal{A}_t^{1}(Z, Z'), \ldots, \mathcal{A}_t^{q}(Z, Z') \in \enmr{\comp^{r_0}}$ with respect to the volume form $dv_{\comp^n}$ on $\comp^n$.
	We assume that there are $c_0, q_1 > 0$, such that for any $l \in \nat$, there are $C > 0$, $Q_{h, 1} \geq 0$, $h = 1, \ldots, q$, such that for any $t \in [0, 1]$, $Z, Z' \in \real^{2n}$, $\alpha, \alpha' \in \nat^{2n}$, $|\alpha| + |\alpha'| \leq l$, we have
	\begin{align}
		& \nonumber
		\bigg| \frac{\partial^{|\alpha|+|\alpha'|}}{\partial Z^{\alpha} \partial Z'{}^{\alpha'}} \mathcal{A}_t^{h}(Z, Z') \bigg|
		 \leq 
		 C \Big(1 + |Z| + |Z'| \Big)^{Q_{h, 1} + q_1 l} \exp\Big(- c_0 \big( |Z_Y - Z'_Y| + |Z_N| + |Z'_N| \big) \Big),
		 \\ \label{eq_a_bnd_1loc}
		 &
		 \bigg| \frac{\partial^{|\alpha|+|\alpha'|}}{\partial Z^{\alpha} \partial Z'{}^{\alpha'}} \mathcal{G}_t(Z, Z') \bigg|
		 \leq 
		 C \Big(1 + |Z| + |Z'| \Big)^{Q_{1, 1} + q_1 l} \exp\Big(- c_0 |Z - Z'|  \Big).
	\end{align}
	\begin{lem}[{\cite[Lemma 3.5]{FinOTAs}}]\label{lem_bnd_prod_aloc}
		The operators $\mathcal{D}_t := \mathcal{A}_t^{1} \circ \cdots \circ \mathcal{A}_t^{q}$,  $\mathcal{D}'_t := \mathcal{G}_t \circ \mathcal{A}_t^{2} \circ \cdots \circ \mathcal{A}_t^{q}$ are well-defined and have smooth Schwartz kernels $\mathcal{D}_t(Z, Z')$, $\mathcal{D}'_t(Z, Z')$ with respect to $dv_{\comp^n}$. 
		Moreover, for any $l \in \nat$, there is $C > 0$, such that for any $t \in [0, 1]$, $Z, Z' \in \real^{2n}$, $\alpha, \alpha' \in \nat^{2n}$, $|\alpha| + |\alpha'| \leq l$, we have
		\begin{multline}\label{eq_bnd_prod_aloc}
			\bigg| \frac{\partial^{|\alpha|+|\alpha'|}}{\partial Z^{\alpha} \partial Z'{}^{\alpha'}} \mathcal{R}_t(Z, Z') \bigg|
			\leq C \Big(1 + |Z| + |Z'| \Big)^{Q_{1, 1} + \cdots +Q_{q, 1} + q_1 l} \cdot
			\\
			\cdot
			\exp\Big(- \frac{c_0}{8} \big( |Z_Y - Z'_Y| + |Z_N| + |Z'_N| \big) \Big),
		\end{multline}
		where $\mathcal{R}_t$ designates either $\mathcal{D}_t$ or $\mathcal{D}'_t$.
	\end{lem}

\section{Kernel calculus and asymptotic criteria of Toeplitz operators}\label{sect_asymp_toepl_type}

	The main goal of this section is to study the basic properties of Schwartz kernels of Toeplitz operators.
	More precisely, in Section \ref{sect_model_calc}, we consider the model situation, for which an explicit formula for the Schwartz kernels of Bergman projectors, the extension and restriction operators can be given. 
	We then study the composition rules for the operators with related kernels.
	Then, in Section \ref{sect_as_crit}, we recall a definition and an asymptotic characterization of Toeplitz operators with weak exponential decay.
	
	\subsection{Model operators on the Bargmann space}\label{sect_model_calc}
		In this section, we consider the model situation, for which an explicit formula for the Schwartz kernels of Bergman projectors, the extension and restriction operators can be given.
		We then use those explicit formulas to give a description for compositions of operators, the Schwartz kernels of which can be expressed using the above kernels.
		This section is motivated in many ways by the works of Ma-Marinescu \cite{MaMarToepl}, \cite{MaHol} and Dai-Liu-Ma \cite{DaiLiuMa}, cf. also \cite[\S 3.2]{FinOTAs}.
		\par 
		Endow $\comp^n$ with the standard Riemannian metric and consider a trivialized holomorphic line bundle $L_0$ on $\comp^n$.
		We endow $L_0$ with the Hermitian metric $h^{L_0}$, given by 
		\begin{equation}\label{eq_model_metrl}
			\| 1 \|_{h^{L_0}}(Z) = \exp \Big(- \frac{\pi}{2} |Z|^2 \Big),
		\end{equation}
		where $Z$ is the natural real coordinate on $\comp^n$, and $1$ is the trivializing section of $L_0$.
		An easy verification shows that (\ref{eq_model_metrl}) implies that (\ref{eq_gtx_def}) holds in our setting. 
		From \cite[Theorem 4.1.20]{MaHol}, we know that the functions 
		\begin{equation}\label{eq_spec_l_d_zero}
			\Big(
				\frac{\pi^{|\beta|}}{\beta!}
			\Big)^{\frac{1}{2}}
			z^{\beta} \exp( - \frac{\pi}{2} \sum_{i = 1}^{n} |z_i|^2),
		\end{equation}
		viewed as section of $L_0$ using the orthonormal trivialization by $1 \exp (\frac{\pi}{2} |Z|^2)$, form an orthonormal basis of $H^0_{(2)}(\comp^n, L_0)$, endowed with the induced $L^2$-metric.
		Clearly, this space coincides with the Bargmann space, cf. \cite{ZhuFock}.
		From (\ref{eq_spec_l_d_zero}), cf. \cite[(4.1.84)]{MaHol}, the Bergman kernel $\mathscr{P}_{n}$ of $\comp^n$ is given by
		\begin{equation}\label{eq_berg_k_expl}
			\mathscr{P}_n(Z, Z') = \exp \Big(
				-\frac{\pi}{2} \sum_{i = 1}^{n} \big( 
					|z_i|^2 + |z'_i|^2 - 2 z_i \overline{z}'_i
				\big)
			\Big), \quad \text{for } Z, Z' \in \comp^n.
		\end{equation}
		\par 
		From (\ref{eq_spec_l_d_zero}), we see easily, cf. \cite[(3.28), (3.29)]{FinOTAs}, that the Schwartz kernel of the \textit{orthogonal Bergman kernel of order} $k \in \nat$, $\mathscr{P}_{n, m}^{\perp, k}$, corresponding to the projection onto the subspace of $H^0_{(2)}(\comp^n, L_0 \otimes \mathcal{J}_{\comp^m}^k)$, orthogonal to $H^0_{(2)}(\comp^n, L_0 \otimes \mathcal{J}_{\comp^m}^{k + 1})$, is given by
		\begin{multline}\label{eq_pperp_defn_fun}
			\mathscr{P}_{n, m}^{\perp, k}(Z, Z')
			=
			 \exp \Big(
			 	-\frac{\pi}{2} \sum_{i = 1}^{m} \big( 
					|z_i|^2 + |z'_i|^2 - 2 z_i \overline{z}'_i
				\big)
				-
				\frac{\pi}{2}
				 \sum_{i = m+1}^{n} \big( |z_i|^2 + |z'_i|^2 \big)
			\Big)
			\cdot
			\\
			\cdot
			\pi^{k}
			\cdot
			\sum_{\substack{\beta \in \nat^{n - m} \\ |\beta| = k}}
			\frac{ z_N^{\beta} \cdot (\overline{z}'_N)^{\beta} }{\beta!}
			.
		\end{multline}
	\par 
	Let us now denote $V := \comp^{n - m}$ and consider the holomorphic and antiholomorphic parts $V^{1, 0}$, $V^{0, 1}$ of the complexification $V \otimes_{\real} \comp$.
	Remark now that $\nabla^{|\beta|} z_N^{\beta} \in ((V^{1, 0})^*)^{\otimes |\beta|}$ satisfies the following identity
	\begin{equation}
		\nabla^{|\beta|} z_N^{\beta} 
		\Big(
			\frac{\partial}{\partial z_{\sigma(1)}} \otimes \cdots \otimes \frac{\partial}{\partial z_{\sigma(|\beta|)}} 
		\Big)
		=
		\frac{\partial}{\partial z_{\sigma(1)}} \cdots \frac{\partial}{\partial z_{\sigma(|\beta|)}} z_N^{\beta} 
		=
		\begin{cases}
		\beta!, \quad &\text{if $\# \sigma^{-1}(i) = \beta_i$},
		\\
		0, \quad &\text{otherwise},
		\end{cases}
	\end{equation}
	where $\sigma : [1, N] \to [m + 1, n]$ is any map.
	From this, we conclude that $\nabla^{|\beta|} z_N^{\beta}$ lies in the image of the natural map ${\rm{Sym}}^{|\beta|} ((V^{1, 0})^*) \to ((V^{1, 0})^*)^{\otimes |\beta|}$, and under this map, from (\ref{eq_pairing_sym}) we have
	\begin{equation}\label{eq_nabla_sym_der_ident}
		\nabla^{|\beta|} z_N^{\beta}
		=
		|\beta|! \cdot d z_N^{\odot \beta}.
	\end{equation}
	From (\ref{eq_pairing_sym}) and (\ref{eq_spec_l_d_zero}), we see that the $L^2$-extension operator $\mathscr{E}_{n, m}^{k}$, extending each element from $H^0_{(2)}(\comp^m, {\rm{Sym}}^k (V^{1, 0})^* \otimes L_0)$ to an element from $H^0_{(2)}(\comp^n, L_0 \otimes \mathcal{J}_{\comp^m}^k)$ as in (\ref{eq_deriv_f_along_y}) with the minimal $L^2$-norm, is given in the the orthonormal trivialization $1 \exp (\frac{\pi}{2} |Z|^2)$ of $L_0$ by the multiplication by $\exp (- \frac{\pi}{2} \sum_{i = m+1}^{n} |z_i|^2)$.
	From this and (\ref{eq_berg_k_expl}), we see that the Schwartz kernel $\mathscr{E}_{n, m}^{k}(Z, Z'_Y)$ of $\mathscr{E}_{n, m}^{k}$ is given by
	\begin{multline}\label{eq_pperp_defn_fu2n}
		\mathscr{E}_{n, m}^{k}(Z, Z'_Y)
		=
		 \exp \Big(
			-\frac{\pi}{2} \sum_{i = 1}^{m} \big( 
				|z_i|^2 + |z'_i|^2 - 2 z_i \overline{z}'_i
			\big)
			-
			\frac{\pi}{2}
			 \sum_{i = m+1}^{n} |z_i|^2
			\Big)
		\cdot
		\\
		\cdot
		\sum_{\substack{\beta \in \nat^{n - m} \\ |\beta| = k}}
		\frac{1}{\beta!} \cdot
		 z_N^{\beta} \cdot \big( \frac{\partial}{\partial z_N} \big)^{\odot \beta},
	\end{multline}
	where we implicitly identified $({\rm{Sym}}^k (V^{1, 0})^*)^*$ and ${\rm{Sym}}^k (V^{1, 0})$ as in (\ref{eq_pairing_sym}).
	Remark that under our identification, the following identity holds
	\begin{equation}\label{eq_id_sym_ident}
		{\rm{Id}}_{{\rm{Sym}}^k (V^{1, 0})^*} 
		=
		\sum_{\substack{\beta \in \nat^{n - m} \\ |\beta| = k}}
		\frac{k!}{\beta!}
		\cdot
		(d z_N)^{\odot \beta} \cdot \big( \frac{\partial}{\partial z_N} \big)^{\odot \beta}.
	\end{equation}
	From (\ref{eq_berg_k_expl}), (\ref{eq_pperp_defn_fun}), (\ref{eq_pperp_defn_fu2n}) and (\ref{eq_id_sym_ident}), we can verify that the following natural relation holds
	\begin{equation}\label{eq_verif_res_ext}
		\res_{\comp^m} \circ \nabla^k \mathscr{E}_{n, m}^{k}
		=
		\mathscr{P}_n
		\cdot
		{\rm{Id}}_{{\rm{Sym}}^k (V^{1, 0})^*},
	\end{equation}
	where we identified $H^0_{(2)}(\comp^m, {\rm{Sym}}^k (V^{1, 0})^* \otimes L_0)$ with $H^0_{(2)}(\comp^m, L_0) \otimes {\rm{Sym}}^k (V^{1, 0})^*$.
	\par 
	For $k \in \nat$, recall that the logarithmic Bergman kernel corresponding to the pair $(\comp^n, k \cdot \comp^k)$, which we denote $\mathscr{P}_{n, m}^{\comp^n, k \comp^m}$, is the orthogonal projection from $L^2(\comp^n, L_0)$ onto $H^0_{(2)}(\comp^n, L_0 \otimes \mathcal{J}_{\comp^m}^k)$. It clearly satisfies
	\begin{equation}
		\mathscr{P}_{n, m}^{\comp^n, k \comp^m} 
		:= 
		\mathscr{P}_n
		-
		\sum_{l = 0}^{k - 1} \mathscr{P}_{n, m}^{\perp, l}. 
	\end{equation}
	From (\ref{eq_spec_l_d_zero}), we can write 
	\begin{multline}\label{eq_res_nm_kernel0001}
		\mathscr{P}_{n, m}^{\comp^n, k \comp^m}(Z, Z')
		=
		 \exp \Big(
		 	-\frac{\pi}{2} \sum_{i = 1}^{m} \big( 
				|z_i|^2 + |z'_i|^2 - 2 z_i \overline{z}'_i
			\big)
			-
			\frac{\pi}{2}
			 \sum_{i = m+1}^{n} \big( |z_i|^2 + |z'_i|^2 \big)
		\Big)
		\cdot
		\\
		\cdot
		\pi^{k}
		\cdot
		\sum_{\substack{\beta \in \nat^{n - m} \\ |\beta| \geq k}}
		\frac{ z_N^{\beta} \cdot (\overline{z}'_N)^{\beta} }{\beta!}
		.
	\end{multline}
	Hence, by (\ref{eq_nabla_sym_der_ident}), the Schwartz kernel of the operator $\mathscr{Res}_{n, m}^{k} := \res_{\comp^m} \circ \nabla^k \mathscr{P}_{n, m}^{\comp^n, k \comp^m}$ is given by 
	\begin{multline}\label{eq_res_nm_kernel}
		\mathscr{Res}_{n, m}^{k}(Z_Y, Z') = 
				\exp \Big(
					-\frac{\pi}{2} \sum_{i = 1}^{m} \big( 
						|z_i|^2 + |z'_i|^2 - 2 z_i \overline{z}'_i
					\big)
					-
					\frac{\pi}{2}
					\sum_{i = m+1}^{n} |z'_i|^2
				\Big)
				\cdot
				\\
				\cdot
				\pi^{k}
				\cdot
				\sum_{\substack{\beta \in \nat^{n - m} \\ |\beta| = k}}
				\frac{k!}{\beta!} \cdot
				 (d z_N)^{\odot \beta} \cdot (\overline{z}'_N)^{\beta}.
	\end{multline}
	\par 
	Let us consider $f \in H^0_{(2)}(\comp^n, L_0 \otimes \mathcal{J}_{\comp^m}^k)$ (resp. $g \in H^0_{(2)}(\comp^m, {\rm{Sym}}^k (V^{1, 0})^* \otimes L_0)$), given in the holomorphic trivialization of $L_0$ by $f = z_N^{\beta}$ (resp. $g = dz_N^{\odot \beta'}$), where $\beta, \beta' \in \nat^{2(n - m)}$, $|\beta|, |\beta'| = k$.
	By (\ref{eq_nabla_sym_der_ident}), we have
	\begin{equation}
		\mathscr{Res}_{n, m}^{k} f = k! \cdot d z_N^{\odot \beta}, 
		\qquad
		\mathscr{E}_{n, m}^k g = \frac{1}{k!} \cdot z_N^{\beta'}.
	\end{equation}
	From the discussion after (\ref{eq_sym_emb_tens}) and (\ref{eq_spec_l_d_zero}), we see that
	\begin{equation}
		\scal{\mathscr{Res}_{n, m}^k f}{g}_{L^2(Y)}
		=
		2^k \cdot \beta!
		=
		(2 \pi)^k \cdot k! \cdot
		\scal{f }{\mathscr{E}_{n, m}^k g}_{L^2(X)},
	\end{equation}
	which implies the following identity
	\begin{equation}\label{eq_res_enmdual}
		(\mathscr{Res}_{n, m}^{k})^*
		=
		(2 \pi)^k \cdot k! \cdot
		\mathscr{E}_{n, m}^k.
	\end{equation}
	Similarly, we deduce the following equality
	\begin{equation}\label{eq_res_enmdual}
		\mathscr{P}_{n, m}^{\perp, k} 
		=
		\mathscr{E}_{n, m}^k
		\circ
		\mathscr{Res}_{n, m}^{k}.
	\end{equation}
	Later we will see that the analogues of (\ref{eq_res_enmdual}) and (\ref{eq_res_enmdual}) hold in the approximate sense for any pair $(X, Y)$ instead of the model case $(\comp^n, \comp^m)$.
	\par 
	Now, a lot of calculations in this article will have something to do with compositions of operators having Schwartz kernels, given by the product of polynomials with the above kernels.
	For that reason, the following lemma will be of utmost importance in what follows.
	\begin{lem}\label{lem_comp_poly}
			For any polynomials $A_1(Z, Z'), A_2(Z, Z')$, $Z, Z' \in \real^{2n}$, there is a polynomial $A_3 := \mathcal{K}_{n, m}[A_1, A_2]$, the coefficients of which are polynomials of the coefficients of $A_1, A_2$, such that
			\begin{equation}\label{eq_lem_comp_poly_1}
				(A_1 \cdot \mathscr{P}_{n, m}^{\perp, 0}) \circ (A_2 \cdot \mathscr{P}_{n, m}^{\perp, 0})
				=
				A_3 \cdot \mathscr{P}_{n, m}^{\perp, 0}.
			\end{equation}
			Moreover, $\deg A_3 \leq \deg A_1 + \deg A_2$. Also, if both polynomials $A_1$, $A_2$ are even or odd (resp. one is even, another is odd), then the polynomial $A_3$ is even (resp. odd).
			\par 
			Also, for any polynomials $A(Z, Z'_Y)$, $B(Z_Y, Z'_Y)$, $Z \in \real^{2n}$, $Z_Y, Z'_Y \in \real^{2m}$, there is a polynomial $A''_3 := \mathcal{K}_{n, m}^{EP}[A, B]$ with the same properties as $A_3$, such that 
			\begin{equation}\label{eq_comp_ext_pmn22}
				(A \cdot \mathscr{E}_{n, m}^{0}) \circ (B \cdot\mathscr{P}_{m})
				=
				A''_3 \cdot \mathscr{E}_{n, m}.
			\end{equation}
			\par 
			Finally, for any polynomials $A(Z, Z'_Y)$, $C(Z_Y, Z')$, $Z, Z' \in \real^{2n}$, $Z_Y, Z'_Y \in \real^{2m}$, there is a polynomial $A'''_3 := \mathcal{K}_{n, m}^{ER}[A, C]$ with the same properties as $A_3$, such that 
			\begin{equation}\label{eq_comp_ext_pmn222}
				(A \cdot \mathscr{E}_{n, m}^{0}) \circ (C \cdot\mathscr{R}_{n, m}^{0})
				=
				A'''_3 \cdot \mathscr{P}_{n, m}^{\perp, 0}.
			\end{equation}
		\end{lem}
		\begin{rem}
			The statement (\ref{eq_lem_comp_poly_1}) for $n = m$ is due to Ma-Marinescu \cite[Lemma 7.1.1, (7.1.6)]{MaHol}.
		\end{rem}
		\begin{proof}
			Statements (\ref{eq_lem_comp_poly_1}) and (\ref{eq_comp_ext_pmn222}) were proved in \cite[Lemma 3.1]{FinToeplImm}.
			Let us recall the relation between $\mathcal{K}_{n, m}^{EP}$ and $\mathcal{K}_{n, m}$, established in \cite[(3.20)]{FinToeplImm}.
			\par 
			We represent $A(Z, Z'_Y) := \sum Z_N^{\alpha} \cdot A^{\alpha}(Z_Y, Z'_Y)$. 
			Then, from (\ref{eq_berg_k_expl}) and (\ref{eq_pperp_defn_fun}), the following equation holds
			\begin{equation}
				(A^{\alpha} \cdot \mathscr{E}_{n, m}^0) \circ (B \cdot \mathscr{P}_{m})
				=
				 \exp \Big(- \frac{\pi}{2} |Z_N|^2 \Big)
				 \cdot
				(A^{\alpha} \cdot \mathscr{P}_{m}) \circ (B \cdot \mathscr{P}_{m}).
			\end{equation}
			By this and (\ref{eq_lem_comp_poly_1}), we clearly have (\ref{eq_comp_ext_pmn22}) for 
			\begin{equation}\label{eq_kep_formula}
				\mathcal{K}_{n, m}^{EP}[A, D]
				=
				\sum_{\alpha} Z_N^{\alpha} \cdot \mathcal{K}_{m, m}[A^{\alpha}, D].
			\end{equation}
			\par 
			To establish (\ref{eq_comp_ext_pmn222}), we decompose polynomials $A(Z, Z'_Y)$, $C(Z_Y, Z')$ as follows
			\begin{equation}\label{eq_compa_3}
				A(Z, Z'_Y)
				=
				\sum_{\alpha} Z_N^{\alpha} \cdot A^{\alpha}(Z_Y, Z'_Y), 
				\qquad
				C(Z_Y, Z')
				=
				\sum_{\alpha'} C^{\alpha'}(Z_Y, Z'_Y) Z'_N{}^{\alpha'},
			\end{equation}
			where $\alpha, \alpha' \in \nat^{2(n - m)}$ verify $|\alpha| \leq \deg A, |\alpha'| \leq \deg C$.
			\par 
			Now, we note that an easy verification, based on (\ref{eq_berg_k_expl}), (\ref{eq_pperp_defn_fun}) and (\ref{eq_res_nm_kernel}), shows that 
			\begin{multline}\label{eq_compa_4}
				\Big( (A^{\alpha} \cdot \mathscr{E}_{n, m}^0) \circ (C^{\alpha'} \cdot \mathscr{Res}_{n, m}^0) \Big) (Z, Z')
				=
				\exp \Big(- \frac{\pi}{2} \big( |Z_N|^2 + |Z'_N|^2 \big) \Big)
				 \cdot
				 \\
				 \cdot
				\Big( (A^{\alpha} \cdot \mathscr{P}_{m}) \circ (C^{\alpha'} \cdot \mathscr{P}_{m}) \Big) (Z_Y, Z'_Y).
			\end{multline}
			From (\ref{eq_compa_3}) and (\ref{eq_compa_4}), we see that (\ref{eq_comp_ext_pmn222}) holds for
			\begin{equation}\label{eq_ker_form}
				\mathcal{K}_{n, m}^{ER}[A, C]
				=
				\sum_{\alpha, \alpha'} Z_N^{\alpha} \cdot Z'_N{}^{\alpha'} \cdot \mathcal{K}_{m, m}[A^{\alpha}, C^{\alpha'}].
			\end{equation}
			This clearly finishes the proof, as the statements about the degree, the coefficients and the parity now follow from (\ref{eq_ker_form}).
		\end{proof}
		\begin{sloppypar}
			From the above, we see that to compute the polynomials from Lemma \ref{lem_comp_poly}, it suffices to give an algorithm for the calculation of $\mathcal{K}_{n, m}$.
			Below, we explain how to do this. 
			Directly from the definitions, we see that $\mathcal{K}_{n, m}[1 \cdot P(Z'), A] = \mathcal{K}_{n, m}[1, P(Z) \cdot A]$ for any polynomial $A$.
			Also, we trivially have $\mathcal{K}_{n, m}[P(Z) \cdot A(Z, Z'), A'(Z, Z')] = P(Z) \mathcal{K}_{n, m}[A(Z, Z'), A'(Z, Z')]$ for any polynomials $P, A, A'$.
			Hence, it is enough to give an algorithm for the calculation of $\mathcal{K}_{n, m}$ where the first argument is given by $1$.
			For this, remark that for any $i = 1, \ldots, n$, $a, b \in \nat$, we have
			\begin{equation}\label{eq_k_calc_2}
				\mathcal{K}_{n, m}[1, P_i(Z) z_i^a \overline{z}_i^b]
				=
				\mathcal{K}_{n, m}[1, P_i(Z)]
				\cdot
				\mathcal{K}_{n, m}[1, z_i^a \overline{z}_i^b],
			\end{equation}
			where the polynomial $P_i(Z)$ doesn't depend on $z_i$ and $\overline{z}_i$.
			Hence, to understand $\mathcal{K}_{n, m}$, it suffices to know how to calculate it for polynomials $z_i^a \overline{z}_i^b$.
			Let us recall the general formulas from \cite[(3.28) and (3.30)]{FinToeplImm}.
			For $i \leq m$, we have
			\begin{equation}\label{eq_kmn_poly}
				\mathcal{K}_{n, m}[1, z_i^a \overline{z}_i^b]
				=
				\sum_{l + k = b}
				\frac{1}{\pi^{k}}
				 \frac{a! b!}{(a - k)! l! k!}
				 z_i^{a - k} \overline{z}'_i{}^{l}
				.
			\end{equation}	 
			For $m+1 \leq i \leq n$, we have
			\begin{equation}\label{eq_kmn_poly23}
				\mathcal{K}_{n, m}[1, z_i^a \overline{z}_i^b]
				=
				\delta_{ab}
				\frac{a!}{\pi^a}
				.
			\end{equation}
		\end{sloppypar}
	
\subsection{Toeplitz operators with weak exponential decay and their properties}\label{sect_as_crit}
	The main goal of this section is to recall the definition of Toeplitz operators with weak exponential decay and to recall the asymptotic characterization of them in terms of their Schwartz kernels.
	\begin{sloppypar}
	More precisely, for a section $f \in \ccal^{\infty}_{b}(X, \enmr{F})$, we associate a sequence of linear operators $T_{f, p}^{X} \in \enmr{L^2(X, L^p \otimes F)}$, $p \in \nat$, called \textit{Berezin-Toeplitz operator}, by
	\begin{equation}
		T_{f, p}^{X}(g) := B_p^X (f \cdot B_p^X g).
	\end{equation}
	\end{sloppypar}
	We fix some Riemannian manifold $(Z, g^{TZ})$ and an embedding $\iota': X \to Z$, such that $(\iota')^* g^{TZ} = g^{TX}$, and such that the triple $(Z, X, g^{TZ})$ is of bounded geometry.
	\begin{defn}\label{defn_ttype}
		A sequence of linear operators $T_p^X \in \enmr{L^2(X, L^p \otimes F)}$, $p \in \nat$, verifying $B_p^X \circ T_p^X \circ B_p^X = T_p^X$, is called a \textit{Toeplitz operator with weak exponential decay with respect to $Z$} if there is a sequence $f_i \in \ccal^{\infty}_{b}(X, \enmr{F})$ and  $c > 0$, $p_1 \in \nat^*$, such that for any $k, l \in \nat$, there is $C > 0$, such that for any $p \geq p_1$, the Schwartz kernel, evaluated with respect to $dv_X$, for $x_1, x_2 \in X$, satisfies
		\begin{equation}\label{eq_toepl_off_diag}
			\Big|  
				T_p^X (x_1, x_2) 
				- 
				\sum_{r = 0}^{k}
				p^{-r}
				T_{f_r, p}^X(x_1, x_2)  
			\Big|_{\ccal^l} 
			\leq 
			C p^{n - k + \frac{l}{2}} 
			\cdot 
			\exp \big(- c \sqrt{p} \cdot \dist_Z(x_1, x_2) \big),
		\end{equation}
		where the pointwise $\ccal^{l}$-norm at a point $(x_1, x_2) \in X \times X$ is interpreted as in (\ref{eq_ext_log_bk}).
		The sections $f_i$ will later be denoted by $[T_p^X]_i$.
	\end{defn}
	\begin{rem}\label{rem_defn_tpl}
		a)
		From Proposition \ref{prop_norm_bnd_distk_expbnd}, we see that (\ref{eq_toepl_off_diag}) implies that for any $k \in \nat$, there is $C > 0$, such that for any $p \geq p_1$, we have $\|  
				T_p^X
				- 
				\sum_{r = 0}^{k}
				p^{-r}
				T_{f_r, p}^{X}
			\|
			\leq 
			C p^{- k}$.
		In particular, for compact $X$, the sequence of operators $T_p^X$, $p \in \nat$, forms a Toeplitz operator in the sense of Ma-Marinescu \cite[\S 7]{MaHol}.
		\par 
		b) 
		In \cite[Corollary 3.13]{FinToeplImm}, we showed that the sections $f_i$, $i \in \nat$, verifying (\ref{eq_toepl_off_diag}), are uniquely defined. Hence, the notation $[\cdot]_i$, $i \in \nat$, from Definition \ref{defn_ttype} is well-defined. 
	\end{rem}
	For the theorem below, we use the notational convention introduced after Definition \ref{defn_vb_bg}.
	Let us fix some further notation.
	Recall that geodesic coordinates were defined in (\ref{eq_phi_defn}).
	We fix $x_0 \in X$.
	Define the function $\kappa_{\phi, x_0}^{X} : B_0^{\real^{2n}}(r_X) \to \real$, by
	\begin{equation}\label{eq_defn_kappaxy10}
		((\phi_{x_0}^X)^* dv_X) (Z)
		=
		\kappa_{\phi, x_0}^{X}
		d Z_1 \wedge \cdots \wedge d Z_{2n}.
	\end{equation}
	\begin{thm}\label{thm_ma_mar_crit_exp_dec}
		A family of operators $T_p^X \in \enmr{ L^2(X, L^p \otimes F) }$, $p \in \nat$, forms a Toeplitz operator with weak exponential decay with respect to $Z$ if and only if the following conditions hold
		\begin{enumerate}
			\item For any $p \in \nat$,  $T_p^X = B_p^X \circ T_p^X \circ B_p^X$.
			\item There is $p_1 \in \nat$, such that for any $l \in \nat$, there is $C > 0$, such that for any $p \geq p_1$, the Schwartz kernel $T_p^X(x_1, x_2)$; $x_1, x_2 \in Y$, of $T_p^X$, evaluated with respect to $dv_X$, satisfies
			\begin{equation}\label{eq_exp_est_ass_1}
				\Big|  
					T_p^X (x_1, x_2) 
				\Big|_{\ccal^l} 
				\leq 
				C p^{n + \frac{l}{2}} 
				\cdot 
				\exp \big(- c \sqrt{p} \cdot \dist_Z(x_1, x_2) \big).
			\end{equation}
			\item
			For any $x_0 \in X$, $r \in \nat$, there are $I_r^X(Z, Z') \in \enmr{F_{x_0}}$ polynomials in $Z, Z' \in \real^{2n}$ of the same parity as $r$, such that the coefficients of $I_r^X$ lie in $\ccal^{\infty}_{b}(X, \enmr{F})$, and for $F_r := I_r^X \cdot \mathscr{P}_{n}$, the following holds.
			There are $\epsilon, c > 0$, $p_1 \in \nat^*$, such that for any $k, l, l' \in \nat$, there are $C, Q > 0$, such that for any $x_0 \in X$, $p \geq p_1$, $Z, Z' \in \real^{2n}$, $|Z|, |Z'| \leq \epsilon$, $\alpha, \alpha' \in \nat^{2n}$, $|\alpha|+|\alpha'| \leq l$, the following bound holds
			\begin{multline}\label{eq_tpy_defn_exp_tay12as}
				\bigg| 
					\frac{\partial^{|\alpha|+|\alpha'|}}{\partial Z^{\alpha} \partial Z'{}^{\alpha'}}
					\bigg(
						\frac{1}{p^n} T_p^X \big(\phi_{x_0}^{X}(Z), \phi_{x_0}^{X}(Z') \big)
						\\
						-
						\sum_{r = 0}^{k}
						p^{-\frac{r}{2}}						
						F_r(\sqrt{p} Z, \sqrt{p} Z') 
						\kappa_{\phi, x_0}^{X}(Z)^{-\frac{1}{2}}
						\kappa_{\phi, x_0}^{X}(Z')^{-\frac{1}{2}}
					\bigg)
				\bigg|_{\ccal^{l'}}
				\\
				\leq
				C p^{-\frac{k + 1 - l}{2}}
				\Big(1 + \sqrt{p}|Z| + \sqrt{p} |Z'| \Big)^{Q} \exp(- c \sqrt{p} |Z - Z'|),
			\end{multline}
			where the $\ccal^{l'}$-norm is taken with respect to $x_0$.
		\end{enumerate}		
		Moreover, for any sequence of operators $T_p^X$, verifying the above assumptions, the polynomial $I_0^X$ is constant and it is related to the expansion from Definition \ref{defn_ttype} by the identity $I_0^X(0, 0) = [T_p^X]_0$.
	\end{thm}
	\begin{proof}
		For Toeplitz operators in the sense of Ma-Marinescu \cite[\S 7]{MaHol}, cf. Remark \ref{rem_defn_tpl}.a), the analogous result was established by Ma-Marinescu in \cite[Theorem 4.9]{MaMar08a}.
		The proof for Toeplitz operators with weak exponential decay, analogous to the one of Ma-Marinescu, was done in \cite[Theorem 3.18]{FinToeplImm}.
	\end{proof}
	\begin{prop}\label{prop_parity_toepl}
		Assume that a family of linear operators $T_p^X \in \enmr{ L^2(X, L^p \otimes F) }$, $p \in \nat$ satisfies all the three assumptions of Theorem \ref{thm_ma_mar_crit_exp_dec}, with the only exception that the parity of $I_r^X$ is equal to the parity of $r + 1$ instead of $r$.
		Then $I_0^X = 0$, $I_1^X(Z, Z')$ is a constant polynomial, $\sqrt{p} T_p^X$ forms a Toeplitz operator with weak exponential decay with respect to $Z$, and we have $[\sqrt{p} T_p^X]_0 = I_1^X(0, 0)$.
	\end{prop}
	\begin{proof}
		In the proof of Ma-Marinescu of \cite[Theorem 4.9]{MaMar08a}, authors established that regardless of the parity assumption on $I_r^X$, the polynomial $I_0^X(Z, Z')$ is always constant.
		This was proved in the compact setting, but the proof remains verbatim for manifolds of bounded geometry, cf. \cite[proof of Theorem 3.19]{FinToeplImm}.
		Now, since the parity of $I_0^X$ is odd by our assumption, the above result implies $I_0^X = 0$.
		Hence all the assumptions of Theorem \ref{thm_ma_mar_crit_exp_dec} are satisfied for the sequence of operators $\sqrt{p} T_p^X$, $p \in \nat^*$, which implies our statement.
	\end{proof}
	\par 
	Now, for technical reasons we will need to consider sequences of operators $T_p^X : L^2(X, L^p \otimes F_1) \to L^2(X, L^p \otimes F_2)$, $p \in \nat$, where $(F_i, h^F_i)$, $i = 1, 2$, are Hermitian vector bundles of bounded geometry over $X$.
	For such sequences of operators, we have a notion of $(F_1, F_2)$-Toeplitz operators with weak exponential decay, analogous to Definition \ref{defn_ttype}. 
	The only difference between this definition and the one for $(F_1, h^F_1) = (F_2, h^F_2)$ is that the Berezin-Toeplitz operators are now associated to $f \in \ccal^{\infty}_{b}(X, {\rm{Hom}}(F_1, F_2))$ as follows $T_{f, p}^{X}(g) := B_{2, p}^X (f \cdot B_{1, p}^X g)$, where $B_{i, p}^X$, $i = 1, 2$, are the Bergman kernels associated to $L^p \otimes F_i$.
	We similarly use the notation $[T_p^X]_i$, $i \in \nat$, to designate elements of $\ccal^{\infty}_{b}(X, {\rm{Hom}}(F_1, F_2))$, corresponding to the asymptotic expansion of $T_p^X$.
	\begin{prop}\label{prop_ma_mar_crit_exp_dec}
		A family of linear operators $T_p^X : L^2(X, L^p \otimes F_1) \to L^2(X, L^p \otimes F_2)$, $p \in \nat$, forms a $(F_1, F_2)$-Toeplitz operator with weak exponential decay with respect to $Z$ if and only if the following conditions hold
		\begin{enumerate}
			\item For any $p \in \nat$,  $T_p^X = B_{2, p}^X \circ T_p^X \circ B_{1, p}^X$.
			\item Exponential bound analogous to (\ref{eq_exp_est_ass_1}) holds.
			\item
			For any $x_0 \in X$, $r \in \nat$, there are polynomials $I_r^X(Z, Z') \in {\rm{Hom}}(F_{1, x_0}, F_{2, x_0})$, satisfying the same assumptions as in (\ref{eq_tpy_defn_exp_tay12as}), and for which an analogue of the expansion (\ref{eq_tpy_defn_exp_tay12as}) holds.
		\end{enumerate}		
		Moreover, for any sequence of operators $T_p^X$, verifying the above assumptions, the polynomial $I_0^X(Z, Z')$ is constant and we have $I_0^X(0, 0) = [T_p^X]_0$.
	\end{prop}
	\begin{proof}
		Consider a sequence of operators $G_p^X \in \enmr{ L^2(X, L^p \otimes (F_1 \oplus F_2)) }$, $p \in \nat$, which in a matrix form associated to the decomposition $L^2(X, L^p \otimes (F_1 \oplus F_2)) = L^2(X, L^p \otimes F_1) \oplus L^2(X, L^p \otimes F_2)$ corresponds to
		\begin{equation}
			G_p^X
			=
			\Big(
			\begin{matrix}
				0 & T_p^X\\
				0 & 0
			\end{matrix}
			\Big).
		\end{equation}
		An easy verification shows that the assumptions of Theorem \ref{thm_ma_mar_crit_exp_dec} are satisfied for $G_p^X$ if and only if the corresponding assumptions from Proposition  \ref{prop_ma_mar_crit_exp_dec} are satisfied for $T_p^X$. We finish the proof by Theorem \ref{thm_ma_mar_crit_exp_dec}.
	\end{proof}
	\begin{rem}\label{rem_toepl_f1f2_par}
		From the above proof and Proposition \ref{prop_parity_toepl}, we see that a proposition, analogous to Proposition \ref{prop_parity_toepl} holds for $(F_1, F_2)$-Toeplitz operator with weak exponential decay.
	\end{rem}
	Later on, we will need a precise bound on the degrees of the polynomials from (\ref{eq_tpy_defn_exp_tay12as}).
	For this, we fix any $f \in \ccal^{\infty}_{b}(X, \enmr{F})$, and denote by $I_{f, r}^X$ the polynomials associated to $T_{f, p}^X$ as in (\ref{eq_tpy_defn_exp_tay12as}).
	\begin{prop}\label{prop_bound_deg_q}
		For any $f \in \ccal^{\infty}_{b}(X, \enmr{F})$, we have $\deg I_{f, r}^X \leq 3 r$.
	\end{prop}
	This result is definitely well-known to experts, but we were not able to find a precise reference for it anywhere in the literature.
	Due to this and to the fact that the preliminary statement, used in the proof of Proposition \ref{prop_bound_deg_q}, will later be used in this article several times, we will explain the proof below.
	To state the result in the generality we will need later on, we fix some further notation.
	Recall that geodesic coordinates were defined in (\ref{eq_phi_defn}).
	We assume as in Introduction that the triple $(X, Y, g^{TX})$ has bounded geometry.
	Similarly to (\ref{eq_defn_kappaxy10}), for $y_0 \in Y$, define the function $\kappa_{\phi, y_0}^Y : B_0^{\real^{2m}}(r_Y) \to \real$, by
	\begin{equation}\label{eq_defn_kappaxy1}
		((\phi_{y_0}^Y)^* dv_Y) (Z_Y)
		=
		\kappa_{\phi, y_0}^Y
		d Z_1 \wedge \cdots \wedge d Z_{2m}.
	\end{equation}
	Recall that Fermi coordinates were defined in (\ref{eq_defn_fermi}).
	Define the function $\kappa_{\psi, y_0}^{X|Y} : B_0^{\real^{2m}}(r_Y) \times B_0^{\real^{2(n - m)}}(r_{\perp}) \to \real$ by
	\begin{equation}\label{eq_defn_kappaxy2}
		((\psi_{y_0}^{X|Y})^* dv_X) (Z)
		=
		\kappa_{\psi, y_0}^{X|Y}
		d Z_1 \wedge \cdots \wedge d Z_{2n}.
	\end{equation}
	Recall that the function $\kappa_N$ was defined in (\ref{eq_kappan}).
	Clearly, for $Z = (Z_Y, Z_N) \in \real^{2n}$, $Z_Y \in \real^{2m}$, we have the following relation between different $\kappa$-functions
	\begin{equation}\label{eq_kappa_relation}
		\kappa_{\psi, y_0}^{X|Y}(Z)
		=
		\kappa_N(\psi_{y_0}(Z))
		\cdot
		\kappa_{\phi, y_0}^Y(Z_Y).
	\end{equation}
	Also, under assumptions (\ref{eq_comp_vol_omeg}), we have $\kappa_{\psi, y_0}^{X|Y}(0) = \kappa_{\phi, y_0}^Y(0) = 1$.
	\par 
	\begin{thm}\label{thm_berg_off_diag}
		For any $r \in \nat$, $y_0 \in Y$, there are $J_r^{X|Y}(Z, Z') \in \enmr{F_{y_0}}$ polynomials in $Z, Z' \in \real^{2n}$, with the same parity as $r$ and $\deg J_r^{X|Y} \leq 3r$, 
		whose coefficients are polynomials in $\omega$, $R^{TX}$, $A$, $R^F$, $(dv_X / dv_{g^{TX}})^{\pm \frac{1}{2n}}$,  $(dv_Y / dv_{g^{TY}})^{\pm \frac{1}{2n}}$,  and their derivatives of order $\leq 2r$, all evaluated at $y_0$, such that for the functions $F_r^{X|Y} := J_r^{X|Y} \cdot \mathscr{P}_{n}$ over $\real^{2n} \times \real^{2n}$, the following holds. 
		There are $\epsilon, c > 0$, $p_1 \in \nat^*$, such that for any $k, l, l' \in \nat$, there exists $C > 0$, such that for any $y_0 \in Y$, $p \geq p_1$, $Z, Z' \in \real^{2n}$, $|Z|, |Z'| \leq \epsilon$, $\alpha, \alpha' \in \nat^{2n}$, $|\alpha|+|\alpha'| \leq l$, $Q^1_{k, l, l'} := 3 (n + k + l' + 2) + l$:
			\begin{multline}\label{eq_berg_off_diag}
				\bigg| 
					\frac{\partial^{|\alpha|+|\alpha'|}}{\partial Z^{\alpha} \partial Z'{}^{\alpha'}}
					\bigg(
						\frac{1}{p^n} B_p^X\big(\psi_{y_0}(Z), \psi_{y_0}(Z') \big)
						\\
						-
						\sum_{r = 0}^{k}
						p^{-\frac{r}{2}}						
						F_r^{X|Y}(\sqrt{p} Z, \sqrt{p} Z') 
						\kappa_{\psi}^{X|Y}(Z)^{-\frac{1}{2}}
						\kappa_{\psi}^{X|Y}(Z')^{-\frac{1}{2}}
					\bigg)
				\bigg|_{\ccal^{l'}}
				\\
				\leq
				C p^{-\frac{k + 1 - l}{2}}
				\Big(1 + \sqrt{p}|Z| + \sqrt{p} |Z'| \Big)^{Q^1_{k, l, l'}}
				\exp\Big(- c \sqrt{p} |Z - Z'| \Big),
			\end{multline}
			where the $\ccal^{l'}$-norm is taken with respect to $y_0$.
			Also, the following identity holds
			\begin{equation}\label{eq_jo_expl_form}
				J_0^{X|Y}(Z, Z') = {\rm{Id}}_{F_{y_0}}.
			\end{equation}
			Moreover, under the assumption (\ref{eq_comp_vol_omeg}), we have
			\begin{multline}\label{eq_j1_expl_form}
				J_1^{X|Y}(Z, Z') = {\rm{Id}}_{F_{y_0}} \cdot \pi 
				\Big(
			 	g \big(z_N, A(\overline{z}_Y - \overline{z}'_Y) (\overline{z}_Y - \overline{z}'_Y) \big)
			 	\\
			 	+
			 	g \big(\overline{z}'_N, A(z_Y - z'_Y) (z_Y - z'_Y) \big)			 
			 	\Big).
			\end{multline}
	\end{thm}
	\begin{proof}
		For $X = Y$, the result is due to Dai-Liu-Ma \cite{DaiLiuMa} and the calculation of $J_1^{X|X}$ is due to Ma-Marinescu \cite[Remark 4.1.26]{MaHol}. Related results were previously obtained by Tian \cite{TianBerg}, Bouche \cite{Bouche}, Zelditch \cite{ZeldBerg}, Catlin \cite{Caltin}, Lu \cite{LuBergman}, Wang \cite{WangBergmKern}.
		The proof of the general case is done in \cite[Theorem 5.5]{FinOTAs} by relying on the result of \cite{DaiLiuMa} and some local calculations.
	\end{proof}	
	\begin{proof}[Proof of Proposition \ref{prop_bound_deg_q}]
		In \cite[Lemma 7.2.4]{MaHol}, Ma-Marinescu proved that for compact manifolds, the following identity holds
		\begin{equation}\label{eq_jrf_expr}
			I_{f, r}^{X}
			:=
			\sum_{a + b + |\alpha| = r}
			\sum_{\alpha \in \nat^{2m}}
			\mathcal{K}_{m, m} \Big[ J_a^{X|X}, \frac{\partial^{\alpha} f(\phi_{x_0}^X(Z))}{\partial Z^{\alpha}}(0) \cdot \frac{Z^{\alpha}}{\alpha!} \cdot J_b^{X|X} \Big],
		\end{equation}
		where $\mathcal{K}_{m, m}$ was defined in Lemma \ref{lem_comp_poly}.
		As it was explained in \cite[(3.51)]{FinToeplImm}, since by the result of Ma-Marinescu \cite{MaMarOffDiag}, cf. Theorem \ref{thm_bk_off_diag}, the Bergman kernel decays exponentially away from the diagonal, the same proof holds for manifolds of bounded geometry.
		The result now follows from (\ref{eq_kmn_poly}), the bound on the degrees of $J_r^{X|X}$ from Theorem \ref{thm_berg_off_diag} and (\ref{eq_jrf_expr}).
	\end{proof}
	The final result about Toeplitz operators we will need is given below.
	\begin{lem}\label{lem_norm_toepl}
		For non-zero $f \in \ccal^{\infty}_{b}(X, \enmr{F})$, as $p \to \infty$, the following asymptotics holds
		\begin{equation}\label{eq_norm_toepl_calc}
			\| T_{f, p}^{X} \| \sim \sup_{x \in X} \| f(x) \|.
		\end{equation}
	\end{lem}
	\begin{proof}
		For compact manifolds $X$, this result is due to Bordemann-Meinrenken-Schlichenmaier \cite[Theorem 4.1]{BordMeinSchl} (for $(F, h^F)$ trivial) and Ma-Marinescu \cite[Theorem 3.19, (3.91)]{MaMarToepl} (for any $(F, h^F)$).
		Essentially the same proof works in our more general situation.
		The details are given in the end of the proof of \cite[Theorem 1.1]{FinOTAs}.
	\end{proof}

\section{Multiplicative defect and optimal extensions}
	The main goal of this section is to make a comparison between the dual of the restriction operator and the extension operator.
	For this, we prove the existence of a sequence of operators, which we call the multiplicative defect.
	It will play a central role in most of the proofs of this article, as it relates the asymptotic expansions of the Schwartz kernels of the Bergman projector, which was studied previously by other authors in details, and the extension operator, which is the main object of this paper.
	\par 
	More precisely, for $k \in \nat$, let us denote by $B_p^{\perp, k}$ the orthogonal projection from $L^2(X, L^p \otimes F)$ to the orthogonal complement of $H^0_{(2)}(X, L^p \otimes F \otimes \mathcal{J}_Y^{k + 1})$ in $H^0_{(2)}(X, L^p \otimes F \otimes \mathcal{J}_Y^k)$.
	The operators $B_p^{\perp, k}$ will later be called the \textit{orthogonal Bergman kernel of order} $k$.
	Clearly, the logarithmic Bergman kernel of order $k$, introduced before Theorem \ref{thm_log_bk} relates to orthogonal Bergman kernels as follows
	\begin{equation}
		B_p^{X, k Y}
		=
		B_p^X
		-
		\sum_{l = 0}^{k - 1}
		B_p^{\perp, \l}.
	\end{equation}
	The main result of this section is as follows.
	\begin{sloppypar}
	\begin{thm}\label{thm_ex_ap}
		Assume that $(X, Y, g^{TX})$ is of bounded geometry.
		Then for any $k \in \nat$, there is $p_1 \in \nat^*$, such that for any $p \geq p_1$, there is a unique operator $A_{k, p} \in \enmr{H^{0}_{(2)}(Y, {\rm{Sym}}^k (N^{1, 0})^* \otimes \iota^* ( L^p \otimes F))}$, verifying
		\begin{equation}\label{eq_resp_ap_lem}
			(\res_Y \circ \nabla^k B_p^{X, k Y})^{*} = \ext_{k, p} \circ A_{k, p}.
		\end{equation}
	\end{thm}
	\end{sloppypar}
	\begin{rem}
		The sequence of operators $A_{k, p}$, will be later called “multiplicative defect".
		For $k = 0$, Theorem \ref{thm_ex_ap} was proved in \cite[Theorem 4.3]{FinToeplImm} in its relation with the problem of transitivity of optimal holomorphic extensions.
	\end{rem}
	To establish this result, we will rely on several statements, which are of independent interest.
	The proofs of those statements extend our methods from \cite{FinOTAs}, \cite{FinToeplImm}.
	The first result we need is the following semi-classical analogue of the trace theorem from the theory of Sobolev spaces. We prove it in Section \ref{sect_trace}.
	\begin{thm}\label{thm_res_is_l2}
		For any $k \in \nat$, $p \in \nat$, and $f \in H^{0}_{(2)}(X, L^p \otimes F \otimes \mathcal{J}_Y^{k})$ the section $(\nabla^k f)|_Y \in \ccal^{\infty}(Y, {\rm{Sym}}^k (N^{1, 0})^* \otimes \iota^*(L^p \otimes F))$ is holomorphic.
		Also, for any $k \in \nat$, there is $p_1 \in \nat^*$, such that for any $f \in H^{0}_{(2)}(X, L^p \otimes F \otimes \mathcal{J}_Y^{k})$, the $L^2$-norm of the above section is finite, i.e. we have $(\nabla^k f)|_Y \in H^{0}_{(2)}(Y, {\rm{Sym}}^k (N^{1, 0})^* \otimes \iota^*(L^p \otimes F))$.
		Moreover, there is $C > 0$, such that 
		\begin{equation}\label{eq_res_is_l211}
			\big\| \res_{k, p} (f)  \big\|_{L^2(Y)} \leq C p^{\frac{n - m + k}{2}} \big\| f \big\|_{L^2(X)}
		\end{equation}
	\end{thm}
	The second result is an asymptotic version of Ohsawa-Takegoshi extension theorem for holomorphic jets, which we prove in Section \ref{sect_spec_low_bnd} by relying on the results from Section \ref{sect_tay_exp}.
	\begin{thm}\label{thm_ot_weak}
		For any $k \in \nat$, there are $C > 0$, $p_1 \in \nat^*$, such that for any $p \geq p_1$ and $g \in H^{0}_{(2)}(Y, {\rm{Sym}}^k (N^{1, 0})^* \otimes \iota^*(L^p \otimes F))$, there is $f \in H^{0}_{(2)}(X, L^p \otimes F \otimes \mathcal{J}_Y^k)$, such that $\res_{k, p} (f) = g$ and the following bound holds
		\begin{equation}\label{eq_ot_weak}
			\norm{f}_{L^2(X)} \leq \frac{C}{p^{\frac{n - m + k}{2}}} \norm{g}_{L^2(Y)}.
		\end{equation}
	\end{thm}
	\begin{rem}
		For $k = 0$, Theorems \ref{thm_res_is_l2} and \ref{thm_ot_weak} were proved in \cite[\S 4.1 and Theorem 4.4]{FinOTAs}.
	\end{rem}
	\begin{proof}[Proof of Theorem \ref{thm_ex_ap}.]
	\begin{sloppypar}
		Clearly, it suffices to prove that the kernels and the images of the operators $(\res_Y \circ \nabla^k B_p^{X, k Y})^*$ and $\ext_{k, p}$ coincide for $p$ big enough.
		First of all, we have
		\begin{equation}\label{eq_eq_ker_im_dual}
			\ker (\res_Y \circ \nabla^k B_p^{X, k Y})^{*} = ( \Im (\res_Y \circ \nabla^k B_p^{X, k Y}) )^{\perp}.
		\end{equation}
		Now, in Theorem \ref{thm_res_is_l2}, we proved that there is $p_1 \in \nat$, such that for any $p \geq p_1$, $\res_Y \circ \nabla^k B_p^{X, k Y}$ has its image inside of $H^0_{(2)}(Y, {\rm{Sym}}^k (N^{1, 0})^* \otimes \iota^*(L^p \otimes F))$.
		In Theorem \ref{thm_ot_weak}, we proved that there is $p_1 \in \nat$, such that for any $p \geq p_1$, the image of $\res_Y \circ \nabla^k B_p^{X, k Y}$ coincides exactly with $H^0_{(2)}(Y, {\rm{Sym}}^k (N^{1, 0})^* \otimes  \iota^*(L^p \otimes F))$.
		From this, and (\ref{eq_eq_ker_im_dual}), we see that the kernels of $(\res_Y \circ \nabla^k B_p^{X, k Y})^{*}$ and $\ext_{k, p}$ coincide.
		\end{sloppypar}
		Similar reasoning shows that the images of those operators coincide as well.
		In particular, for $p \geq p_1$, there is a unique sequence of operators $A_{k, p}$ as in (\ref{eq_resp_ap_lem}).
	\end{proof}
	\begin{rem}
		As we see from the proof, we never used the precise calculations of the constants on the right-hand side of Theorems \ref{thm_res_is_l2} and \ref{thm_ot_weak}.
		We nevertheless decided to include the precise estimates there due to the fact that our proof becomes no easier if one wishes to get simply the existence.
	\end{rem}

\subsection{Bounds on jets of holomorphic sections}\label{sect_trace}
	The main goal of this section is to prove Theorem \ref{thm_res_is_l2}, i.e. to give an estimate of the $L^2$-norm of a jet of a holomorphic section in terms of the $L^2$-norm of the section.
	The following result is crucial for this proof and later arguments.
	\begin{thm}\label{thm_bk_off_diag}
		For any $r \in \nat$, there are $c, C > 0$, $p_1 \in \nat^*$ such that for $p \geq p_1$, we have
		\begin{equation}\label{eq_bk_off_diag}
			\Big|  B_p^X(x_1, x_2) \Big|_{\ccal^r(X \times X)} \leq C p^{n + \frac{r}{2}} \cdot \exp \big(- c \sqrt{p} \cdot \dist(x_1, x_2) \big).
		\end{equation}
	\end{thm}
	\begin{proof}[Proof of Theorem \ref{thm_bk_off_diag}]
		For compact manifolds, Theorem \ref{thm_bk_off_diag} was implicit in Dai-Liu-Ma \cite[Theorem 4.18]{DaiLiuMa}.
		For manifolds of bounded geometry it was proved by Ma-Marinescu in \cite[Theorem 1]{MaMarOffDiag}. For the summary of numerous previous works on the off-diagonal estimates of the Bergman kernel as in (\ref{eq_bk_off_diag}), refer to \cite[p. 1328]{MaMarOffDiag}.
	\end{proof}
	\begin{proof}[Proof of Theorem \ref{thm_res_is_l2}]
		The main idea of our proof is to study the Schwartz kernel of the operator sending a holomorphic section to its $k$-jet and to use the exponential bound on the Bergman kernel.
		\par 
		First of all, let us verify that for any $f$ as above, $(\nabla^k f)|_Y$ is a holomorphic section of the vector bundle ${\rm{Sym}}^k (N^{1, 0})^* \otimes \iota^*(L^p \otimes F)$. 
		As this is a local statement, let us fix a point $y_0 \in Y$ with holomorphic coordinates $t_1, \ldots, t_m$, $z_1, \ldots, z_{n - m}$, around $y_0$ in $X$ in such way that $z_1, \ldots, z_{n - m}$ vanish along $Y$. 
		From Weierstrass division theorem, cf. \cite[Theorem II.2.3]{DemCompl}, we see that in a local holomorphic trivialization of $L$ and $F$ around $y_0$, we may represent $f$ in the following way
		\begin{equation}\label{eq_weier_div}
			f = \sum_{\substack{ \beta \in \nat^{n - m} \\ |\beta| = k }} z^{\beta} \cdot g_{\beta}(t, z),
		\end{equation}
		where $g_{\beta}$ are some local holomorphic functions.
		Using this representation and (\ref{eq_nabla_sym_der_ident}), we see that in the induced local frame over $Y$, the following identity holds
		\begin{equation}
			(\nabla^k f)|_Y = k! \cdot \sum_{\substack{ \beta \in \nat^{n - m} \\ |\beta| = k }} (dz)^{\odot \beta} \cdot g_{\beta}(t, 0).
		\end{equation}
		Since $g_{\beta}$ are holomorphic, we see that $(\nabla^k f)|_Y$ is a holomorphic section.
		\par 
		Let us now verify that the restriction satisfies the stated $L^2$-bound.
		Consider the operator $T_p : L^2(X, L^p \otimes F) \to \ccal^{\infty}(Y, {\rm{Sym}}^k (\iota^* TX^{1, 0})^* \otimes \iota^*(L^p \otimes F))$, defined as follows
		\begin{equation}
			T_p = \res_Y \circ ((\nabla^{1, 0})^k B_p^X),
		\end{equation}
		where $\nabla^{1, 0}$ represents above the holomorphic component of the covariant derivative, induced by the Chern connection on $L$, $F$ and $TX$.
		Let us now calculate the norm of $T_p$.
		\par 
		For this, we remark that the Schwartz kernel $T_p(y, x)$, $y \in Y$, $x \in X$, evaluated with respect to the volume form $dv_X$, of $T_p$ is given by
		\begin{equation}\label{eq_relat_res_oper_bergm_kern}
			T_p(y, x) =  (\nabla^{1, 0})^k B_p^X(y, x).
		\end{equation}
		Consider now the operator $U_p := T_p \circ T_p^*$.
		Directly from (\ref{eq_relat_res_oper_bergm_kern}) and the fact that $B_p^X \circ B_p^X = B_p^X$, we conclude that the Schwartz kernel $U_p(y, y')$, $y, y' \in Y$, of $U_p$, evaluated with respect to the volume form $dv_Y$, satisfies
		 \begin{equation}\label{eq_up_schw}
		 	U_p(y, y') = ((\nabla^{1, 0})^k B_p^X (\nabla^{1, 0; *})^k)(y, y'),
		 \end{equation}
		 where $\nabla^{1, 0; *}$ is the induced connection on the duals of $L^p \otimes F$ and $T^{1, 0}X$.
		 Directly from Theorem \ref{thm_bk_off_diag}, Proposition \ref{prop_norm_bnd_distk_expbnd} and (\ref{eq_up_schw}), we deduce that there are $C > 0$ and $p_1 \in \nat$, such that for $p \geq p_1$, we have $\| U_p \| \leq C p^{n - m + k}$.
		 We deduce the bound (\ref{eq_res_is_l211}) by this, the trivial remark $\| U_p \| = \| T_p \|^2$ and the fact that for any $f \in H^{0}_{(2)}(X, L^p \otimes F \otimes \mathcal{J}_Y^{k})$, we have $T_p (f) = \res_{k, p}(f)$.
	\end{proof}
	
\subsection{Taylor expansion of the holomorphic differential near a submanifold}\label{sect_tay_exp}
	The main goal of this section is to recall the calculation of the first two terms of the Taylor expansion of $\dbar^{L^p \otimes F}$-operator, considered in a shrinking neighborhood of $Y$ of size $\frac{1}{\sqrt{p}}$, as $p \to \infty$. 
	This section is taken almost entirely from \cite[\S 4.1]{FinOTAs}.
	\par 
	We consider a triple $(X, Y, g^{TX})$ of bounded geometry.
	By means of the exponential map as in (\ref{eq_kappan}), we identify a neighborhood of the zero section $B_{r_{\perp}}(N)$ in the normal bundle $N$, to a neighborhood $U := B_Y^X(r_{\perp})$ of $Y$ in $X$.
	\par 
	Recall that the projection $\pi : U \to Y$ and the identifications of $L, F$ to $\pi^* (L|_Y), \pi^* (F|_Y)$ in $B_Y^X(r_{\perp})$ were defined before (\ref{eq_ext0_op}).
	We similarly identify $TX$ to $\pi^* (TX|_Y)$ over $B_Y^X(r_{\perp})$ using the parallel transport with respect to the Levi-Civita connection $\nabla^{TX}$.
	Remark that since $g^{TX}$ is Kähler by (\ref{eq_gtx_def}), the decomposition $TX \otimes_{\real} \comp = T^{1, 0} X \oplus T^{0, 1} X$ is preserved by $\nabla^{TX}$, cf. \cite[Theorem 1.2.8]{MaHol}. In other words, the identification of $TX$ with $\pi^* (TX|_Y)$ induces the identifications 
	\begin{equation}\label{eq_tgt_isom}
		\tau: \pi^* (T^{1, 0} X|_Y) \to T^{1, 0} X|_U, \qquad \tau: \pi^* (T^{0, 1} X|_Y) \to T^{0, 1} X|_U.
	\end{equation}
	\par 
	We define the $1$-form $\Gamma^{F}$ with values in $\enmr{\pi^* (F|_Y)}$
	\begin{equation}\label{eq_gamma_l_def}
		\Gamma^{E} = \nabla^E - \pi^*(\nabla^E|_{Y}),
	\end{equation}
	where we implicitly used the above isomorphism.
	Similarly, we define $\Gamma^L$. 
	Recall also that the connection $\nabla^N$ on $N$ was introduced before (\ref{eq_sec_fund_f}).
	\begin{comment}
	As both manifolds $(X, g^{TX})$, $(Y, g^{TY})$ are Kähler, and, hence, by \cite[Theorem 1.2.8]{MaHol}, the respective Levi-Civita connections coincide with the connections induced by the Chern connections on the holomorhic tangent bundles, the connection $\nabla^N$ coincides with the Chern connection on $(N, g^N)$ by the usual properties of short exact sequences for Hermitian vector bundles, cf. \cite[Proposition 1.6.6]{KobaVB}.
	\end{comment}
	The connection $\nabla^N$ induces the splitting 
	\begin{equation}
		TN = N \oplus T^H N
	\end{equation}
	of the tangent space of the total space of $N$. Here $T^H N$ is the horizontal part of $N$ with respect to the connection $N$. 
	If $U \in TY$, we denote by $U^H \in T^H N$ the horizontal lift of $U$ in $T^H N$.
	\par 
	For $\epsilon > 0$, we denote by $\mathbb{E}(\epsilon)$ (resp. $\mathbb{E}$) the set of smooth sections of $\pi^*(L^p|_Y\otimes F|_Y)$ on $B_{\epsilon}(N)$ (resp. on the total space of $N$).
	We also denote by $\mathbb{E}^{(0, 1)}(\epsilon)$ (resp. $\mathbb{E}^{(0, 1)}$) the set of smooth sections of $\pi^*(T^{*(0, 1)}X|_{Y}) \otimes \pi^*(L^p|_Y\otimes F|_Y)$ on $B_{\epsilon}(N)$ (resp. on the total space of $N$).
	\par 
	Clearly, the above isomorphisms allow us to see $\dbar^{L^p \otimes F}$ as an operator
	\begin{equation}\label{eq_dbar_op_interpr}
		\dbar^{L^p \otimes F} : \mathbb{E}(r_{\perp}) \to \mathbb{E}^{(0, 1)}(r_{\perp}).
	\end{equation}
	\par 
	We fix a point $y_0 \in Y$ and an orthonormal frame $(e_1, \ldots, e_{2m})$ (resp. $(e_{2m+1}, \ldots, e_{2n})$) in $(T_{y_0}Y, g^{TY})$ (resp. in $(N_{y_0}, g^{N}_{y_0})$) such that (\ref{eq_cond_jinv}) is satisfied.
	Using the exponential coordinates on $Y$ and the parallel transport of $(e_{2m+1}, \ldots, e_{2n})$ along the geodesics on $Y$, as in Fermi coordinates $\psi_{y_0}$ in (\ref{eq_defn_fermi}), we introduce complex coordinates $z_1, \ldots, z_m$ on $Y$ and linear “vertical" coordinates $z_{m + 1}, \ldots, z_n$ on $N$.
	Using those coordinates, we define the operators
	\begin{equation}\label{eq_defn_dbarhn}
		\dbar_H^{L^p \otimes F}, \mathcal{L}_N^{L^p \otimes F} : \mathbb{E} \to \mathbb{E}^{(0, 1)}, 
	\end{equation}
	by prescribing their action at a point $(y_0, Z_N)$, $Z_N \in \real^{2(n-m)}$, as follows
	\begin{equation}\label{eq_hor_norm_dop}
		\dbar_H^{L^p \otimes F} = \sum_{i = 1}^{m} d \overline{z}_i|_{y_0} \cdot \Big( \frac{\partial}{\partial \overline{z}_i} \big|_{y_0}\Big)^{H},
		\qquad 
		\mathcal{L}_N^{L^p \otimes F} = \sum_{i = m+1}^{n} d\overline{z}_i|_{y_0} \cdot \Big( \frac{\partial}{\partial \overline{z}_i} + \frac{\pi z_i}{2} \Big).
	\end{equation}
	The first differentiation in (\ref{eq_hor_norm_dop}) is well-defined because $\pi_* ( \frac{\partial}{\partial \overline{z}_i} |_{y_0} )^{H} =  \frac{\partial}{\partial \overline{z}_i} |_{y_0}$ is of type $(0, 1)$, and the second derivation is well-defined because the vector bundles are trivialized along fibers of $\pi$.
	\par 
	Below a variable $t \in \real$ is related to $p \in \nat$ by 
	\begin{equation}\label{eq_t_p_rel}
		t = \frac{1}{\sqrt{p}}.
	\end{equation}
	For any $\epsilon > 0$ define the rescaling operator $F_t : \mathbb{E}(\epsilon) \to  \mathbb{E}(\frac{\epsilon}{t})$ for $f \in \mathbb{E}(\epsilon)$ as follows
	\begin{equation}\label{eq_ft_defn}
		(F_t f )(y, Z_N) := f \big( y, t Z_N \big), \qquad (y, Z_N) \in B_{\frac{\epsilon}{t}}(N).
	\end{equation}
	The operator $F_t : \mathbb{E}^{(0, 1)}(\epsilon) \to  \mathbb{E}^{(0, 1)}(\frac{\epsilon}{t})$ is defined in an analogous way.
	\begin{thm}[{\cite[Theorem 4.3]{FinOTAs}}]\label{thm_tay_exp_dbar}
		As $p \to \infty$, we have
		\begin{equation}\label{eq_tay_exp_dbar}
			F_t \circ \dbar^{L^p \otimes F} \circ F_t^{-1}
			=
			\frac{1}{t}
			\mathcal{L}_N^{L^p \otimes F}
			+
			\dbar_H^{L^p \otimes F}
			+
			O\big( 
			t |Z_N|^2 \dbar_N + t |Z_N| \dbar_H + t|Z_N|
			\big),
		\end{equation}
		where $O(t |Z_N|^2 \dbar_N + t |Z_N| \dbar_H + t|Z_N|)$
		is an operator of the form 
		$
			\sum_{i = 1}^{m} a_i(t, y, Z_N) \cdot d \overline{z}_i|_{y_0} \cdot ( \frac{\partial}{\partial \overline{z}_i} |_{y_0})^{H},
			+
			\sum_{j = m+1}^{n} b_j(t, y, Z_N) \cdot d \overline{z}_j|_{y_0} \cdot \frac{\partial}{\partial \overline{z}_j},
			+
			c(t, y, Z_N),
		$
		such that there is a constant $C > 0$, for which $|a_i(t, y, Z_N)| \leq C t |Z_N|^2$, $|b_j(t, y, Z_N)| \leq C t |Z_N|$, $|c(t, y, Z_N)| \leq C t|Z_N|$ holds for any $y \in Y$, $|Z_N| < r_{\perp}$, $i = 1, \ldots, m$, and $j = m + 1, \ldots, n$. 
	\end{thm}
	\begin{rem}
		Bismut-Lebeau in \cite[Theorem 8.18]{BisLeb91} established an analogue of Theorem \ref{thm_tay_exp_dbar}, which corresponds to trivial $(L, h^L)$ in our setting. 
		In this case the operator $\mathcal{L}_N^{L^p \otimes F}$ from (\ref{eq_hor_norm_dop}) doesn't have an additional $\frac{\pi z_i}{2}$ term.
		Another closely related Taylor expansion is due to Dai-Liu-Ma \cite[Theorem 4.6]{DaiLiuMa}, and it corresponds to $Y$ equal to a point in our setting. 
	\end{rem}

\subsection{Asymptotic Ohsawa-Takegoshi extension theorem for holomorphic jets}\label{sect_spec_low_bnd}	
	The main goal of this section is to establish the asymptotic analogue of Ohsawa-Takegoshi extension theorem for holomorphic jets. In other words, we show that Theorem \ref{thm_ot_weak} holds.
	\par 
	The main idea of the proof of Theorem \ref{thm_ot_weak}, which follows rather closely our proof from \cite[Theorem 4.4]{FinOTAs}, corresponding to the case $k = 0$, is to pass through the general framework of the proof of Ohsawa-Takegoshi extension theorem. We choose a smooth extension of $g$ over $X$, and then obtain the holomorphic extension by modifying the smooth one using a solution of $\dbar$-equation with singular weight, which forces the solution to annihilate along $Y$.
	For holomorphic jets, such a strategy was applied in many places before, e.g. \cite{PopovNonredExt}.
	\par
	The novelty here as well as in \cite{FinOTAs} is that instead of choosing an arbitrary smooth extension, we choose a specific one, given by the operator (\ref{eq_ext0_op}). 
	This allows us to significantly simplify the original proof of Ohsawa-Takegoshi extension theorem (in our asymptotic setting).
	In particular, instead of considering a double sequence of singular weights (depending on $p$ and on additional parameter $\epsilon$) which would dampen the fact that we do not know much of an arbitrary smooth extension, it would suffice to consider a single sequence of weights (only depending on $p$).
	\par
	Recall that the function $\rho$ was defined in (\ref{defn_rho_fun}). 
	Let us now consider the functions $\delta_Y : X \setminus Y \to \real$, $\alpha_Y : X \to \real$,  defined as
	\begin{equation}\label{eq_delta_defn_y}
	\begin{aligned}
		&
		\delta_Y(x) := \log \big(\dist_X(x, Y) \big) \cdot \rho 
		\Big(
			 \frac{\dist_X(x, Y)}{r_{\perp}} 
		\Big),
		\\
		&
		\alpha_Y(x) := \dist_X(x, Y)^2 \cdot \rho 
		\Big(
			 \frac{\dist_X(x, Y)}{r_{\perp}} 
		\Big)
		+
		\Big(
		1
		-
		\rho 
		\Big(
			 \frac{\dist_X(x, Y)}{r_{\perp}} 
		\Big)
		\Big)
		.
	\end{aligned}
	\end{equation}
	\par 
	Now, recall that a function $f : X \to [-\infty, +\infty[$ on a complex Hermitian manifold $(X, \omega)$ is called quasi-plurisubharmonic if it is upper-semicontinuous, and there is a constant $C \in \real$, such that the following inequality holds in the distributional sense
	\begin{equation}\label{eq_distr_ineq}
		\imun \partial \dbar f \geq -C \omega.
	\end{equation}
	We denote by ${\rm{PSH}}(X, C \omega)$ the set of quasi-plurisubharmonic functions $f$, verifying (\ref{eq_distr_ineq}).
	\begin{thm}[{\cite[Theorem 2.31]{FinOTAs}}]\label{thm_plurisub}
		There is $C > 0$, such that $\delta_Y, \alpha_Y, -\alpha_Y \in {\rm{PSH}}(X, C \omega)$.
	\end{thm}
	\par 
	Another result we will use concerns the $L^2$-bounds of derivatives of holomorphic sections.
	\begin{prop}[{\cite[Proposition 4.5]{FinOTAs}}]\label{prop_der_bound}
		For any $k \in \nat$, there are $C > 0$, $p_1 \in \nat^*$, such that for any $p \geq p_1$ and $f \in H^{0}_{(2)}(X, L^p \otimes F)$, we have
		\begin{equation}\label{eq_der_bound}
			\big\| \nabla^k f \big\|_{L^2(X)} \leq C p^{\frac{k}{2}} \big\| f \big\|_{L^2(X)},
		\end{equation}
		where $\nabla$ is the covariant derivative with respect to the induced Chern and Levi-Civita connections.
	\end{prop} 
	\begin{proof}[Proof of Theorem \ref{thm_ot_weak}]
		Recall that the operator $\ext_{k, p}^{0}$ was defined in (\ref{eq_ext0_op}). 
		We would like to verify that for any $g \in H^{0}_{(2)}(Y, {\rm{Sym}}^k (N^{1, 0})^* \otimes \iota^*(L^p \otimes F))$, the form $\alpha := \dbar^{L^p \otimes F} (\ext_{k, p}^{0} g)$ vanishes at least up to order $k + 1$ over $Y$.
		\par 
		Indeed, let us work in a neighborhood $V := B_Y^X(\frac{r_{\perp}}{4})$ of $Y$ in $X$. 
		Recall that $t \in \real_+$ and $F_t$ were defined in (\ref{eq_ft_defn}).
		Then in the notations of (\ref{eq_ext0_op}), on $V$, we have
		\begin{equation}\label{eq_gtild_d}
			\ext_{k, p}^{0} g = \frac{F_t^{-1} \tilde{g}}{\sqrt{p}^k}, \qquad \tilde{g}(y, Z_N) = g(y) \cdot z_N^{\otimes k} \cdot  \exp \Big(- \frac{\pi}{2} |Z_N|^2 \Big).
		\end{equation}
		Recall that $\dbar_H^{L^p \otimes F}, \mathcal{L}_N^{L^p \otimes F}$ were defined in (\ref{eq_hor_norm_dop}).
		A trivial calculation shows that on $V$, we have
		\begin{equation}\label{eq_gtild_n}
			\mathcal{L}_N^{L^p \otimes F} \tilde{g} = 0.
		\end{equation}
		Also, since $\nabla^N$ preserves $g^N$, in the notations of (\ref{eq_hor_norm_dop}), similarly to \cite[(8.97)]{BisLeb91}, we have
		\begin{equation}\label{eq_gtild_h}
			\Big( \Big( \frac{\partial}{\partial \overline{z}_i} \big|_{y_0}\Big)^{H} \tilde{g}\Big) (y_0, Z_N)
			=
			\Big( \frac{\partial}{\partial \overline{z}_i} \big( g \cdot z_N^{\otimes k} \big) \Big)(y_0)  \exp \Big(- \frac{\pi}{2} |Z_N|^2 \Big).
		\end{equation}
		As a consequence of (\ref{eq_gtild_h}) and the fact that $g$ is holomorphic, we obtain 
		\begin{equation}\label{eq_gtild_h2}
			\dbar_H^{L^p \otimes F}  \tilde{g} = 0.
		\end{equation}
		From (\ref{eq_gtild_d}), (\ref{eq_gtild_n}), (\ref{eq_gtild_h2}) and the fact that all the residue terms in Theorem \ref{thm_tay_exp_dbar} contain $|Z_N|$, we deduce that $\alpha$ vanishes at least up to order $k + 1$ along $Y$.
		\par 
		Now, using the $L^2$-estimates, let us construct a holomorphic perturbation of $\ext_{k, p}^{0} g$, satisfying the assumptions of Theorem \ref{thm_ot_weak}.
		Recall that $\delta_Y : X \setminus Y \to \real$, $\alpha_Y : X \to \real$,  were defined in (\ref{eq_delta_defn_y}).
		For $\epsilon > 0$, let us now define the weight $\delta_p : X \setminus Y \to \real$ as follows
		\begin{equation}\label{eq_delta_p_defn_wght}
			\delta_p := 2(n - m + k)\delta_Y - \epsilon p \alpha_Y.
		\end{equation}
		By taking $\epsilon$ small, by Theorem \ref{thm_plurisub}, we see that there exists $p_1 \in \nat^*$, such that for any $p \geq p_1$, over $X$, the following inequality holds in the distributional sense
		\begin{equation}\label{eq_pomg_deltap}
			p \omega + \frac{\imun}{2 \pi}  \partial \dbar \delta_p >  \frac{p}{2}\omega.
		\end{equation} 
		Let us fix $\epsilon$ small enough, so that it verifies the above inequality and for any $|Z_N| < r_{\perp}$, we have
		\begin{equation}\label{eq_zn_diff_phi}
			\frac{\pi}{2}  |Z_N|^2 - \epsilon \alpha_Y(y, Z_N) \geq \frac{\pi}{4} |Z_N|^{2}.
		\end{equation}
		\par 
		\begin{sloppypar}
		We will now prove that there are $C > 0$, $p_1 \in \nat^*$, such that for any $p \geq p_1$, $g \in H^{0}_{(2)}(Y, {\rm{Sym}}^k (N^{1, 0})^* \otimes \iota^*(L^p \otimes F))$ and $\alpha := \dbar^{L^p \otimes F} (\ext_{k, p}^{0} g)$, we have
		\begin{equation}\label{eq_int_alph}
			\int_{X \setminus Y} | \alpha |^2 e^{-\delta_p} dv_X \leq C \| g \|_{L^2(Y)}^{2}.
		\end{equation}
		\end{sloppypar}
		\par 
		By Proposition \ref{prop_bndg_tripl}, cf. \cite[proof of (4.33)]{FinOTAs}, we see that there are $c, C > 0$, $p_1 \in \nat^*$, such that for any $p \geq p_1$, we have
		\begin{equation}\label{eq_int_alph222}
			\int_{X \setminus B_Y^X(\frac{r_{\perp}}{4})} | \alpha |^2 e^{-\delta_p} dv_X \leq C \exp(-cp) \Big( \| g \|^2_{L^2(Y)} + \| \nabla g \|^2_{L^2(Y)}	 \Big).
		\end{equation}
		Now, as $\alpha$ has support in $B_Y^X(\frac{r_{\perp}}{2})$, it is enough to work in $(y, Z_N)$, $y \in Y$, $Z_N \in N_y$ coordinates.
		To estimate the integral over $B_Y^X(\frac{r_{\perp}}{4})$, we use (\ref{eq_gtild_d}) and make the change of variables by $F_t$ to get
		\begin{multline}\label{eq_int_rescalling}
			\int_{B_Y^X(\frac{r_{\perp}}{4})} | \alpha |^2 e^{-\delta_p} dv_Y \wedge dv_N
			\\
			=
			\int_{B_{\frac{r_{\perp}}{4t}}(N)} \Big| \big(F_t \circ \dbar^{L^p \otimes F} \circ F_t^{-1} \big) \tilde{g} \Big|^2(y, Z_N) 
			 \frac{e^{\epsilon p \alpha_Y(y, tZ_N)}}{|Z_N|^{2(n - m + k)}} dv_Y \wedge dv_N.
		\end{multline}
		We see that to estimate the right-hand side of (\ref{eq_int_rescalling}), we may apply Theorem \ref{thm_tay_exp_dbar}.
		From (\ref{eq_gtild_n}), (\ref{eq_gtild_h2}), we see that the first two terms of the asymptotic expansion of $(F_t \circ \dbar^{L^p \otimes F} \circ F_t^{-1} ) \tilde{g}$ vanish.
		From this and the trivial fact that there is $C > 0$, such that for $j = 1, 2$, we have
		\begin{equation}
			\int_{\real^{2(n-m)}} \frac{|Z_N|^{2k + j} \exp (- \frac{\pi}{4} |Z_N|^2) dv_{\real^{2(n - m)}}(Z_N)}{|Z_N|^{2(n - m + k)}} < C,
		\end{equation}		 
		we conclude that there are $C > 0$, $p_1 \in \nat$, such that for any $p \geq p_1$, $g \in H^{0}_{(2)}(Y, {\rm{Sym}}^k (N^{1, 0})^* \otimes \iota^*(L^p \otimes F))$ and $\alpha := \dbar^{L^p \otimes F} (\ext_{k, p}^{0} g)$, we have
		\begin{equation}\label{eq_int_alph2}
			\int_{B_Y^X(\frac{r_{\perp}}{4})} | \alpha |^2 e^{-\delta_p} dv_X 
			\leq 
			\frac{C}{\sqrt{p}} \Big( \| g \|^2_{L^2(Y)} + \| \nabla g \|^2_{L^2(Y)}	 \Big).
		\end{equation}
		From Proposition \ref{prop_der_bound}, (\ref{eq_int_alph222}) and (\ref{eq_int_alph2}), we deduce (\ref{eq_int_alph}).
		\par 
		By \cite[Theorem 1.5]{Dem82}, $X \setminus Y$ is a complete Kähler manifold.
		Hence, by (\ref{eq_pomg_deltap}), we may resolve the $\dbar$-equation on $X \setminus Y$, see \cite[Proposition 13.4]{DemBookAnMet}. From this, the trivial fact that $\dbar^{L^p \otimes F} \alpha = 0$ and (\ref{eq_int_alph}), we see that there are $C > 0$, $p_1 \in \nat^*$, such that for any $p \geq p_1$, $g \in H^{0}_{(2)}(Y, {\rm{Sym}}^k (N^{1, 0})^* \otimes \iota^*(L^p \otimes F))$, there is a section $f_0 \in \ccal^{\infty}(X \setminus Y, L^p \otimes F)$, such that
		\begin{equation}\label{eq_dbar_sol}
			\dbar^{L^p \otimes F} f_0 = \alpha, \qquad \int_{X \setminus Y} | f_0 |^2 e^{-\delta_p} dv_X \leq  \frac{C}{p} \| g \|_{L^2(Y)}^{2}.
		\end{equation}
		Let us prove that $f := \ext_{k, p}^{0} g - f_0$ verifies the assumptions of Theorem \ref{thm_ot_weak}.
		\par From (\ref{eq_dbar_sol}), we see that over $X \setminus Y$, $\dbar^{L^p \otimes F} f = 0$. 
		Also, by (\ref{eq_dbar_sol}), we easily get that $f \in L^2(X, L^p \otimes F)$. 
		By the standard regularity result, \cite[Lemme 6.9]{Dem82}, $f$ extends smoothly and the equation $\dbar^{L^p \otimes F} f = 0$ holds on $X$.
		In particular, $f_0$ extends smoothly as well.
		However, since $\exp(-2(n - m + k)\delta_Y)$ is not integrable, the $L^2$-bound (\ref{eq_dbar_sol}) implies that $f_0$ has to vanish at least up to order $k + 1$ along $Y$.
		Hence, we conclude that $(\nabla^k f)|_Y = g$.
		It is only left to verify that $f$ satisfies the needed $L^2$-bound (\ref{eq_ot_weak}).
		\par 
		An easy calculation, using Proposition \ref{prop_bndg_tripl} and (\ref{eq_vol_comp_unif}), cf. \cite[(4.38)]{FinOTAs}, shows that there are $c, C > 0$, such that we have
		\begin{equation}\label{eq_exp0_bnd_norm}
			\frac{c}{p^{\frac{n - m + k}{2}}} \norm{g}_{L^2(Y)} \leq \big \| \ext_{k, p}^{0} g \big \| _{L^2(X)} \leq \frac{C}{p^{\frac{n - m + k}{2}}} \norm{g}_{L^2(Y)}.
		\end{equation}
		Let us now prove the following bound
		\begin{equation}\label{eq_l2_est_fin}
			\int_{X} | f_0 |^2 e^{-\delta_p} dv_X  \geq C p^{n-m + k} \int_{X} | f_0 |^2 dv_X.
		\end{equation}
		This will be clearly enough for our needs, as from the $L^2$-bound in (\ref{eq_dbar_sol}), (\ref{eq_exp0_bnd_norm}) and (\ref{eq_l2_est_fin}), we would deduce the $L^2$-bound (\ref{eq_ot_weak}).
		\par 
		First of all, since $\alpha_Y \geq \min\{  \frac{1}{2} (\frac{r_{\perp}}{4})^2, \frac{1}{2} \}$ on $X \setminus B_Y^X(\frac{r_{\perp}}{4})$, there are $c, C > 0$, such that 
		\begin{equation}\label{eq_l2_est_fin0}
			\int_{X \setminus B_Y^X(\frac{r_{\perp}}{4})} | f_0 |^2 e^{-\delta_p} dv_X 
			\geq
			C \exp(\epsilon c p) \int_{X \setminus B_Y^X(\frac{r_{\perp}}{4})} | f_0 |^2 dv_X 
		\end{equation}
		\par 
		It is now only left to give the lower bound for the integrand on the left-hand side of (\ref{eq_l2_est_fin}), where the integration is done over $B_Y^X(\frac{r_{\perp}}{4})$.
		But remark that from (\ref{eq_delta_defn_y}) and (\ref{eq_delta_p_defn_wght}), over $B_Y^X(\frac{r_{\perp}}{4})$, there is $C > 0$, such that for any $p \in \nat^*$, we have $e^{-\delta_p} \geq C p^{n - m + k}$.
		From this, we deduce 
		\begin{equation}\label{eq_l2_est_fin132}
			\int_{B_Y^X(\frac{r_{\perp}}{4})} | f_0 |^2 e^{-\delta_p} dv_X  \geq C p^{n - m + k} \int_{B_Y^X(\frac{r_{\perp}}{4})} | f_0 |^2 dv_X.
		\end{equation}
		From (\ref{eq_l2_est_fin0}) and (\ref{eq_l2_est_fin132}), we obtain (\ref{eq_l2_est_fin}).
	\end{proof}
	\begin{rem}
	\begin{sloppypar}
		Our proof shows that there is $\ext_{k, p}^{1} :  H^{0}_{(2)}(Y, {\rm{Sym}}^k (N^{1, 0})^* \otimes  \iota^*(L^p \otimes F)) \to H^{0}_{(2)}(X, L^p \otimes F \otimes \mathcal{J}_Y^k)$, verifying  $(\nabla^k \ext_{k, p}^{1} g) |_Y = g$ for $g \in  H^{0}_{(2)}(Y, {\rm{Sym}}^k (N^{1, 0})^* \otimes \iota^*(L^p \otimes F))$, and such that (\ref{eq_ext_as}) holds for $\ext_{k, p}^{1}$ instead of $\ext_{k, p}^{0}$.
	\end{sloppypar}
	\end{rem}

\section{Asymptotics of the extension operator for holomorphic jets}\label{sect_asy_main_op}

The main goal of this section is to study the asymptotic expansion of the extension operator for holomorphic jets.
For this, in Section \ref{sect_schw_assss}, we prove the exponential bounds for the Schwartz kernels of the extension operator and orthogonal Bergman kernel of order $k$, and study their asymptotics.
We also deduce from those statements Theorems \ref{thm_high_term_ext} and \ref{thm_log_bk}.
Then, in Section \ref{sect_adj_t_oper}, we establish the announced exponential bounds and asymptotics, and deduce Theorem \ref{thm_isom} from our methods.
Finally, in Section \ref{sect_peak_ho}, we apply the results from Section \ref{sect_schw_assss} to the case when the submanifold corresponds to a fixed point and deduce the asymptotics of higher order peak sections.

\subsection{Schwartz kernels of the extension and orthogonal Bergman projectors}\label{sect_schw_assss}
The main goal of this section is to study the Schwartz kernels of the extension operator and orthogonal Bergman kernels of order $k \in \nat$.
In particular, we prove the exponential estimates for the Schwartz kernels of those operators, and show that those Schwartz kernels admit a full asymptotic expansion, as powers of the line bundle tend to infinity.
\par 
We use notations from Section \ref{sect_intro} and assume that the triple $(X, Y, g^{TX})$ is of bounded geometry.
Let us fix $k \in \nat$ once and for all. Our first main result goes as follows.

	\begin{thm}\label{thm_ext_exp_dc}
		There are $c > 0$, $p_1 \in \nat^*$,  such that for any $r, l \in \nat$, there is $C > 0$, such that for any $p \geq p_1$, $x_1, x_2 \in X$, $y \in Y$, the following estimates hold
		\begin{equation}\label{eq_ext_exp_dc}
		\begin{aligned}
			&
			a) \, \big|  \ext_{k, p}(x_1, y) \big|_{\ccal^r} \leq C p^{m + \frac{r - k}{2}} \exp \big(- c \sqrt{p} \cdot \dist(x_1, y) \big),
			\\
			&
			b) \, \big|  B_{k, p}^{\perp}(x_1, x_2) \big|_{\ccal^r} \leq C p^{n + \frac{r}{2}} \exp \big(- c \sqrt{p} \cdot (  \dist(x_1, x_2) + \dist(x_1, Y) + \dist(x_2, Y) ) \big).
		\end{aligned}
		\end{equation}
	\end{thm}
	\begin{rem}
		For $k = 0$, the corresponding result was proved in \cite[Theorems 1.5, 1.8]{FinOTAs}.
		Our proof here is different even in the case $k = 0$.
	\end{rem}
	Theorem \ref{thm_ext_exp_dc} implies that to understand fully the asymptotics of the Schwartz kernel of the extension operator and orthogonal Bergman kernels of order $k$, it suffices to do so in a neighborhood of a fixed point $(y_0, y_0) \in Y \times Y$ in $X \times Y$ and $X \times X$.
	Our next result shows that after a reparametrization by a homothety with factor $\sqrt{p}$ in Fermi coordinates around $(y_0, y_0)$, the kernel of those operators admit a complete asymptotic expansion in integer powers of $\sqrt{p}$, as $p \to \infty$.
	\par 
	\begin{thm}\label{thm_ext_as_expk}
		For any $r \in \nat$, $y_0 \in Y$, there are polynomials $J_{k, r}^{E}(Z, Z'_Y) \in {\rm{Sym}}^k (N_{y_0}^{1, 0}) \otimes \enmr{F_{y_0}}$ in $Z \in \real^{2n}$, $Z'_Y \in \real^{2m}$, of parity $k + r$, $\deg J_{k, r}^{E} \leq k^2 + k^2 r + 3 r$, whose coefficients have the same properties as the coefficients of the polynomials from Theorem \ref{thm_berg_off_diag}, and which vanish at least up to order $k$ along $\real^{2m} \times \real^{2m} \subset \real^{2n} \times \real^{2m}$, such that for $F_{k, r}^{E} := J_{k, r}^{E} \cdot \mathscr{E}_{n, m}^{0}$, the following holds.
		\par 
		There are $\epsilon, c > 0$, $p_1 \in \nat^*$, such that for any $r, l, l' \in \nat$, there are $C, Q  > 0$, such that for any $y_0 \in Y$, $p \geq p_1$, $Z = (Z_Y, Z_N)$, $Z_Y, Z'_Y \in \real^{2m}$, $Z_N \in \real^{2(n - m)}$, $|Z|, |Z'_Y| \leq \epsilon$, $\alpha \in \nat^{2n}$, $\alpha' \in \nat^{2m}$, $|\alpha| + |\alpha'| \leq l$, we have
		\begin{multline}\label{eq_ext_as_expk}
			\bigg| 
				\frac{\partial^{|\alpha|+|\alpha'|}}{\partial Z^{\alpha} \partial Z'_Y{}^{\alpha'}}
				\bigg(
					\frac{1}{p^{m - \frac{k}{2}}} \ext_{k, p} \big(\psi_{y_0}^{X|Y}(Z), \phi_{y_0}^Y(Z'_Y) \big)
					\\
					-
					\sum_{j = 0}^{r}
					p^{-\frac{j}{2}}						
					F_{k, j}^{E}(\sqrt{p} Z, \sqrt{p} Z'_Y) 
					\kappa_{\psi}^{X|Y}(Z)^{-\frac{1}{2}}
					\kappa_{\phi}^Y(Z'_Y)^{-\frac{1}{2}}
				\bigg)
			\bigg|_{\ccal^{l'}}
			\\
			\leq
			C p^{- \frac{r + 1 - l}{2}}
			\cdot
			\Big(1 + \sqrt{p}|Z| + \sqrt{p} |Z'_Y| \Big)^{Q}
			\exp\Big(- c \sqrt{p} \big( |Z_Y - Z'_Y| + |Z_N| \big) \Big),
		\end{multline}
		where the $\ccal^{l'}$-norm is taken with respect to $y_0$.
		Also, the following identity holds
		\begin{equation}\label{eq_je0_expk}
			J_{k, 0}^{E}(Z, Z'_Y) = 
			\kappa_N^{\frac{1}{2}}(y_0) \cdot
			\sum_{\substack{\beta \in \nat^{n - m} \\ |\beta| = k}}
			\frac{1}{\beta!} \cdot
			z_N^{\beta} \cdot \big( \frac{\partial}{\partial z'_N} \big)^{\odot \beta}.
		\end{equation}
		Moreover, under the assumption (\ref{eq_comp_vol_omeg}), we have
		\begin{multline}\label{eq_je1_expk}
			J_{k, 1}^{E}(Z, Z'_Y) = {\rm{Id}}_{F_{y_0}} \cdot 
			g \big(z_N, A(\overline{z}_Y - \overline{z}'_Y) (\overline{z}_Y - \overline{z}'_Y) \big)
		 	\cdot
		 	\sum_{\substack{\beta \in \nat^{n - m} \\ |\beta| = k}}
			\frac{1}{\beta!} \cdot
			z_N^{\beta} \cdot \big( \frac{\partial}{\partial z'_N} \big)^{\odot \beta},
		\end{multline}
		where $A$ is the second fundamental form of $\iota$, introduced in (\ref{eq_sec_fund_f}).
 	\end{thm}
	\begin{rem}
		In particular, from (\ref{eq_pperp_defn_fu2n}), we have $F_{k, 0}^{E}(Z, Z'_Y) = \mathscr{E}_{n, m}^{k}$ and 
		\begin{equation}\label{eq_f1e_form}
			F_{k, 1}^{E}(Z, Z'_Y) = {\rm{Id}}_{F_{y_0}} \cdot 
			g \big(z_N, A(\overline{z}_Y - \overline{z}'_Y) (\overline{z}_Y - \overline{z}'_Y) \big)
		 	\cdot
		 	\mathscr{E}_{n, m}^{k}.
		\end{equation}
	\end{rem}
	\begin{thm}\label{thm_berg_perp_off_diagk}
		For any $r \in \nat$, $y_0 \in Y$, there are polynomials $J_{k, r}^{\perp}(Z, Z') \in \enmr{F_{y_0}}$, $Z, Z' \in \real^{2n}$, of parity $k$, $\deg J_{k, r}^{\perp} \leq 2k^2 + k^2 r + 3 r$, whose coefficients have the same properties as the coefficients of the polynomials from Theorem \ref{thm_berg_off_diag}, and which vanish at least up to order $k$ along $\real^{2m} \times \real^{2n} \subset \real^{2n} \times \real^{2n}$ and $\real^{2n} \times \real^{2m} \subset \real^{2n} \times \real^{2n}$, such that for $F_{k, r}^{\perp} := J_{k, r}^{\perp} \cdot \mathscr{P}_{n, m}^{\perp, 0}$, the following holds.
		\par 
		There are $\epsilon, c > 0$, $p_1 \in \nat^*$, such that for any $r, l, l' \in \nat$, there are $C, Q  > 0$, such that for any $y_0 \in Y$, $p \geq p_1$, $Z = (Z_Y, Z_N)$, $Z' = (Z'_Y, Z'_N)$, $Z_Y, Z'_Y \in \real^{2m}$, $Z_N, Z'_N \in \real^{2(n-m)}$, $|Z|, |Z'| \leq \epsilon$, $\alpha, \alpha' \in \nat^{2n}$, $|\alpha|+|\alpha'| \leq l$, we have
		\begin{multline}\label{eq_berg_perp_off_diagk}
			\bigg| 
				\frac{\partial^{|\alpha|+|\alpha'|}}{\partial Z^{\alpha} \partial Z'{}^{\alpha'}}
				\bigg(
					\frac{1}{p^n} B_{k, p}^{\perp} \big(\psi_{y_0}(Z), \psi_{y_0}(Z') \big)
					\\
					-
					\sum_{j = 0}^{r}
					p^{-\frac{j}{2}}						
					F_{k, j}^{\perp}(\sqrt{p} Z, \sqrt{p} Z') 
					\kappa_{\psi}^{X|Y}(Z)^{-\frac{1}{2}}
					\kappa_{\psi}^{X|Y}(Z')^{-\frac{1}{2}}
				\bigg)
			\bigg|_{\ccal^{l'}}
			\\
			\leq
			C p^{- \frac{r + 1 - l}{2}}
			\cdot
			\Big(1 + \sqrt{p}|Z| + \sqrt{p} |Z'| \Big)^{Q}
			\exp\Big(- c \sqrt{p} \big( |Z_Y - Z'_Y| + |Z_N| + |Z'_N| \big) \Big),
		\end{multline}
		where the $\ccal^{l'}$-norm is taken with respect to $y_0$.
		Also, we have
		\begin{equation}\label{eq_jopep_0}
			J_{k, 0}^{\perp}(Z, Z')
			=
			\pi^{k}
			\cdot
			\sum_{\substack{\beta \in \nat^{n - m} \\ |\beta| = k}}
			\frac{ z_N^{\beta} \cdot (\overline{z}'_N)^{\beta} }{\beta!}.
		\end{equation}
		Moreover, under the assumption (\ref{eq_comp_vol_omeg}), we have
		\begin{multline}\label{eq_jopep_1}
			J_{k, 1}^{\perp}(Z, Z') = 
			\pi^{k + 1}
			\cdot
			{\rm{Id}}_{F_{y_0}} \cdot 
			\Big(
			 	g \big(z_N, A(\overline{z}_Y - \overline{z}'_Y) (\overline{z}_Y - \overline{z}'_Y) \big)
			 	\\
			 	+
			 	g \big(\overline{z}'_N, A(z_Y - z'_Y) (z_Y - z'_Y) \big)			 
			\Big)
		 	\cdot
			\sum_{\substack{\beta \in \nat^{n - m} \\ |\beta| = k}}
			\frac{ z_N^{\beta} \cdot (\overline{z}'_N)^{\beta} }{\beta!}.
		\end{multline}
	\end{thm}
	\begin{rem}\label{rem_bpper_refin}
		a) In particular, from (\ref{eq_pperp_defn_fun}), we have $F_{k, 0}^{\perp}(Z, Z') = \mathscr{P}_{n, m}^{\perp, k}$ and 
		\begin{multline}
			F_{k, 1}^{\perp}(Z, Z') = 
			\pi^{k + 1}
			\cdot
			{\rm{Id}}_{F_{y_0}} \cdot 
			\Big(
			 	g \big(z_N, A(\overline{z}_Y - \overline{z}'_Y) (\overline{z}_Y - \overline{z}'_Y) \big)
			 	\\
			 	+
			 	g \big(\overline{z}'_N, A(z_Y - z'_Y) (z_Y - z'_Y) \big)			 
			\Big)
		 	\cdot
			\mathscr{P}_{n, m}^{\perp, k}.
		\end{multline}
		\par 
		b) Theorems \ref{thm_ext_as_expk} and \ref{thm_berg_perp_off_diagk} were established for $k = 0$ in \cite[Theorems 1.6, 1.8]{FinOTAs}.
		Our proof here is different from the one from \cite{FinOTAs}, even for $k = 0$.
		\par 
		c) Our methods allow to give a precise estimate for $Q$ from Theorems \ref{thm_ext_as_expk} and \ref{thm_berg_perp_off_diagk}, but as the derivation is quite lengthy and cumbersome, and we never use those estimates, we leave their derivation to the interested reader.
		\par 
		d)
		By Taylor expansion due to the obvious vanishing properties of the Schwartz kernels, instead of the estimate $p^{- \frac{r + 1 - l}{2}}$ on the right-hand side of (\ref{eq_ext_as_expk}) (resp. of (\ref{eq_berg_perp_off_diagk})) one can put $p^{- \frac{r + 1 - l - k}{2}} |Z_N|^k$ (resp. $p^{- \frac{r + 1 - l - 2k}{2}} |Z_N \cdot Z'_N|^k$).
	\end{rem}
	\begin{proof}[Proof of Theorem \ref{thm_high_term_ext} assuming Theorems \ref{thm_ext_exp_dc}.a) and \ref{thm_ext_as_expk}]
		The main idea of the proof is to compare the Schwartz kernels of $\ext_{k, p}$ and $\ext_{k, p}^{0}$.
		For this, we introduce $K_p := \ext_{k, p} - \ext_{k, p}^{0}$. 
		From Theorems \ref{thm_bk_off_diag}, \ref{thm_ext_exp_dc}.a) and (\ref{eq_ext0_op}), we conclude that there are $c > 0$, $p_1 \in \nat^*$,  such that for any $r \in \nat$, there is $C > 0$, such that for any $p \geq p_1$, $x \in X$, $y \in Y$, the following estimate holds
		\begin{equation}\label{eq_ext_kp}
			\Big|  K_p(x, y) \Big| \leq C p^{m - \frac{k}{2}} \exp \big(- c \sqrt{p} \dist(x, y) \big).
		\end{equation}
		\par 
		Now, by studying the asymptotic expansion of the Schwartz kernel $K_p(x, y)$, we will improve the estimate (\ref{eq_ext_kp}) by lowering the degree of $p$.
		For $r \in \nat$, let us expand $\kappa_{N}(\psi(Z))^{\frac{1}{2}}$ in a neighborhood of $Z = 0$ as follows
		\begin{equation}\label{eq_kappa_N_tay_exp}
			\kappa_{N}(\psi(Z))^{\frac{1}{2}} = \sum_{i = 0}^{r} \kappa_{N, [i]}^{\frac{1}{2}}(Z) + O(|Z|^{r + 1}),
		\end{equation}
		where $\kappa_{N, [i]}^{\frac{1}{2}}(Z)$ are homogeneous polynomials of degree $i$.
		We denote by $J^{Y|Y}_{k, r}$ the polynomials from Theorem \ref{thm_berg_off_diag} associated to $X = Y$ and $F = {\rm{Sym}}^k (N^{1, 0})^* \otimes \iota^* F$.
		From Theorems \ref{thm_berg_off_diag}, \ref{thm_ext_as_expk} and (\ref{eq_pperp_defn_fu2n}), we deduce that for polynomials $J_{r, K}^E(Z, Z'_Y)$, defined for $r \in \nat$ as follows
		\begin{multline}\label{eq_jrke}
			J_{r, K}^E(Z, Z'_Y) := J_{k, r}^E(Z, Z'_Y) -  \sum_{\substack{\beta \in \nat^{n - m} \\ |\beta| = r}}
			\frac{1}{\beta!} \cdot
			z_N^{\beta} \cdot \big( \frac{\partial}{\partial z'_N} \big)^{\odot \beta}
			\cdot
			\\
			\cdot
			\sum_{a + b = r}  \kappa_{N, [a]}^{\frac{1}{2}}(Z) \cdot J^{Y|Y}_{k, b}(Z_Y, Z'_Y),
		\end{multline}
		and the functions $F_{r, K}^E := J_{r, K}^E \cdot \mathscr{E}_{n, m}^{0}$ over $\real^{2n} \times \real^{2m}$, the following holds.
		There are $\epsilon, c > 0$, $p_1 \in \nat^*$, such that for any $r \in \nat$, there are $C, Q > 0$, such that for any $y_0 \in Y$, $p \geq p_1$, $Z = (Z_Y, Z_N)$, $Z_Y, Z'_Y \in \real^{2m}$, $Z_N \in \real^{2(n - m)}$, $|Z|, |Z'_Y| \leq \epsilon$, we have
		\begin{multline}\label{eq_kp_as_exp}
			\bigg| 
					\frac{1}{p^{m - \frac{k}{2}}} K_p \big(\psi_{y_0}(Z), \phi_{y_0}^Y(Z'_Y) \big)
					-
					\sum_{j = 0}^{r}
					p^{-\frac{j}{2}}						
					F_{j, K}^E(\sqrt{p} Z, \sqrt{p} Z'_Y) 
					\kappa_{\psi}^{X|Y}(Z)^{-\frac{1}{2}}
					\kappa_{\phi}^Y(Z'_Y)^{-\frac{1}{2}}
			\bigg|
			\\
			\leq
			C p^{- \frac{r + 1}{2}}
			\Big(1 + \sqrt{p}|Z| + \sqrt{p} |Z'_Y| \Big)^{Q}
			\exp\Big(- c \sqrt{p} \big( |Z_Y - Z'_Y| + |Z_N| \big) \Big).
		\end{multline}
		From (\ref{eq_jo_expl_form}), (\ref{eq_je0_expk}) and (\ref{eq_jrke}), we deduce 
		\begin{equation}\label{eq_jk0_exp}
			J_{0, K}^E(Z, Z'_Y) = 0.
		\end{equation}
		In particular, from (\ref{eq_kp_as_exp}) and (\ref{eq_jk0_exp}), we see that we can improve (\ref{eq_ext_kp}) as follows
		\begin{equation}\label{eq_ext_kp_impro}
			\Big|  K_p(x, y) \Big| \leq C p^{m - \frac{1 + k}{2}} \exp \big(- c \sqrt{p} \dist(x, y) \big).
		\end{equation} 
		From Proposition \ref{prop_norm_bnd_distk_expbnd} and (\ref{eq_ext_kp_impro}), we deduce that there are $C > 0$, $p_1 \in \nat^*$, such that for any $p \geq p_1$, we have  $\| K_p \| \leq \frac{C}{p^{\frac{n - m + k + 1}{2}}}$, which is exactly (\ref{eq_ext_as}) by the definition of $K_p$.
		\par 
		Now, similarly to the derivation of (\ref{eq_jk0_exp}), we see that under assumption (\ref{eq_comp_vol_omeg}), from (\ref{eq_j1_expl_form}), (\ref{eq_jrke}) and the fact from \cite[(5.35)]{FinOTAs}, stating 
		\begin{equation}\label{eq_kappa_der_norm}
			\frac{\partial}{\partial Z_N} \kappa_N = 0,
		\end{equation}
		we deduce that we have
		\begin{equation}\label{eq_jk1_exp}
			J_{1, K}^E(Z, Z'_Y) =  J_{k, 1}^E(Z, Z'_Y).
		\end{equation}
		\par 
		Let us now prove (\ref{eq_norm_asymp}). Recall the following calculation: for $\beta \in \nat^{n - m}$, $|\beta| = k$, we have
		\begin{equation}
			\int_{\real^{2(n - m)}} |z_N^{2 \beta}| \cdot  \exp (- \pi |Z_N|^2 ) \cdot dZ_{2m + 1} \wedge \cdots \wedge dZ_{2n}
			=
			\frac{\beta!}{\pi^k}
			.
		\end{equation}
		Remark also from (\ref{eq_sym_emb_tens}) and the description below (\ref{eq_pairing_sym}) that 
		\begin{equation}
			\| dz_N^{\odot \beta} \|^2 = 2^k \cdot \frac{\beta!}{k!}.
		\end{equation}
		In particular, we have
		\begin{multline}\label{eq_int_kapp_betta_etc}
			\int_{\real^{2(n - m)}} \kappa_N(y, \sqrt{p} Z_N) \cdot \frac{|z_N^{2 \beta}|}{k!^2} \cdot \exp (- p \pi |Z_N|^2 ) \rho \Big(  \frac{|Z_N|}{r_{\perp}} \Big)^2 dZ_{2m + 1} \wedge \cdots \wedge dZ_{2n}
			\\
			=
			\frac{\kappa_N(y)}{p^{n - m + k}} \cdot \frac{1}{k! \cdot (2 \pi)^{k}} \cdot \| dz_N^{\odot \beta} \|^2   + O \Big( \frac{1}{p^{n - m + k + \frac{1}{2}}} \Big).
		\end{multline}
		\par 
		An easy calculation, using (\ref{eq_kappan}) and (\ref{eq_int_kapp_betta_etc}), cf. \cite[(4.38)]{FinOTAs} for the case $k = 0$, shows that for any $g \in L^2(Y, {\rm{Sym}}^k (N^{1, 0})^* \otimes \iota^*(L^p \otimes F))$, we have
		\begin{equation}\label{eq_norm_1stterm_er}
			\big\| \ext_{k, p}^{0} g \big\|_{L^2(dv_X)}^2
			=
			\frac{1}{p^{n - m + k}} \cdot \frac{1}{k! \cdot (2 \pi)^k} \cdot
			\Big\| \kappa_N(y)^{\frac{1}{2}} \cdot B_{k, p}^Y g \Big\|_{L^2(Y)}^2
			+
			O
			\Big(
			\frac{\| g \|_{L^2(Y)}^2}{p^{n - m + k + \frac{1}{2}}}
			\Big)
			.
		\end{equation}
		\par 
		\begin{sloppypar}
		Consider the Toeplitz operator $T_{\kappa_N, p}^Y \in \enmr{ L^2(Y, {\rm{Sym}}^k (N^{1, 0})^* \otimes \iota^*( L^p \otimes F))}$, given by
		\begin{equation}
			T_{\kappa_N, p}^Y g := B_{k, p}^Y ( \kappa_N \cdot B_{k, p}^Y g).
		\end{equation}
		Then, we clearly have
		\begin{equation}\label{eq_reextprs_toep}
			\big\langle T_{\kappa_N, p}^Y g,  g  \big\rangle_{L^2(Y)}
			=
			\big\langle \kappa_N \cdot B_{k, p}^Y g,  B_{k, p}^Y  g  \big\rangle_{L^2(Y)}.
		\end{equation}
		Thus, by (\ref{eq_norm_1stterm_er}) and (\ref{eq_reextprs_toep}), we have 
		\begin{equation}\label{eq_top_final_1}
			\big\| \ext_{k, p}^{0} \big\|
			=
			\frac{1}{p^{\frac{n - m + k}{2}}}
			\cdot
			\frac{1}{\sqrt{k! \cdot (2 \pi)^k}} 
			\cdot
			\big\| T_{\kappa_N, p}^Y \big\|^{\frac{1}{2}} + O \Big( \frac{1}{p^{\frac{n - m + k + 1}{2}}} \Big).
		\end{equation}
		From Lemma \ref{lem_norm_toepl} and (\ref{eq_top_final_1}), we deduce (\ref{eq_norm_asymp}).
		\end{sloppypar}
		\par 
		Now it is only left to prove that if $A \neq 0$, then under additional assumption (\ref{eq_comp_vol_omeg}), one can not replace $p^{-\frac{n - m + k + 1}{2}}$ by an asymptotically better estimate.
		For this, remark that as long as $A \neq 0$, by (\ref{eq_f1e_form}), the operator, acting on $\comp^n$ with the kernel $F_{k, 1}^E(Z, Z'_Y)$, has non-zero norm.
		Then, by the calculations, similar to (\ref{eq_norm_1stterm_er}), we see that the operator, acting on $\comp^n$ with the kernel $F_{k, 1}^E(\sqrt{p} Z, \sqrt{p} Z'_Y)$, has norm of order $p^{-\frac{n - m + k}{2}}$, as $p \to \infty$.
		We deduce from this and Theorem \ref{thm_ext_as_expk} the needed result.
	\end{proof}
	\begin{proof}[Proof of Theorem \ref{thm_log_bk} assuming Theorems \ref{thm_ext_exp_dc}.b) and \ref{thm_berg_perp_off_diagk}]
		First of all, remark that by definition
		\begin{equation}\label{eq_berg_k_logbk}
			B_p^{X} - B_p^{X, k Y}
			=
			\sum_{l = 0}^{k - 1} B_{l, p}^{\perp}.
		\end{equation}
		Clearly, now the estimate (\ref{eq_ext_log_bk}) follows from Theorems \ref{thm_bk_off_diag}, \ref{thm_ext_exp_dc}.b) and (\ref{eq_berg_k_logbk}).
		The statement (\ref{eq_ext_log_bk2}) follows from Theorems \ref{thm_berg_off_diag}, \ref{thm_berg_perp_off_diagk},  the use of Remark \ref{rem_bpper_refin}.d) and (\ref{eq_berg_k_logbk}) by using considerations similar to the ones from (\ref{eq_jrke}) and (\ref{eq_jk0_exp}). 
	\end{proof}

\subsection{Multiplicative defect as a bridge from Bergman projections to extension operators}\label{sect_adj_t_oper}

	The main goal of this section is to prove Theorems \ref{thm_ext_exp_dc}, \ref{thm_ext_as_expk} and \ref{thm_berg_perp_off_diagk}.
	As a consequence of our method of the proof, we also establish Theorem \ref{thm_isom}.
	To do so, we rely heavily on the existence of the multiplicative defect operator, $A_{k, p} \in \enmr{H^{0}_{(2)}(Y, {\rm{Sym}}^k (N^{1, 0})^* \otimes \iota^* ( L^p \otimes F))}$, $p \geq p_1$, established in Theorem \ref{thm_ex_ap}.
	More precisely, the following statement plays a crucial role in our approach.
	\begin{thm}\label{thm_mult_def_toepl_k}
		The sequence of operators $\frac{1}{p^{n -m + k}} A_{k, p}$, $p \geq p_1$, viewed as a sequence of elements from $\enmr{L^2(Y, {\rm{Sym}}^k (N^{1, 0})^* \otimes \iota^* (L^p \otimes F))}$ by precomposing with the Bergman projector $B_{k, p}^Y$, forms a Toeplitz operator with weak exponential decay with respect to $X$.
		Moreover, the first term of this asymptotic expansion can be calculated as follows
		\begin{equation}\label{eq_akp_first_term}
			\Big[ \frac{1}{p^{n -m + k}} A_{k, p} \Big]_0 
			=
			(2 \pi)^k \cdot k! \cdot \kappa_N^{-1}|_Y \cdot {\rm{Id}}_{{\rm{Sym}}^k N^{1, 0}_{y_0}} \otimes {\rm{Id}}_{F_{y_0}}.
		\end{equation}
	\end{thm}
	\begin{rem}
		For $k = 0$, Theorem \ref{thm_mult_def_toepl_k} was established in \cite[Theorem 4.3]{FinToeplImm}.
	\end{rem}
	\par 
	The core of our argument in the proofs of Theorems \ref{thm_ext_exp_dc}, \ref{thm_ext_as_expk} and \ref{thm_berg_perp_off_diagk} lies in the following inductive step.
	We fix $k_0 \in \nat$.
	\begin{lem}\label{lem_ind_step}
		Assume that for any $k < k_0$, Theorems \ref{thm_ext_exp_dc}.b) and \ref{thm_berg_perp_off_diagk} hold.
		Then for $k := k_0$,
		\begin{tasks}[style=enumerate](2)
			\task Theorem \ref{thm_mult_def_toepl_k} holds,
			\task Theorems \ref{thm_ext_exp_dc}.a) and \ref{thm_ext_as_expk} hold,
			\task Theorems \ref{thm_ext_exp_dc}.b) and \ref{thm_berg_perp_off_diagk} hold.
		\end{tasks}
	\end{lem}
	The following statement will be useful in our proof of Lemma \ref{lem_ind_step}.
	\begin{lem}\label{lem_ext_as_expk_new}
		Assume that for any $k < k_0$, Theorems \ref{thm_ext_exp_dc}.b) and \ref{thm_berg_perp_off_diagk} hold.
		Then for $k := k_0$, and any $r \in \nat$, $y_0 \in Y$, there are polynomials $J_{k, r}^{R, 0}(Z_Y, Z') \in {\rm{Sym}}^k (N_{y_0}^{1, 0})^* \otimes \enmr{F_{y_0}}$ in $Z' \in \real^{2n}$, $Z_Y \in \real^{2m}$, with the same properties as those from Theorem \ref{thm_ext_as_expk} (including the analogous vanishing, degree and parity requirements), such that for $F_{k, r}^{R, 0} := J_{k, r}^{R, 0} \cdot \mathscr{Res}_{n, m}^{0}$, the following holds.
		\par 
		There are $\epsilon, c > 0$, $p_1 \in \nat^*$, such that for any $r, l, l' \in \nat$, there are $C, Q  > 0$, such that for any $y_0 \in Y$, $p \geq p_1$, $Z = (Z_Y, Z_N)$, $Z_Y, Z'_Y \in \real^{2m}$, $Z_N \in \real^{2(n - m)}$, $|Z|, |Z'_Y| \leq \epsilon$, $\alpha \in \nat^{2n}$, $\alpha' \in \nat^{2m}$, $|\alpha| + |\alpha'| \leq l$, the following bound holds
		\begin{multline}\label{eq_ext_as_expk_new}
			\bigg| 
				\frac{\partial^{|\alpha|+|\alpha'|}}{\partial Z^{\alpha} \partial Z'_Y{}^{\alpha'}}
				\bigg(
					\frac{1}{p^{n + \frac{k}{2}}} \cdot (\res_Y \circ \nabla^k B_p^{X, k Y})  \big(\phi_{y_0}^Y(Z_Y), \psi_{y_0}(Z') \big)
					\\
					-
					\sum_{i = 0}^{r}
					p^{-\frac{i}{2}}						
					F_{k, i}^{R, 0}(\sqrt{p} Z_Y, \sqrt{p} Z') 
					\kappa_{\phi}^Y(Z_Y)^{-\frac{1}{2}}
					\kappa_{\psi}^{X|Y}(Z')^{-\frac{1}{2}}
				\bigg)
			\bigg|_{\ccal^{l'}}
			\\
			\leq
			C p^{- \frac{k + 1 - l}{2}}
			\Big(1 + \sqrt{p}|Z_Y| + \sqrt{p} |Z'| \Big)^{Q}
			\exp\Big(- c \sqrt{p} \big( |Z_Y - Z'_Y| + |Z'_N| \big) \Big),
		\end{multline}
		where the $\ccal^{l'}$-norm is taken with respect to $y_0$.
		Also, the following identity holds
		\begin{equation}\label{eq_jr000_expk}
			J_{k, 0}^{R, 0}(Z_Y, Z') = 
			\pi^{k}
			\cdot
			\kappa_N^{-\frac{1}{2}}(y_0)
			\cdot
			\sum_{\substack{\beta \in \nat^{n - m} \\ |\beta| = k}}
			\frac{k!}{\beta!} \cdot
			 (d z_N)^{\odot \beta} \cdot (\overline{z}'_N)^{\beta}.
		\end{equation}
		In particular, from (\ref{eq_res_nm_kernel}), we have $F_{k, 0}^{R, 0}(Z_Y, Z') = \mathscr{Res}_{n, m}^{k}$.
		Moreover, under the assumption (\ref{eq_comp_vol_omeg})
		\begin{multline}\label{eq_jr000111_expk}
			J_{k, 1}^{R, 0}(Z_Y, Z') = 
			{\rm{Id}}_{F_{y_0}} \cdot \pi^{k + 1}
		 	g \big(\overline{z}'_N, A(z_Y - z'_Y) (z_Y - z'_Y) \big)			 
			\cdot
			\\
			\cdot
			\sum_{\substack{\beta \in \nat^{n - m} \\ |\beta| = k}}
			\frac{k!}{\beta!} \cdot
			 (d z_N)^{\odot \beta} \cdot (\overline{z}'_N)^{\beta}.
		\end{multline}
 	\end{lem}
 	\begin{proof}
 		We define the $Z_Y$-polynomials $\nabla^l \kappa_{N, [i]}^{-  \frac{1}{2}}(Z)|_{Z_N = 0}$ for $l \in \nat$ similarly to (\ref{eq_kappa_N_tay_exp}).
		From Section \ref{sect_model_calc}, remark that the trivial identities 
		\begin{equation}\label{eq_trrrr_id}
		\begin{aligned}
			&
			\mathscr{P}_n(Z, Z') = \mathscr{P}_m(Z_Y, Z'_Y) \cdot \mathscr{P}_{n - m}(Z_N, Z'_N),
			\\
			&
			\mathscr{P}_{n, m}^{\perp, 0}(Z, Z') = \mathscr{P}_m(Z_Y, Z'_Y) \cdot \mathscr{P}_{n - m, 0}^{\perp, 0}(Z_N, Z'_N).
		\end{aligned}
		\end{equation}
		\begin{sloppypar}
		From the fact that in the notations introduced in the end of Section \ref{sect_sff}, we have $\res_Y( \tilde{f}_1^{X|Y}, \ldots, \tilde{f}_r^{X|Y}) = \tilde{f}'_1{}^{Y}, \ldots, \tilde{f}'_r{}^{Y}$, the fact that $\frac{\partial \psi}{\partial Z_j}|_{Z_N = 0}$, $j = 2m + 1, \ldots, 2n$, are parallel along $B_0^{\real^{2m}}(r_Y) \subset \real^{2n}$, (\ref{eq_kappa_relation}), (\ref{eq_berg_off_diag}), (\ref{eq_berg_perp_off_diagk}), (\ref{eq_berg_k_logbk}) and (\ref{eq_trrrr_id}), we see that the expansion (\ref{eq_ext_as_expk_new}) holds for the polynomials $J_{k, r}^{R, 0}(Z_Y, Z')$, $Z_Y \in \real^{2m}$, $Z' \in \real^{2n}$, defined as follows
		\begin{multline}\label{eq_defn_jra22}
			J_{k, r}^{R, 0}(Z_Y, Z')
			:=
			\sum_{a + b + d = r} \sum_{c + d = k} 
			\nabla^c \Big( J_{a}^{X|Y}(Z, Z') \cdot \frac{\mathscr{P}_{n - m}(Z_N, Z'_N)}{\mathscr{P}_{n-m}^{\perp, 0}(0, Z'_N)}
			\\
			-
			\sum_{i = 0}^{k - 1} J_{i, a}^{\perp}(Z, Z') \cdot \mathscr{P}_{n-m}^{\perp, 0}(Z_N, 0)
			\Big) 
			\cdot 
			(\nabla^d \kappa_{N, [b]}^{-\frac{1}{2}})(Z)^{-\frac{1}{2}} \Big|_{Z_N = 0}.
		\end{multline}
		In particular, from (\ref{eq_defn_jra22}), we obtain that 
		\begin{multline}\label{eq_defn_jra2232}
			J_{k, 0}^{R, 0}(Z_Y, Z')
			=
			\kappa_N^{- \frac{1}{2}}(y_0)
			\nabla^k \Big( J_{0}^{X|Y}(Z, Z') \cdot \frac{\mathscr{P}_{n - m}(Z_N, Z'_N)}{\mathscr{P}_{n-m}^{\perp, 0}(0, Z'_N)}
			\\
			-
			\sum_{i = 0}^{k - 1} J_{i, 0}^{\perp}(Z, Z') \cdot \mathscr{P}_{n-m}^{\perp, 0}(Z_N, 0)
			\Big) \Big|_{Z_N = 0}.
		\end{multline}
		From this, (\ref{eq_jo_expl_form}), (\ref{eq_jopep_0}) and the calculation similar to the one used in (\ref{eq_res_nm_kernel}), we deduce (\ref{eq_jr000_expk}).
		Moreover, under the assumption (\ref{eq_comp_vol_omeg}), from (\ref{eq_kappa_der_norm}) and (\ref{eq_defn_jra22}), we also obtain that 
		\begin{multline}\label{eq_defn_jra22324}
			J_{k, 1}^{R, 0}(Z_Y, Z')
			=
			\kappa_N^{- \frac{1}{2}}(y_0)
			\nabla^k \Big( J_{1}^{X|Y}(Z, Z') \cdot \frac{\mathscr{P}_{n - m}(Z_N, Z'_N)}{\mathscr{P}_{n-m}^{\perp, 0}(0, Z'_N)}
			\\
			-
			\sum_{i = 0}^{k - 1} J_{i, 1}^{\perp}(Z, Z') \cdot \mathscr{P}_{n-m}^{\perp, 0}(Z_N, 0)
			\Big) \Big|_{Z_N = 0}.
		\end{multline}
		From this, as in the proof of (\ref{eq_jr000_expk}), but now using (\ref{eq_j1_expl_form}) and (\ref{eq_jopep_1}), we deduce (\ref{eq_jr000111_expk}).
		\end{sloppypar}
		\par 
		The fact that the parity of $J_{k, r}^{R, 0}$ coincides with the parity of $k + r$ follows from the analogous statements from Theorems \ref{thm_berg_off_diag}, \ref{thm_berg_perp_off_diagk} and (\ref{eq_defn_jra22}).
		From  the bounded geometry assumption and the boundness properties of the coefficients of $J_{a}^{X|Y}$, $J_{i, a}^{\perp}$ from Theorems \ref{thm_berg_off_diag}, \ref{thm_berg_perp_off_diagk}, we see that the coefficients of $J_{k, r}^{R, 0}$ are bounded with all their derivatives.
		From (\ref{eq_defn_jra22}), we see that
		\begin{equation}\label{eq_deg_bound_jkrr0}
			\deg (J_{k, r}^{R, 0}) 
			\leq
			\max 
			\Big\{ 
			 \deg(J_{a}^{X|Y}) + k,
			 \deg(J_{i, a}^{\perp}) + k
			\Big\},
		\end{equation}
		where the maximum is taken over $a = 0, \ldots, k$ and $i = 0, \ldots, k - 1$.
		From (\ref{eq_deg_bound_jkrr0}) and the corresponding bounds on the degrees of $J_{a}^{X|Y}$ and $J_{i, a}^{\perp}$ from Theorems \ref{thm_berg_off_diag} and \ref{thm_berg_perp_off_diagk}, we deduce the needed bound on the degree of $J_{k, r}^{R, 0}$.
		\par 
		It is only left to establish the vanishing of $J_{k, r}^{R, 0}$ at least up to order $k$ along $\real^{2m} \times \real^{2m} \subset \real^{2m} \times \real^{2n}$.
		For this, let us define the polynomials $J_{k, r}^{E, 0}(Z, Z'_Y) \in {\rm{Sym}}^k N_{y_0}^{1, 0} \otimes \enmr{F_{y_0}}$ in $Z \in \real^{2n}$, $Z'_Y \in \real^{2m}$ as follows
		\begin{equation}\label{eq_defn_jkre0}
			J_{k, r}^{E, 0}(Z, Z'_Y) := (J_{k, r}^{R, 0}(Z'_Y, Z))^*.
		\end{equation}
		Clearly, it is enough to establish the analogous vanishing property for the polynomials $J_{k, r}^{E, 0}$.
		Assume that $r_0 \in \nat$ is the first index, for which this vanishing property doesn't hold.
		We denote by $k' < k$ the order of vanishing of $J_{k, r_0}^{E, 0}$ along $\real^{2m} \times \real^{2m} \subset \real^{2n} \times \real^{2m}$.
		\par 
		Let us consider the operator $T_p^Y : L^2(Y, {\rm{Sym}}^k (N^{1, 0})^* \otimes \iota^*( L^p \otimes F)) \to L^2(Y, {\rm{Sym}}^{k'} (N^{1, 0})^* \otimes \iota^*( L^p \otimes F))$, given by the following formula
		\begin{equation}\label{eq_tpy_exp}
			T_p^Y := p^{m - n - \frac{k + k'}{2}} \cdot \res_Y \circ \nabla^{k'} (\res_Y \circ \nabla^k B_p^{X, k Y})^*.
		\end{equation}
		\par 
		We will now establish that the sequence of operators $T_p^Y$ satisfies the assumptions of Proposition \ref{prop_ma_mar_crit_exp_dec} or Remark \ref{rem_toepl_f1f2_par}, depending on the parity of $k + k'$.
		Clearly, the first property from Proposition \ref{prop_ma_mar_crit_exp_dec} follows from Theorem \ref{thm_res_is_l2} and (\ref{eq_tpy_exp}). 
		\par 
		From Theorem \ref{thm_bk_off_diag} and Theorem \ref{thm_ext_exp_dc}.b), applied for $k \leq k_0 - 1$, we deduce that there are $c > 0$, $p_1 \in \nat^*$,  such that for any $r \in \nat$, there is $C > 0$, such that for any $p \geq p_1$, $x_1, x_2 \in X$, the following estimate holds
		\begin{equation}\label{eq_exp_dc_akp11}
			\Big|  B_p^{X, k Y}(x_1, x_2) \Big|_{\ccal^r} \leq C p^{n + \frac{r}{2}} \exp \big(- c \sqrt{p} \cdot \dist(x_1, x_2) \big).
		\end{equation}		
		From (\ref{eq_exp_dc_akp11}), we deduce that there are $c > 0$, $p_1 \in \nat^*$,  such that for any $r \in \nat$, there is $C > 0$, such that for any $p \geq p_1$, $y_1, y_2 \in Y$, the following estimate holds
		\begin{equation}\label{eq_exp_dc_akp}
			\Big|  T_p^Y(y_1, y_2) \Big|_{\ccal^r} \leq C p^{m + \frac{r}{2}} \exp \big(- c \sqrt{p} \cdot \dist_X(y_1, y_2) \big).
		\end{equation}		
		This implies the second property from Proposition \ref{prop_ma_mar_crit_exp_dec}  with respect to $X$.
		\par 
		We will now show that the third property is a direct consequence of Theorems \ref{thm_berg_off_diag} and \ref{thm_berg_perp_off_diagk} for $k \leq k_0 - 1$.
		From the reasoning similar to (\ref{eq_defn_jra22}), we see that the expansion (\ref{eq_tpy_defn_exp_tay12as}) for $T_p^Y$ as above holds for the polynomials $I_r^Y(Z_Y, Z'_Y)$, $Z_Y, Z'_Y \in \real^{2m}$, defined as follows
		\begin{multline}\label{eq_I_r_y_form_for_ext}
			I_r^Y(Z_Y, Z'_Y) =
			\sum_{a + b + d = r}
			\sum_{c + d = k'} 
			\nabla^c \Big(  J_{k, a}^{E, 0}(Z, Z'_Y) \cdot \mathscr{P}_{n - m}(Z_N, 0) 
			\Big)
			\cdot
			\\
			\cdot 
			(\nabla^d \kappa_{N, [b]}^{-\frac{1}{2}})(Z)^{-\frac{1}{2}} \Big|_{Z_N = 0}.
		\end{multline}
		From the parity statements on $J_{k, r}^{R, 0}$ from above (implying the corresponding statements for $J_{k, r}^{E, 0}$), we see that the parity of $I_r^Y$ is equal to the parity of $r + k + k'$.
		From (\ref{eq_I_r_y_form_for_ext}), by our choice of $r_0$, we also see that $I_r^Y(Z_Y, Z'_Y) = 0$ for $r < r_0$ and $I_{r_0}^Y(Z_Y, Z'_Y) = \nabla^{k'}_{z_N} J_{k, r_0}^{E, 0}(Z, Z'_Y)|_{Z_N = 0}$.
		Hence the sequence of operators $p^{\frac{r_0}{2}} \cdot T_p^Y$ still satisfies the assumptions of either Proposition \ref{prop_ma_mar_crit_exp_dec} or Remark \ref{rem_toepl_f1f2_par}, and the first term of the associated asymptotic expansion (\ref{eq_tpy_defn_exp_tay12as}) holds with nonzero term by our choice of $k'$ and $r_0$.
		This, however, goes in contradiction with the last part of Proposition \ref{prop_ma_mar_crit_exp_dec} or Remark \ref{rem_toepl_f1f2_par} and the trivial fact that $T_p^Y = 0$.
 	\end{proof}

	\begin{proof}[Proof of Lemma \ref{lem_ind_step}.a)]
		Let us verify that the sequence of operators $\frac{1}{p^{n -m + k}} A_{k, p}$, $p \geq p_1$, for $k := k_0$, satisfies all the properties of Theorem \ref{thm_ma_mar_crit_exp_dec} once we know the validity of Theorems \ref{thm_ext_exp_dc}.b) and \ref{thm_berg_perp_off_diagk} for any $k < k_0$.
		\par 
		In fact, from (\ref{eq_resp_ap_lem}) and the fact that for $p \geq p_1$, we have $\res_Y \circ \nabla^k \ext_{k, p} = B_{k, p}^Y$, we obtain the explicit formula
		\begin{equation}\label{eq_ap_form}
			A_{k, p} = \res_Y \circ \nabla^k  (\res_Y \circ \nabla^k B_p^{X, k Y})^{*}.
		\end{equation}
		From the above identity, we see that $\frac{1}{p^{n -m + k}} A_{k, p}$ coincides with the operator $T_p^Y$, considered in (\ref{eq_tpy_exp}) for $k' = k$.
		By the proof of Lemma \ref{lem_ext_as_expk_new}, we see that the sequence of operators $\frac{1}{p^{n -m + k}} A_{k, p}$, $p \geq p_1$, for $k = k_0$, forms a Toeplitz operator with weak exponential decay with respect to $X$.
		\par 
		From (\ref{eq_jr000_expk}), (\ref{eq_defn_jkre0}) and the calculation similar to (\ref{eq_res_enmdual}), we obtain that
		\begin{equation}\label{eq_jk0e0form_ex}
			J_{k, 0}^{E, 0}(Z, Z'_Y) 
			=
			(2 \pi)^k \cdot \kappa_N^{-\frac{1}{2}}(y_0) \cdot
			\sum_{\substack{\beta \in \nat^{n - m} \\ |\beta| = k}}
			\frac{1}{\beta!} \cdot
			 z_N^{\beta} \cdot \big( \frac{\partial}{\partial z_N} \big)^{\odot \beta}.
		\end{equation}
		From this, (\ref{eq_id_sym_ident}), (\ref{eq_jo_expl_form}) and (\ref{eq_I_r_y_form_for_ext}), similarly to (\ref{eq_verif_res_ext}), we deduce that
		\begin{equation}\label{eq_jo_expl_form12121}
			I_0^Y(Z_Y, Z'_Y) = (2 \pi)^k \cdot k! \cdot \kappa_N^{-1}(y_0) \cdot {\rm{Id}}_{{\rm{Sym}}^k (N^{1, 0}_{y_0})^*} \otimes {\rm{Id}}_{F_{y_0}}.
		\end{equation}
		From the last statement of Theorem \ref{thm_ma_mar_crit_exp_dec}, we deduce (\ref{eq_akp_first_term}), which finishes our proof.
	\end{proof}
	In the proof of Lemma \ref{lem_ind_step}.b), we will need to express the extension operator in terms of the orthogonal Bergman projector. 
	For this, we need to invert the operators $A_{k, p}$.
	The following result gives a sufficient condition for inverting Toeplitz operators with weak exponential decay.
	\begin{lem}[{\cite[Lemma 4.5]{FinToeplImm}}]\label{lem_inverse_toepl}
		Assume that a sequence of operators $G_p$, $p \in \nat$, forms a Toeplitz operator with weak exponential decay with respect to a manifold $Z$ in the notations from Definition \ref{defn_ttype}.
		Assume that for $f := [G_p]_0$, and any $y \in Y$, the element $f(y)$ is invertible and $f^{-1} \in \ccal^{\infty}_{b}(Y, \enmr{\iota^* F})$.
		Then there is $p_1 \in \nat$, such that for $p \geq p_1$, the operators $G_p$ are invertible.
		Moreover, the sequence of operators $G_p^{-1}$, $p \geq p_1$, forms a Toeplitz operator with weak exponential decay with respect to the same manifold $Z$ and we have $[(G_p)^{-1}]_0 = f^{-1}$.
	\end{lem}

	\begin{proof}[Proof of Lemma \ref{lem_ind_step}.b)]
		Remark that by  Lemma \ref{lem_ind_step}.a) and Lemma \ref{lem_inverse_toepl}, for any $k \in \nat$, there is $p_1 \in \nat$, such that the operator $A_{k, p}$ is invertible for $p \geq p_1$.
		The main idea of our proof is to use the following formula
		\begin{equation}\label{eq_resp_ap_lem_ap_inv}
			 \frac{1}{p^{m - \frac{k}{2}}} \ext_{k, p} = \frac{1}{p^{n + \frac{k}{2}}} \cdot (\res_Y \circ \nabla^k B_p^{X, k Y})^{*} \circ \Big( \frac{1}{p^{n -m + k}} A_{k, p} \Big)^{-1},
		\end{equation}
		which follows from (\ref{eq_resp_ap_lem}).
		On the level of Schwartz kernels, the identity (\ref{eq_resp_ap_lem_ap_inv}) means
		\begin{multline}\label{eq_b_f}
			\frac{1}{p^{m - \frac{k}{2}}} \ext_{k, p}(x, y)
			=
			\int_{y_ 1 \in Y}
			 \frac{1}{p^{n + \frac{k}{2}}} \cdot (\res_Y \circ \nabla^k B_p^{X, k Y})^{*}(x, y_1)
			 \cdot
			 \\
			 \cdot
			 \Big( \frac{1}{p^{n -m + k}} A_{k, p} \Big)^{-1}(y_1, y)
			 dv_Y(y_1).
		\end{multline}
		\par 
		From Theorem \ref{thm_mult_def_toepl_k}, Proposition \ref{prop_norm_bnd_distk_expbnd}, Lemma \ref{lem_inverse_toepl} and (\ref{eq_exp_dc_akp11}), we see already that there are $c > 0$, $p_1 \in \nat^*$,  such that for any $r, l \in \nat$, there is $C > 0$, such that for any $p \geq p_1$, $x \in X$, $y \in Y$, the following estimate holds
		\begin{equation}\label{eq_ext_exp_dc_weak}
			\, \big|  \ext_{k, p}(x, y) \big|_{\ccal^r} \leq C p^{m + \frac{r}{2}} \exp \big(- c \sqrt{p} \cdot \dist(x, y) \big).
		\end{equation}
		To improve (\ref{eq_ext_exp_dc_weak}) to (\ref{eq_ext_exp_dc}), we will need to study the Schwartz kernel of $\ext_{k, p}$ more precisely.
		\par 
		Now, let $\epsilon > 0$ be the minimum of the corresponding values from Theorems \ref{thm_berg_off_diag} and \ref{thm_berg_perp_off_diagk} for $k \leq k_0 - 1$.
		We put $\epsilon_0 := \frac{\epsilon}{2}$.
		Let $y_0 \in Y$ and $y_1, y_2 \in B_{y_0}^{Y}(\epsilon_0)$.
		We decompose the integral in (\ref{eq_b_f}) into two parts: the first one over $B_{y_0}^{Y}(\epsilon)$, and the second one is over its complement, which we denote by $Q$.
		Clearly, for $y_3 \in Q$, we get 
		\begin{equation}\label{eq_triangle}
			\dist(y_1, y_3) + \dist(y_3, y_2) \geq \epsilon.
		\end{equation}
		Hence, from Theorems \ref{thm_bk_off_diag}, \ref{thm_ext_exp_dc}.b), \ref{thm_mult_def_toepl_k}, Proposition \ref{prop_exp_bound_int} and (\ref{eq_vol_comp_unif}), we see that the contribution from the integration over $Q$ is smaller than $\exp(- c \sqrt{p}(1 + \dist(y_1, y_2)))$ for some constant $c > 0$.
		Consequently, only the integration over $y_1 \in B_{y_0}^{Y}(\epsilon)$ is non-negligible.
		To evaluate it, we apply Theorem \ref{thm_bk_off_diag}.
		We calculate the integral over the pull-back with respect to the exponential map of our differential forms.
		We use the notations introduced before Theorem \ref{thm_berg_off_diag}.
		After the change of variables $Z \mapsto \sqrt{p} Z$, an estimate, similar to the one which bounded the integral over $Q$, Lemma \ref{lem_bnd_prod_aloc} and the second part of Lemma \ref{lem_comp_poly}, applied for $n := m$, we see that (\ref{eq_ext_as_expk}) holds for
		\begin{equation}\label{eq_jrf_expr2}
			J_{k, r}^{E}(Z, Z'_Y)
			:=
			\sum_{a + b = r}
			\mathcal{K}_{n, m}^{EP} \big[ J_{k, a}^{E, 0}, I_{A_{k, p}^{-1}, b}^{Y} \big],
		\end{equation}
		where $I_{A_{k, p}^{-1}, b}^{Y}$ are the polynomials associated to $( \frac{1}{p^{n - m + k}} A_{k, p} )^{-1}$ as in Theorem \ref{thm_ma_mar_crit_exp_dec}.
		From the corresponding statements about the parity of $J_{k, r}^{E, 0}$, $I_{A_{k, p}^{-1}, r}^{Y}$ from Theorem \ref{thm_ma_mar_crit_exp_dec} and Lemmas \ref{lem_comp_poly}, \ref{lem_ext_as_expk_new}, we deduce that the parity of  $J_{k, r}^{E}$ coincides with $k + r$.
		From Proposition \ref{prop_bound_deg_q} and (\ref{eq_jrf_expr2}), we deduce that
		\begin{equation}\label{eq_deg_bound_jkrr01}
			\deg (J_{k, r}^{E}) 
			\leq
			\max 
			\Big\{ 
			 \deg(J_{k, a}^{E, 0})+
			 \deg(I_{A_{k, p}^{-1}, b}^{Y})
			\Big\},
		\end{equation}
		where the maximum is taken over $a + b = r$.
		From Proposition \ref{prop_bound_deg_q}, Lemma \ref{lem_ext_as_expk_new}, (\ref{eq_defn_jkre0}), (\ref{eq_deg_bound_jkrr01}) and the bound on the degree of $J_{k, a}^{R, 0}$ from Lemma \ref{lem_ext_as_expk_new}, we deduce the needed bound on the degree of $J_{k, r}^{E}$.
		\par 
		\begin{sloppypar}
			Directly from (\ref{eq_pperp_defn_fu2n}), (\ref{eq_res_nm_kernel0001}), (\ref{eq_res_nm_kernel}), (\ref{eq_jk0e0form_ex}) and (\ref{eq_jrf_expr2}), we see that the formula (\ref{eq_je0_expk})  holds for $J_{k, 0}^{E}$.
			Moreover, under assumption (\ref{eq_comp_vol_omeg}), by Theorem \ref{thm_mult_def_toepl_k}, for $i = 0, 1$, we have 
			\begin{equation}
				I_{A_{k, p}^{-1}, i}^{Y} = \frac{1}{(2 \pi)^k \cdot k!} \cdot J_{k, i}^{Y|Y},
			\end{equation}
			where $J_{k, i}^{Y|Y}$ are the polynomials from Theorem \ref{thm_berg_off_diag}, associated to $X, Y := Y$ and $F := {\rm{Sym}}^k (N^{1, 0})^* \otimes \iota^*(F)$.
			From this, (\ref{eq_jo_expl_form}), (\ref{eq_j1_expl_form}), we conclude that
			\begin{equation}
				J_{k, 1}^{E}(Z, Z'_Y)
				=
				\frac{1}{(2 \pi)^k \cdot k!}
				\mathcal{K}_{n, m}^{EP} \big[ J_{k, 1}^{E, 0}, 1 \big].
			\end{equation}
			We deduce (\ref{eq_je1_expk}) from this, Lemma \ref{lem_ext_as_expk_new}, (\ref{eq_kep_formula}) and (\ref{eq_defn_jkre0}).
		\end{sloppypar}
		\par 
		The proof of the fact that $J_{k, r}^{E}$ vanishes at least up to order $k$ along $\real^{2m} \times \real^{2m} \subset \real^{2n} \times \real^{2m}$ proceeds along the same lines as the corresponding result from Lemma \ref{lem_ext_as_expk_new}.
		In total, this finishes the proof of Theorem \ref{thm_ext_as_expk} for $k := k_0$.
		\par 
		From the expansion (\ref{eq_ext_as_expk}), we see that we can improve (\ref{eq_ext_exp_dc_weak}) to (\ref{eq_ext_exp_dc}).
		This finishes completely the proof of Theorem \ref{thm_ext_exp_dc}.a) for $k := k_0$.
	\end{proof}
	\begin{proof}[Proof of Lemma \ref{lem_ind_step}.c)]
		First of all, an easy verification shows the following identity
		\begin{equation}\label{eq_eq_resp_ap_lem_ap_inv}
			\frac{1}{p^n}
			B_p^{\perp, k}
			=
			\frac{1}{p^{m - \frac{k}{2}}}
			\ext_{k, p}
			\circ
			\frac{1}{p^{n - m + \frac{k}{2}}}
			(\res_Y \circ \nabla^k B_p^{X, k Y}).
		\end{equation}
		On the level of Schwartz kernels, the identity (\ref{eq_eq_resp_ap_lem_ap_inv}) basically means that
		\begin{equation}\label{eq_b_f_aa}
			\frac{1}{p^n}
			B_p^{\perp, k}(x_1, x_2)
			=
			\int_{y \in Y}
			 \frac{1}{p^{m - \frac{k}{2}}} \ext_{k, p}(x_1, y)
			 \cdot
			 \frac{1}{p^{n - m + \frac{k}{2}}} (\res_Y \circ \nabla^k B_p^{X, k Y}) (y, x_2)
			 dv_Y(y).
		\end{equation}
		From Theorems \ref{thm_berg_off_diag} and \ref{thm_ext_exp_dc}, for $k \leq k_0 - 1$, Proposition \ref{prop_exp_bound_int} and (\ref{eq_b_f_aa}), we see that Theorem \ref{thm_ext_exp_dc}.b) holds for $k = k_0$.
		\par 
		From (\ref{eq_b_f_aa}), the expansion (\ref{thm_berg_perp_off_diagk}) follows from the same reasoning as we used in (\ref{eq_jrf_expr2}), except that we need to rely on the third part of Lemma \ref{lem_comp_poly} and $J_{k, r}^{\perp}$ are defined as follows
		\begin{equation}\label{eq_defn_jra2233}
			J_{k, r}^{\perp}(Z, Z')
			:=
			\sum_{a + b = r}
			\mathcal{K}_{n, m}^{ER} \big[ J_{k, a}^{E}, J_{k, b}^{R, 0} \big].
		\end{equation}
		From (\ref{eq_ker_form}), (\ref{eq_kmn_poly}), (\ref{eq_je0_expk}), (\ref{eq_je1_expk}), (\ref{eq_jr000_expk}), (\ref{eq_jr000111_expk}) and (\ref{eq_defn_jra2233}), we deduce (\ref{eq_jopep_0}) and (\ref{eq_jopep_1}).
		The parity statement about the polynomials $J_{k, r}^{\perp}$ follows from Lemma \ref{lem_comp_poly} and the corresponding statements for $J_{k, a}^{E}$ and $J_{k, b}^{R, 0}$, proved in Theorem \ref{thm_ext_as_expk} and Lemma \ref{lem_ext_as_expk_new}.
		The vanishing up to order $k$ of the polynomials $J_{k, r}^{\perp}$ follows from (\ref{eq_ker_form}), (\ref{eq_defn_jra2233}) and the corresponding statements for $J_{k, a}^{E}$ and $J_{k, b}^{R, 0}$, proved in Theorem \ref{thm_ext_as_expk} and Lemma \ref{lem_ext_as_expk_new}.
		From Proposition \ref{prop_bound_deg_q} and (\ref{eq_defn_jra2233}), we deduce that
		\begin{equation}\label{eq_deg_bound_jkrr02}
			\deg (J_{k, r}^{\perp}) 
			\leq
			\max 
			\Big\{ 
			 \deg(J_{k, a}^{E}) +
			 \deg(J_{k, b}^{R, 0})
			\Big\},
		\end{equation}
		where the maximum is taken over $a + b = r$.
		From Lemma \ref{lem_ext_as_expk_new}, (\ref{eq_deg_bound_jkrr02}) and the bound on the degree of $J_{k, a}^{R, 0}$ (resp. $J_{k, a}^{E}$)  from Lemma \ref{lem_ext_as_expk_new} (resp. Theorem \ref{thm_ext_as_expk}), we deduce the needed bound on the degree of $J_{k, r}^{\perp}$.
	\end{proof}
	\begin{proof}[Proof of Theorems \ref{thm_ext_exp_dc}, \ref{thm_ext_as_expk}, \ref{thm_berg_perp_off_diagk}, \ref{thm_mult_def_toepl_k}]
	It follows directly from Lemma \ref{lem_ind_step} by induction and the fact that for $k < 0$, the statements of Theorems \ref{thm_ext_exp_dc}.b) and \ref{thm_berg_perp_off_diagk} are void.
	\end{proof}
	\begin{proof}[Proof of Theorem \ref{thm_isom} ]
		First of all, let us establish that under the assumption (\ref{eq_comp_vol_omeg}), there is $C > 0$, such that as $p \to \infty$, we have
		\begin{equation}\label{eq_ep_norm}
		\begin{aligned}
			&
			\Big|
				\big\|
				\ext_{k, p}
				\big\|
				-
				\frac{1 }{p^{\frac{n - m + k}{2}}} \cdot \frac{1}{\sqrt{k! \cdot (2 \pi)^k}}
			\Big|
			\leq
			\frac{C}{p^{\frac{n - m + k + 2}{2}}},
			\\
			& 
			\Big|
				\big\|
				\res_{k, p}
				\big\|
				-
				p^{\frac{n - m + k}{2}} \cdot \sqrt{k! \cdot (2 \pi)^k}
			\Big|
			\leq
			C p^{\frac{n - m + k - 2}{2}}.
		\end{aligned}
		\end{equation}
		\par 
		Indeed, remark that from Theorem \ref{thm_mult_def_toepl_k} and Proposition \ref{prop_norm_bnd_distk_expbnd} that for any $k \in \nat$, there are $p_1 \in \nat$, $C > 0$, such that for any $p \geq p_1$, we have
		\begin{equation}
			\Big\|
				\frac{1}{p^{n - m + k}} A_{k, p} - k! \cdot (2 \pi)^k \cdot B_{k, p}^Y
			\Big\|
			\leq
			\frac{C}{p}.
		\end{equation}
		In particular, we conclude that 
		\begin{equation}\label{eq_akp_norm}
			\Big|
				\big\|
				A_{k, p} 
				\big\|
				-
				p^{n - m + k} \cdot k! \cdot (2 \pi)^k
			\Big|
			\leq
			C p^{n - m + k - 1}.
		\end{equation}
		Using Lemma \ref{lem_inverse_toepl} and similar arguments, we see that
		\begin{equation}\label{eq_akpmin1_norm}
			\Big|
				\big\|
				A_{k, p}^{-1}
				\big\|
				-
				\frac{1}{p^{n - m + k}} \cdot \frac{1}{k! \cdot (2 \pi)^k}
			\Big|
			\leq
			\frac{C}{p^{n - m + k + 1}}.
		\end{equation}
		\par 
		From (\ref{eq_resp_ap_lem}) and the fact that $\res_{k, p} \circ \ext_{k, p} = B_{k, p}^Y$, we have the following identities
		\begin{equation}\label{eq_comp_ext_oper}
			(\ext_{k, p})^* \circ \ext_{k, p} = \big( (A_{k, p})^* \big)^{-1},
			\qquad
			\res_{k, p} \circ (\res_{k, p})^{*} = A_{k, p}.
		\end{equation}
		Clearly, we have $ \| (\ext_{k, p})^* \circ \ext_{k, p} \| = \| \ext_{k, p} \|^2$ and $ \| \res_{k, p} \circ (\res_{k, p})^{*} \| = \|  \res_{k, p} \|^2$.
		The bounds (\ref{eq_ep_norm}) now follows from this observation, (\ref{eq_akp_norm}), (\ref{eq_akpmin1_norm}) and (\ref{eq_comp_ext_oper}).
		\par 
		Now, from (\ref{eq_ep_norm}) and the fact that $\res_{k, p} \circ \ext_{k, p} = B_{k, p}^Y$, we conclude that there are $p_1 \in \nat$, $C > 0$, such that for any $k \in \nat$, $p \geq p_1$, $g \in H^0_{(2)}(Y, {\rm{Sym}}^k (N^{1, 0})^* \otimes \iota^*( L^p \otimes F))$, $g \neq 0$, we have
		\begin{equation}\label{eq_ext_op_as_isom}
			\bigg|
			\frac{\| \ext_{k, p} (g) \|_{L^2(X, L^p \otimes F)}}{\| g \|_{k, L^2(Y, L^p \otimes F)}}
			-
			\frac{1 }{p^{\frac{n - m + k}{2}}} \cdot \frac{1}{\sqrt{k! \cdot (2 \pi)^k}}
			\bigg|
			\leq
			\frac{C}{p^{\frac{n - m + k + 2}{2}}}.
		\end{equation}		 
		\par 
		Now, let us fix $f \in H^0_{(2)}(X, L^p \otimes F)$, and denote by $[f]$ the element it represents in the quotient space $H^0_{(2)}(X, L^p \otimes F) / H^0_{(2)}(X, L^p \otimes F \otimes \mathcal{J}_Y^{k + 1})$.
		Directly from the definition, we have
		\begin{equation}\label{eq_f_norm_quot}
			\big\| [f] \big\|_{L^2(X, L^p \otimes F)}^2
			=
			\sum_{l = 0}^{k} \| B_{l, p}^{\perp} f \|_{L^2(X, L^p \otimes F)}^2.
		\end{equation}
		We let ${\rm{Jet}}_{k, p}(f) = (g_0, \ldots, g_k)$, where $g_i \in H^0_{(2)}(Y, {\rm{Sym}}^i (N^{1, 0})^* \otimes \iota^*( L^p \otimes F))$.
		From the definition of the map ${\rm{Jet}}_{k, p}$ from (\ref{eq_defn_jetmap}) and the characterization of $g_i$ from (\ref{eq_jet_char}), we see that 
		\begin{equation}\label{eq_gi_defn}
			g_i = \res_{i, p} (B_{i, p}^{\perp} f).
		\end{equation}
		We conclude from (\ref{eq_ep_norm}) and (\ref{eq_gi_defn}) that
		\begin{equation}\label{eq_gi_norm111}
			\big\| g_i \big\|_{k, L^2(Y, L^p \otimes F)} = 
			\Big( 
				p^{\frac{n - m + k}{2}} \cdot \sqrt{k! \cdot (2 \pi)^k} 
				+
				O(p^{\frac{n - m + k - 2}{2}})
			\Big)
			\cdot \big\| B_{i, p}^{\perp} f \big\|_{L^2(X, L^p \otimes F)}.
		\end{equation}
		The result now follows from (\ref{eq_ext_op_as_isom}), (\ref{eq_f_norm_quot}), (\ref{eq_gi_norm111}) and the definition of the scalar product $\scal{\cdot}{\cdot}_{{\rm{Jet}}_{k, p}}$ from (\ref{eq_scal_prod_jet}).
	\end{proof}

\subsection{Higher order peak sections as a special case of extension theorem}\label{sect_peak_ho}
	The main goal of this section is to illustrate Theorems \ref{thm_ext_exp_dc}, \ref{thm_ext_as_expk} on a simple example when the submanifold corresponds to a fixed point.
	\par 
	Let us denote the extension operator in this case by $\ext_{k, p}^{\{ x \}}$. 
	It corresponds to
	\begin{equation}
		\ext_{k, p}^{\{ x \}} : {\rm{Sym}}^k T^{1, 0}X_x^* \otimes (L^p \otimes F)_x \to H^0_{(2)}(X, L^p \otimes F \otimes \mathcal{J}_x^k).
	\end{equation}
	Clearly, for any $v \in {\rm{Sym}}^k (N_{x}^{1, 0})^* \otimes (L^p \otimes F)_x$, the section $s_{k, p}^{x, v} := \ext_{k, p}^{\{ x \}}(v)$ minimizes the $L^2$-norm among all sections from $H^0_{(2)}(X, L^p \otimes F \otimes \mathcal{J}_x^k)$, having $k$-th jet equal to $v$.
	\par 
	The sections $s_{k, p}^{x, v}$ were defined in complex geometry by Tian \cite{TianBerg}.
	For $k = 0$, they bear the name “peak sections". 
	Due to this reason, for $k \geq 1$, we call $s_{k, p}^{x, v}$ the \textit{higher order peak sections}.
	\par 
	\begin{thm}\label{thm_peak_sect_ho}
		There are $c, C > 0$, such that for any $x, y \in X$, $p \in \nat$, $p \geq p_1$, $v \in {\rm{Sym}}^k (N_{x}^{1, 0})^* \otimes (L^p \otimes F)_x$, we have
		\begin{equation}
			\big| s_{k, p}^{x, v}(y) \big|_{\ccal^r}
			\leq
			C p^{\frac{r - k}{2} } \exp(- c \sqrt{p} \dist(x, y)).
		\end{equation}
		Moreover, there are $\epsilon > 0$, $c, C > 0$, such that for any $x \in X$, for any $Z \in \real^{2n}$, $|Z| < \epsilon$, $p \in \nat$, $p \geq p_1$, $v \in {\rm{Sym}}^k (N_{x}^{1, 0})^* \otimes (L^p \otimes F)_x$, we have
		\begin{equation}
			\Big| s_{k, p}^{x, v}(\phi_x^X(Z)) - (v \cdot Z^{\otimes k}) \cdot \exp \big( - \frac{\pi}{2} p |Z|^2 \big) \Big|_{\ccal^r}
			\leq
			C \cdot |v| \cdot p^{\frac{r}{2} - 1 } \cdot |Z|^k \cdot \exp(- c \sqrt{p} |Z|),
		\end{equation}
		where the $\ccal^r$-norm is taken with respect to $x, Z$.
	\end{thm}
	\begin{rem}
		For $r \leq 2$, similar results were obtained by Tian \cite[Lemma 1.2]{TianBerg}. 
		For $k = 0$, this statement follows directly from Dai-Liu-Ma \cite{DaiLiuMa}.
	\end{rem}
	\begin{proof}
		The first statement follows directly by applying Theorem \ref{thm_ext_exp_dc}.a) for $Y := \{ y_0 \}$.
		Remark that for $Y$ as above, the second fundamental form vanishes, so the second term of the asymptotic expansion from Theorem \ref{thm_ext_as_expk} vanishes according to (\ref{eq_je1_expk}).
		The second statement of Theorem \ref{thm_peak_sect_ho} then follows from this remark, the application of Theorem \ref{thm_ext_as_expk} for $Y := \{ y_0 \}$ and $r = 1$ and the use of Remark \ref{rem_bpper_refin}.d).
	\end{proof}

\bibliography{bibliography}

\begin{thebibliography}{10}

\bibitem{BisDem}
J.-M. Bismut.
\newblock Demailly's asymptotic morse inequalities: A heat equation proof.
\newblock {\em J. Funct. Anal.}, 72(2):263 -- 278, 1987.

\bibitem{BisLeb91}
J.-M. Bismut and G.~Lebeau.
\newblock {Complex immersions and Quillen metrics}.
\newblock {\em Publ. Math. IHES}, 74(1):1--291, 1991.

\bibitem{BVas}
J.-M. Bismut and E.~Vasserot.
\newblock {The asymptotics of the Ray-Singer analytic torsion associated with
  high powers of a positive line bundle}.
\newblock {\em Comm. Math. Phys.}, 125(2):355--367, 1989.

\bibitem{BordMeinSchl}
M.~{Bordemann}, E.~{Meinrenken}, and M.~{Schlichenmaier}.
\newblock {Toeplitz quantization of K\"ahler manifolds and \(gl(N)\), \(N\to
  \infty\) limits}.
\newblock {\em {Commun. Math. Phys.}}, 165(2):281--296, 1994.

\bibitem{Bouche}
T.~Bouche.
\newblock {Sur les in{\'e}galit{\'e}s de Morse holomorphes lorsque la courbure
  du fibr{\'e} en droites est d{\'e}g{\'e}n{\'e}r{\'e}e}.
\newblock {\em {Ann. Sc. Norm. Super. Pisa}}, 18(4):501--523, 1991.

\bibitem{CaoDemMatsumura}
J.~Cao, J.-P. Demailly, and S.-I. Matsumura.
\newblock A general extension theorem for cohomology classes on non reduced
  analytic subspaces.
\newblock {\em Sci. China, Math.}, 60(6):949--962, 2017.

\bibitem{CaoPaunOT}
J.~Cao and M.~P{\u a}un.
\newblock {On extension of pluricanonical forms defined on the central fiber of
  a K{\"a}hler family, ArXiv: 2012.05063}.
\newblock 2020.

\bibitem{Caltin}
D.~{Catlin}.
\newblock {The Bergman kernel and a theorem of Tian}.
\newblock In {\em {Proceedings of the 40th Taniguchi symposium, Katata, Japan,
  June 23--28, 1997}}, pages 1--23. Boston, MA: Birkh\"auser, 1999.

\bibitem{ComMarPartBerg}
D.~Coman and G.~Marinescu.
\newblock {On the first order asymptotics of partial Bergman kernels}.
\newblock {\em {Ann. Fac. Sci. Toulouse, Math. (6)}}, 26(5):1193--1210, 2017.

\bibitem{DaiLiuMa}
X.~Dai, K.~Liu, and X.~Ma.
\newblock On the asymptotic expansion of bergman kernel.
\newblock {\em J. Diff. Geom.}, 72(1):1--41, 2006.

\bibitem{Dem82}
J.-P. {Demailly}.
\newblock {Estimations $L^2$ pour l'op\'erateur $\overline{\partial}$ d'un
  fibre vectoriel holomorphe semi-positif au-dessus d'une vari\'et\'e
  Kaehlerienne complete}.
\newblock {\em {Ann. Sci. \'Ec. Norm. Sup\'er. (4)}}, 15:457--511, 1982.

\bibitem{DemBookAnMet}
J.-P. {Demailly}.
\newblock {\em {Analytic methods in algebraic geometry}}, volume~1.
\newblock {Somerville, MA: International Press; Beijing: Higher Education
  Press}, 2012.

\bibitem{DemCompl}
J.-P. Demailly.
\newblock {\em {Complex Analytic and Differential Geometry}}.
\newblock 2012.

\bibitem{DemExtRed}
J.-P. {Demailly}.
\newblock {Extension of holomorphic functions defined on non reduced analytic
  subvarieties}.
\newblock In {\em The legacy of Bernhard Riemann after one hundred and fifty
  years. Volume I}, pages 191--222. Somerville, MA: International Press;
  Beijing: Higher Education Press, 2016.

\bibitem{EichBoundG}
J.~{Eichhorn}.
\newblock {The boundedness of connection coefficients and their derivatives}.
\newblock {\em {Math. Nachr.}}, 152:145--158, 1991.

\bibitem{FinOTAs}
S.~Finski.
\newblock {Semiclassical Ohsawa-Takegoshi extension theorem and asymptotic of
  the orthogonal Bergman kernel, arXiv: 2109.06851.}
\newblock 2021.

\bibitem{FinToeplImm}
S.~Finski.
\newblock {Complex embeddings, Toeplitz operators and transitivity of optimal
  holomorphic extensions, arXiv:2201.04102}.
\newblock 2022.

\bibitem{GrosSchnBound}
N.~{Gro{\ss}e} and C.~{Schneider}.
\newblock {Sobolev spaces on Riemannian manifolds with bounded geometry:
  general coordinates and traces}.
\newblock {\em {Math. Nachr.}}, 286(16):1586--1613, 2013.

\bibitem{HeintzKarch}
E.~{Heintze} and H.~{Karcher}.
\newblock {A general comparison theorem with applications to volume estimates
  for submanifolds}.
\newblock {\em {Ann. Sci. \'Ec. Norm. Sup\'er. (4)}}, 11(4):451--470, 1978.

\bibitem{HosonoJet}
G.~Hosono.
\newblock The optimal jet {{\(L^2\)}} extension of {Ohsawa}-{Takegoshi} type.
\newblock {\em Nagoya Math. J.}, 239:153--172, 2020.

\bibitem{LuBergman}
Z.~Lu.
\newblock {On the Lower Order Terms of the Asymptotic Expansion of
  Tian-Yau-Zelditch}.
\newblock {\em Amer. J. Math.}, 122(2):235--273, 2000.

\bibitem{MaHol}
X.~Ma and G.~Marinescu.
\newblock {\em Holomorphic Morse inequalities and Bergman kernels}, volume 254
  of {\em Progr. Math.}
\newblock Birkh{\"a}user Verlag Basel, 2007.

\bibitem{MaMar08a}
X.~Ma and G.~Marinescu.
\newblock Generalized {B}ergman kernels on symplectic manifolds.
\newblock {\em Adv. in Math.}, 217(4):1756--1815, 2008.

\bibitem{MaMarToepl}
X.~{Ma} and G.~{Marinescu}.
\newblock {Toeplitz operators on symplectic manifolds.}
\newblock {\em {J. Geom. Anal.}}, 18(2):565--611, 2008.

\bibitem{MaMarOffDiag}
X.~{Ma} and G.~{Marinescu}.
\newblock {Exponential estimate for the asymptotics of Bergman kernels}.
\newblock {\em {Math. Ann.}}, 362(3-4):1327--1347, 2015.

\bibitem{MaZhBKSR}
X.~{Ma} and W.~{Zhang}.
\newblock {\em {Bergman Kernels and symplectic reduction}}.
\newblock Ast{\'e}risque 318, 2008.

\bibitem{ManExt}
L.~Manivel.
\newblock A theorem of {{\(L^ 2\)}} extension of holomorphic sections of a
  {Hermitian} bundle.
\newblock {\em Math. Z.}, 212(1):107--122, 1993.

\bibitem{McNealVarolin}
J.~McNeal and D.~Varolin.
\newblock {Extension of Jets With $L^2$ Estimates, and an Application, ArXiv:
  1707.04483}.
\newblock 2017.

\bibitem{Ohsawa}
T.~{Ohsawa}.
\newblock {On the extension of $L^2$ holomorphic functions. II}.
\newblock {\em {Publ. Res. Inst. Math. Sci.}}, 24(2):265--275, 1988.

\bibitem{OhsTak1}
T.~{Ohsawa} and K.~{Takegoshi}.
\newblock {On the extension of \(L^ 2\) holomorphic functions}.
\newblock {\em {Math. Z.}}, 195:197--204, 1987.

\bibitem{PokSing}
F.~T. Pokorny and M.~Singer.
\newblock Toric partial density functions and stability of toric varieties.
\newblock {\em Math. Ann.}, 358(3-4):879--923, 2014.

\bibitem{PopovNonredExt}
D.~Popovici.
\newblock {\(L^2\) extension for jets of holomorphic sections of a Hermitian
  line bundle}.
\newblock {\em {Nagoya Math. J.}}, 180:1--34, 2005.

\bibitem{RandriamTh}
H.~Randriambololona.
\newblock Hauteurs pour les sous-sch{\'e}mas et exemples d'utilisation de
  m{\'e}thodes arakeloviennes en th{\'e}orie de l'approximation diophantienne
  (th{\`e}se).
\newblock 2002.

\bibitem{RaoZhang}
S.~Rao and R.~Zhang.
\newblock {{\(L^2\)}} extension theorem for jets with variable denominators.
\newblock {\em C. R., Math., Acad. Sci. Paris}, 359(2):181--193, 2021.

\bibitem{RossSinger}
J.~Ross and M.~Singer.
\newblock {Asymptotics of partial density functions for divisors}.
\newblock {\em {J. Geom. Anal.}}, 27(3):1803--1854, 2017.

\bibitem{RossThomasObstr}
J.~Ross and R.~Thomas.
\newblock An obstruction to the existence of constant scalar curvature
  {K{\"a}hler} metrics.
\newblock {\em J. Differ. Geom.}, 72(3):429--466, 2006.

\bibitem{SchBound}
T.~{Schick}.
\newblock {Manifolds with boundary and of bounded geometry}.
\newblock {\em {Math. Nachr.}}, 223:103--120, 2001.

\bibitem{TianBerg}
G.~Tian.
\newblock {On a set of polarized K\"ahler metrics on algebraic manifolds}.
\newblock {\em J. Diff. Geom.}, 32(1):99--130, 1990.

\bibitem{WangBergmKern}
X.~{Wang}.
\newblock {Canonical metrics on stable vector bundles}.
\newblock {\em {Commun. Anal. Geom.}}, 13(2):253--285, 2005.

\bibitem{ZeldBerg}
S.~Zelditch.
\newblock {Szeg{\"o} kernels and a theorem of Tian}.
\newblock {\em Internat. Math. Res. Notices}, 1998(6):317--331, 1998.

\bibitem{ZeldZhouInter}
S.~Zelditch and P.~Zhou.
\newblock Interface asymptotics of partial {Bergman} kernels on
  {{\(S^1\)}}-symmetric {K{\"a}hler} manifolds.
\newblock {\em J. Symplectic Geom.}, 17(3):793--856, 2019.

\bibitem{ZhuFock}
K.~Zhu.
\newblock {\em Analysis on {Fock} spaces}, volume 263 of {\em Grad. Texts
  Math.}
\newblock New York, NY: Springer, 2012.

\end{thebibliography}

		\bibliographystyle{abbrv}

\Addresses

\end{document}